\documentclass[reqno,11pt,a4paper]{amsart}
\usepackage[utf8]{inputenc}
\usepackage{pifont}
\usepackage[T1]{fontenc}
\usepackage{enumerate}
\usepackage{textcomp}
\usepackage{verbatim}
\usepackage{pgffor}

\usepackage{esint}
\usepackage{pdfsync}
\usepackage[colorlinks=true,linkcolor=blue]{hyperref}%
\usepackage[T1]{fontenc}
\usepackage{pifont}
\usepackage[english]{babel}

\usepackage{ulem}
\usepackage{amsmath,esint,color}
\usepackage{amsthm,latexsym,epsfig,graphicx,marvosym, mathrsfs,bigints,stmaryrd,textgreek,upgreek}
\usepackage{amsmath,stmaryrd}
\usepackage{relsize}
\usepackage{amsfonts}
\usepackage{amssymb}
\usepackage[headings]{fullpage}

\allowdisplaybreaks

\usepackage{amsmath}
\usepackage{amsfonts}
\usepackage{amssymb}
\usepackage{amsthm}

\newcommand{\NN}{\mathbb{N}}
\newcommand{\TT}{\mathbb{T}}

\newcommand{\RR}{\mathbb{R}}

\newcommand{\bb}{\ensuremath{\mathcal{B}}}
\newcommand{\BB}{\ensuremath{\widetilde{\mathcal{B}}}}
\newtheorem{theorem}{Theorem}[section]
\newtheorem{proposition}{Proposition}[section]
\newtheorem{lemma}{Lemma}[section]
\newtheorem{remark}{Remark}[section]
\newtheorem{coro}{Corollary}[section]
\newtheorem{defn}{Definition}[section]

\newcommand{\dd}{\mathrm{d}}

\begin{document}

\title[Temperature patches for 2D non-diffusive Boussinesq system]
{Global regularity and infinite Prandtl number limit of temperature patches for the 2D Boussinesq system}
\author{Omar Lazar}
\address{College of Engineering and Technology, American University of the Middle East, Kuwait}
\email{omar.lazar@aum.edu.kw}
\author{Liutang Xue}
\address{School of Mathematical Sciences, Laboratory of Mathematics and Complex Systems (MOE), Beijing Normal University, Beijing 100875, P.R. China}
\email{xuelt@bnu.edu.cn}
\author{Jiakun Yang}
\address{School of Mathematical Sciences, Laboratory of Mathematics and Complex Systems (MOE), Beijing Normal University, Beijing 100875, P.R. China}
\email{yjk20000123@126.com}

\keywords{Boussinesq-Navier-Stokes system, Stokes-transport system, temperature patches, infinite Prandtl number limit.}
\subjclass[2010]{Primary 76D03, 35Q35, 35Q86}
\date{\today}

\maketitle
\begin{abstract}
  We prove global regularity and study the infinite Prandtl number limit of temperature patches for the 2D non-diffusive
  Boussinesq system with dissipation in the full subcritical regime.
  The temperature satisfies a transport equation and the temperature initial data
  are given in the form of non-constant patches.
  Our first main result is a persistence of regularity of the patches globally in time.
  More precisely, we prove that if the boundary of the initial temperature patch lies in $C^{k+\gamma}$ with $k\geq 1$
  and $\gamma\in(0,1)$ then this initial regularity is preserved for all  time. Importantly, our proof is robust enough to show uniform dependence on the Prandtl number in some cases.
  This result solves a question in Khor and Xu \cite{KX22} concerning the global control of the curvature of
  the patch boundary.  Besides, by studying the limit when the Prandtl number goes to infinity,
  we find that the patch solutions to the 2D Boussinesq-Navier-Stokes system in the torus   converge to the unique patch solutions
  of the (fractional) Stokes-transport equation and that the $C^{k+\gamma}$ regularity of the patch boundary is
  globally preserved. This allows us to extend the $C^{k+\gamma}$ persistence result of  Grayer II \cite{Gray23}
  from the range $k\in \{0,1,2\}$ to the full range $k\geq 1$.
\end{abstract}

\tableofcontents

\section{Introduction}
In this paper, we study the Cauchy problem of the two-dimensional  non-diffusive Boussinesq-Navier-Stokes system:
\begin{equation}\label{BoussEq}
  (\mathrm{B}_\alpha)\left\{\begin{aligned}
  \partial_t \theta + u \cdot \nabla \theta & =0, \qquad \qquad\; (t, x) \in \mathbb{R}_{+} \times \mathbf{D}, \\
  \tfrac{1}{\mathrm{Pr}} \big(\partial_t u + u \cdot \nabla u \big)
  + \nu \Lambda^{2 \alpha} u
  & = - \nabla p+\theta e_2, \\
  \nabla \cdot u & =0, \\
  \left.(u,\theta)\right|_{t=0} & = (u_0,\theta_0),
  \end{aligned}\right.
\end{equation}
where $\mathbf{D}$ is either $\mathbb{R}^2$ or torus $\mathbb{T}^2$, $e_2:= (0,1)^T$ is the second canonical vector of $ \mathbb{R}^2$, $\nu\geq0$ is the kinematic viscosity,
$\mathrm{Pr}>0$ is the non-dimensional Prandtl number,
the dissipation operator $\Lambda^{2\alpha}:=(-\Delta)^\alpha$ ($\alpha\in (0,1]$) is the classical (fractional) Laplacian operator defined {\it{via}} the Fourier transform through the formula
$\widehat{\Lambda^{2 \alpha} f} (\cdot)= |\cdot|^{2 \alpha} \widehat{f}(\cdot)$.
The vector field $u=\left(u_1( x,t), u_2( x,t)\right)^T$ is the velocity of the fluid,
and the scalars $p=p( x,t) $ and $\theta=\theta( x,t) $  denote the pressure and temperature of the fluid, respectively.
This system $(\mathrm{B}_\alpha)$ with $\alpha=1$ is used to model the natural convection phenomena in the ocean and atmospheric dynamics \cite{AJM03,JP1987}. It is also an important mathematical model used to study the Rayleigh-B\'enard convective motion as noticed for example in the work by Constantin and Doering \cite{CD1999}. The system $(\mathrm{B}_\alpha)$ with $\nu>0$, $\alpha\in(0,1)$ can be viewed as an intermediate model connecting the inviscid case (i.e. $\nu=0$) and the full Laplacian case that is $\nu>0$, $\alpha=1$. We refer to \cite{FLB74,MGSIG} for some physical background on the fractional Navier-Stokes equations which corresponds to the case $\theta\equiv 0$.

From the mathematical point of view, the Boussinesq system \eqref{BoussEq} contains the incompressible Navier-Stokes
and Euler equations as special cases \cite{LEM16,MB02}. Furthermore, in the inviscid case, {\it i.e.} $\nu=0$, the Boussinesq system shares many similarities with the 3D axi-symmetric  Euler equations with swirl.
In light of the maximum principle of $\theta$  and the maximal regularity estimates of fractional parabolic equations, we can distinguish 3 cases in the the viscous case (i.e. $\nu>0$). Namely, the cases $\alpha>\frac{1}{2}$, $\alpha=\frac{1}{2}$ and $0<\alpha<\frac{1}{2}$ which are classically called  \textit{subcritical}, \textit{critical} and \textit{supercritical}  respectively.
It is worth recalling that, so far, the global well-posedness  of smooth solutions for the 2D Boussinesq system \eqref{BoussEq} remains a challenging open problem in the inviscid case or in the supercritical case.
Some important recent advances in the study of the 2D inviscid Boussinesq system \eqref{BoussEq} have been obtained in \cite{ElgJ20,CheH21,CheH23} where  finite time blow-up results in various domains are proved. We  refer also to \cite{CCL19,ElgW16} for the global stability results.

For the 2D Boussinesq-Navier-Stokes system $(\mathrm{B}_\alpha)$ with $\alpha=1$ and $\mathrm{Pr}=1$,
Chae \cite{DC06} and Hou, Li \cite{HL05}
independently proved the global existence and uniqueness result associated with
the smooth initial data $(u_0,\theta_0)\in H^s\times H^s$ with $s>2$,
which gives an answer to the problem number 3 in Moffatt \cite{Moff01} by ruling out the possible development of
singularity in the gradient for this system.
The same type of results have been obtained by Abidi and Hmidi \cite{AH07} who proved the global well-posedness result for less regular initial data
$\theta_0\in B^0_{2,1}$, $u_0\in L^2\cap B^{-1}_{\infty,1}$. Then,
Hmidi and Keraani \cite{HK2007} proved global existence of weak solutions to the Boussinesq-Navier-Stokes system $(\mathrm{B}_1)$
with initial data $\theta_0\in L^2$, $u_0\in H^s$, $s\in [0,2)$. The uniqueness  of weak solutions obtained in \cite{HK2007}  has been  solved by Danchin and Pai\"cu \cite{DP08} using  new regularizing effect  together with paradifferential calculus. Let us also mention the work of Hu, Kukavica and Ziane \cite{HKZ15} who proved global persistence of regularity  in Sobolev spaces.

For the 2D Boussinesq-Navier-Stokes system $(\mathrm{B}_\alpha)$ with fractional dissipation and $\mathrm{Pr}=1$,
Hmidi, Keraani and Rousset \cite{HKR10} considered the critical viscous case $\alpha=\frac{1}{2}$
and proved that  $(\mathrm{B}_{\frac{1}{2}})$ is globally well-posed for any data
$\theta_0\in L^2\cap B^0_{\infty,1}$ and $u_0\in H^1\cap\dot{W}^{1,p}$, $p>2$.

Recently, there has also been significant attention on the \textit{Boussinesq temperature patch problem}
for the non-diffusive Boussinesq system $(\mathrm{B}_\alpha)$, which is
a free boundary problem of the system  $(\mathrm{B}_\alpha)$
associated with an initial data which is given as the characteristic function of an initial domain $D_0\subset \mathbf{D}$
which is assumed to be simply connected and bounded.
Since $\theta$ satisfies a transport equation and since the velocity $u$ is assumed to be regular enough,
it  implies (at least formally) that the temperature patch structure is  preserved.
In other words, any  initial data $\mathbf{1}_{D_0}$  gives rise to a solution $\theta(x, t)=\mathbf{1}_{D(t)}$
where $D(t)=X_t\left(D_0\right)$ and $X_t(\cdot)$ is the particle trajectory generated
by the velocity $u$ which verifies
\begin{equation}\label{eq:X}
  \frac{\partial X_t(y)}{\partial t}=u\left(X_t(y), t\right),\left.\quad X_t(y)\right|_{t=0}=y.
\end{equation}
This point of view allows us to study many regularity questions, in particular, one may wonder whether the initial regularity of the patch boundary is globally preserved along the evolution. More precisely, one may study and try to answer the following question:
\begin{align*}
  \textrm{suppose $\partial D_0 \in C^{k+\gamma}$, $k \in \mathbb{Z}^{+}, \gamma \in(0,1)$,
  whether $\partial D(t) \in C^{k+\gamma}$ for all time?}
\end{align*}
Here the notation $\partial D(t) \in C^{k+\gamma}$ means  that there is a parametrization of the patch boundary
$\partial D(t)=\left\{z(\alpha, t) \in \mathbf{D}, \alpha \in \mathbb{S}^1=[0,1]\right\}$
with $z(\cdot, t) \in C^{k+\gamma}$.

Such regularity problem were initiated in the 1980s, in particular with the study of the vorticity patch problem for the 2D Euler equations. This problem was solved by Chemin \cite{Chem93} (using paradifferential calculus together with  striated regularity estimates) and Bertozzi and Constantin \cite{bertozzi1993} (using a more geometric approach based on cancellation of singular kernel). They were able to prove that if an initial patch  of the 2D Euler equations has its boundary in $C^{k+\gamma}$ then this regularity remains forever.

As for the temperature patch problem for the Boussinesq system $(\mathrm{B}_\alpha)$
with $\alpha=1$ and $\mathrm{Pr}=1$, Danchin and Zhang \cite{DZ17}  proved that the $C^{1+\gamma}$ regularity of the patch boundary is globally preserved in the $2 \mathrm{D}$ case as well as in the
$3 \mathrm{D}$ case under an additional smallness condition by using the striated estimates. Using another approach,  Gancedo and Garc\'ia-Juárez \cite{GGJ17}  gave a proof of 
the persistence of the $C^{1+\gamma}$ regularity in the 2D case. Moreover, they proved that the
$W^{2, \infty}$ and $C^{2+\gamma}$ regularity of the boundary of the temperature patch  is globally preserved.
Their approach are based on new cancellations in time-dependent Calder\'on-Zygmund operators.
In particular, the result of \cite{GGJ17} implies that the curvature of the patch boundary is uniformly-in-time bounded.
Furthermore, Chae, Miao and Xue \cite{CMX22} established the global
$C^{k+\gamma}$ (for all $k\in\mathbb{Z}^+$)-regularity
persistence of the patch boundary. The same type of results were obtained in the 3D case in \cite{GGJ20,LLX24}.

For the 2D Boussinesq-Navier-Stokes system $(\mathrm{B}_\alpha)$ with
$\frac{1}{2}<\alpha<1$, $\mathrm{Pr}=1$, Khor and Xu \cite{KX22} were able to prove that, given  an initial  temperature patch data
$\theta_0 = \mathbf{1}_{D_0}$ whose boundary of in $C^{2\alpha}$ then this regularity is preserved for all time. Besides, the authors in \cite{KX22} raised two  interesting questions, namely,
\begin{enumerate}[(a)]
  \item \textit{Is it possible to control the curvature for $\alpha\in (1/2,1)$, as it is possible in the case $\alpha=1$?}
  \item \textit{Can the critical equation $\alpha=1/2$ support  unique temperature patch solutions, and what regularity of their boundary would be preserved?}
\end{enumerate}
As we shall see later, we give an affirmative answer to the question (a) in Theorem \ref{thm:main},
and we make some comments on the question (b) in Remark \ref{rem:crit}.

When the Prandtl number $\mathrm{Pr}$ tends to infinity, as observed by Grayer II \cite{Gray23},
the material derivative of $u$ in the momentum equation in the system \eqref{BoussEq} vanishes (at least formally) and
the system \eqref{BoussEq} becomes the two-dimensional (fractional) Stokes-transport system (by assuming $\nu=1$):
\begin{equation}\label{eq:StokTrans}
  (\mathrm{ST}_\alpha)\left\{\begin{aligned}
  \partial_t \theta + u \cdot \nabla \theta & =0, \qquad \qquad \quad (t, x) \in \mathbb{R}_{+} \times \mathbf{D},
  \\
  \Lambda^{2 \alpha} u  & = - \nabla p+\theta e_2, \\
  \nabla \cdot u & =0, \\
  \left.\theta \right|_{t=0} & = \theta_0,
  \end{aligned}\right.
\end{equation}
where 
$\alpha \in (0,1]$.
The Stokes-transport system, which is the system $(\mathrm{ST}_\alpha)$ with $\alpha =1$,
can also be recovered by taking a limit of sedimentation of inertialess rigid particles in a viscous fluid satisfying Stokes system
\cite{Hofer18}, where $\theta$ stands for the probability density function of the particles and $(u,p)$ are the velocity and pressure of the fluid. The global existence and uniqueness issue for the 3D Stokes-transport system associated with regular or rough initial data
has been intensely studied in various settings \cite{Hofer18,HofS21,Lebl22,MS22,Inv23}.
For the 2D Stokes-transport system $(\mathrm{ST}_1)$,
Grayer II \cite{Gray23} proved tha the Cauchy problem associated with data in $L^1\cap L^\infty$ is globally well-posed. As well, he proved the global persistence of $C^{k+\gamma}$ ($k\in\{0,1,2\}$) boundary regularity of the associated patch solutions.
Dalibard, Guillod and Leblond \cite{DGL23} studied the long-time behavior for the 2D Stokes-transport system in a channel
$\mathbf{D}= \mathbb{T}\times (0,1)$ (see also \cite{Park24}).
One can refer to \cite{AMY00,GGBS23,GGBS24} for some further regularity results for an interface of density in the Stokes-transport system.

The fractional Stokes-transport system $(\mathrm{ST}_\alpha)$ with $\alpha\in (0,1)$ can be seen as an intermediate model
between the inviscid incompressible porous media (IPM) equation (i.e. the $\alpha =0$ case, see e.g. \cite{CCGO09,CGO07})
and the Stokes-transport system.
In a very recent work, Cobb \cite{Cobb23} proved various well-posedness results in critical function spaces for the fractional Stokes-transport system in any dimension $d\geq 2$. In particular, the author showed a global well-posedness result
in the case $\{d=2,\,\alpha\in(\frac{1}{2},1)\}$ associated with data  $\theta_0\in L^p \cap L^{\frac{2}{2\alpha-1}}$
($1\leq p <\frac{1}{\alpha}$) and in the case $\{d=2,\,\alpha=\frac{1}{2}\}$ associated with
$\theta_0\in \dot B^0_{2,1}\cap \dot B^0_{\infty,1}$.
\vskip1mm

In this paper we study the patch problem for the 2D Boussinesq-Navier-Stokes system $(\mathrm{B}_\alpha)$
in the full subcritical regime, namely $\frac{1}{2}<\alpha\leq 1$.
One of the main task is to get a control which is independent on the Prandtl number $\mathrm{Pr}\in [1,\infty)$
in order to let the Prandtl number goes to infinity. This will allow us to recover the patch solutions
for the system $(\mathrm{ST}_\alpha)$ by passing to the limit.
The initial temperature $\theta_0$ is the patch of nonconstant values,
which is classically called the temperature front initial data.
This setting describes the evolution of the temperature front governed by the fluid flow, and it is an important physical scenario in geophysics \cite{AG1982,AJM03}. Let us now introduce our analytical setting regarding the initial condition. \\

Let $k\geq 2$  be an integer,
\begin{align}\label{eq:the-assum1}
  \theta_0(x) = \bar{\theta}_0(x) \mathbf{1}_{D_0}(x),\quad
  \bar{\theta}_0(x) \in
  \begin{cases}
    C^{k+\gamma-2\alpha}(\overline{D_0}),\quad & \textrm{for}\;\;\gamma\in(0,2\alpha-1), \alpha\in (\frac{1}{2},1], \\
    C^{k-1+\widetilde{\gamma}}(\overline{D_0}),\,\widetilde{\gamma}>0, \quad
    & \textrm{for}\;\; \gamma = 2\alpha-1,\alpha\in (\frac{1}{2},1),
  \end{cases}
\end{align}
where $D_0\subset \mathbf{D}$ is a bounded simply connected domain
with boundary
\begin{align}\label{eq:parD0-assum}
  \partial D_0\in C^{k+\gamma},\quad
  \gamma \in
  \begin{cases}
    (0,2\alpha-1],\quad & \textrm{if}\;\;\alpha\in (\frac{1}{2},1), \\
    (0,1),\quad & \textrm{if}\;\;\alpha =1.
  \end{cases}
\end{align}
We consider the level-set characterization of the domain $D_0$: there exists a function $\varphi_0 \in C^{k+\gamma}\left(\mathbf{D}\right)$
such that
\begin{equation}\label{eq:D_0}
  \partial D_0=\left\{x \in \mathbf{D}: \varphi_0(x)=0\right\},\;\;
  D_0=\left\{x \in \mathbf{D}: \varphi_0(x)>0\right\},
  \quad \nabla \varphi_0 \neq 0 \text { on } \partial D_0 .
\end{equation}
Then, the boundary $\partial D_0$ can be parameterized as
\begin{equation}\label{eq:z_0}
  z_0: \mathbb{S}^1 \mapsto \partial D_0 \quad \text { with }
  \partial_\alpha z_0(\alpha)=\nabla^{\perp} \varphi_0\left(z_0(\alpha)\right):= W_0\left(z_0(\alpha)\right),
\end{equation}
with $\nabla^{\perp}=\left(-\partial_2, \partial_1\right)^T$. In what follows we also set the viscosity $\nu=1$ for simplicity. \\

Our first main result is to show that the  $C^{k+\gamma}$ regularity of the boundary of the  patch is globally preserved for
the 2D subcritical Boussinesq system $(\mathrm{B}_\alpha)$, where
$k \in \mathbb{Z}^+$. In particular, we may clearly show the dependence on the Prandtl number $\mathrm{Pr}\in [1,\infty)$,
and also positively answer the question (a) raised by \cite{KX22} in the case $\alpha\in (1/2,1)$. As a matter of fact,  $\partial D(t)\in L^\infty_T(C^{k+\gamma})$, $k\geq 2$ implies that the curvature of $\partial D(t)$ is uniformly bounded.
\begin{theorem}\label{thm:main}
Let $\varepsilon =\frac{1}{\mathrm{Pr}} \in (0,1]$, $\alpha\in (\frac{1}{2},1]$, and $k\geq 2$ be an integer.
Let $\mathbf{D}$ be either $\mathbb{R}^2$ or $\mathbb{T}^2$.
Suppose that $\theta_0(x) = \bar{\theta}_0(x) \,\mathbf{1}_{D_0}(x)$ satisfies the above conditions
\eqref{eq:the-assum1}-\eqref{eq:parD0-assum}. Assume that the initial  velocity $u_0$ satisfies
\begin{itemize}
\item $u_0\in H^1\cap\dot{W}^{1,p}(\mathbf{D})$,
\item $(\partial_{W_0}u_0,\cdots, \partial_{W_0}^{k-1}u_0)\in W^{1,p}(\mathbf{D})$ for some $p>\frac{2}{2\alpha-1}$,
\item $\nabla\cdot u_0=0.$
\end{itemize}

Then, there exists a unique global solution $(u,\theta)$
to the 2D Boussinesq-Navier-Stokes system \eqref{BoussEq} which satisfies
\begin{equation}\label{eq:the-exp}
  \theta(x, t) = \bar{\theta}_0(X_t^{-1}(x)) \mathbf{1}_{D(t)}(x),
\end{equation}
with
\begin{equation}\label{eq:reg-parD(t)}
  \partial D(t) \in L^{\infty}\big([0, T], C^{k+\gamma} \big),
\end{equation}
where $D(t)=X_t(D_0)$,  $X_t$ is the particle-trajectory solving the equation \eqref{eq:X}
and $X^{-1}_t$ is its inverse. In particular, for the cases that
\begin{align}\label{eq:cases}
  \textrm{either \big\{$\alpha\in (\tfrac{1}{2},1)\big\}$,\quad
  or\quad \big\{$\alpha=1$, $\mathbf{D}=\mathbb{T}^2$\big\}},
\end{align}
the result \eqref{eq:reg-parD(t)}
holds uniformly with respect to $\varepsilon$.
\end{theorem}

Note that the notation $\partial_{W_0} u_0 := W_0\cdot \nabla u_0 = \mathrm{div}\,(W_0\, u_0)$
means the directional derivative of $u_0$ along the divergence-free vector field $W_0 := \nabla^\perp \varphi_0$.

\begin{remark}\label{rem:k=1}
  It is important to mention  that in the case $k=1$ in Theorem \ref{thm:main}, we still have the same conclusion regarding the regularity of the patch. Indeed, if the initial data $u_0(x)$
and $\theta_0(x) = \bar{\theta}_0(x) \mathbf{1}_{D_0}(x)$ are such that
\begin{equation*}
\begin{cases}
  u_0\in H^1, \;\;\bar{\theta}_0\in L^\infty(\overline{D}_0), \;
  \partial D_0 \in C^{1+\gamma},\; \gamma\in (0,1),
  \; & \textrm{for}\;\; \alpha =1, \\
  u_0\in H^1\cap W^{1,p},\; p\in(2,\infty), \;
  \bar{\theta}_0\in L^\infty(\overline{D_0}), \; \partial D_0\in C^{1+\gamma},
  \gamma \in (0,2\alpha -1], \; & \textrm{for}\;\; \alpha \in (\frac{1}{2},1),
\end{cases}
\end{equation*}
then according to Propositions \ref{prop:ap-es1} and \ref{prop:ap-es2} below,
we have $\nabla u \in L^\infty([0,T], C^\gamma(\mathbf{D}))$ where $\gamma$ is such that \eqref{eq:parD0-assum}
with the corresponding estimates. Therefore, by combining with the following estimate (see e.g. \cite[Lemma 2.10]{CMX22})
\begin{align}\label{eq:Xt-C1gam}
  \|\nabla X_t^{\pm 1}\|_{C^\gamma} \leq e^{(2+\gamma) \int_0^t \|\nabla u\|_{L^\infty} \dd \tau}
  \bigg(1+ \int_0^t \|\nabla u(\tau)\|_{C^\gamma} \dd \tau  \bigg),
\end{align}
we find that $\partial D(t) \in L^\infty(0,T; C^{1+\gamma}(\mathbf{D}))$. Importantly, the above estimate is
uniform with respect to $\varepsilon$ for the cases \eqref{eq:cases}.
Note that the special structure of the temperature patches is not used in the proof of the global preservation of the
$C^{1+\gamma}$ boundary regularity of the patch. This is completely  different from the proof of the $C^{k+\gamma}$
($k\geq 2$) case treated in Theorem \ref{thm:main} where it makes a systematic use of the structure of the temperature patch.
\end{remark}

\begin{remark}\label{rem:crit}
  In the critical case $\alpha =\frac{1}{2}$,
it remains a very interesting question to show the global well-posedness and the persistence of regularity
of the patch boundary  for both the Boussinesq-Navier-Stokes system $(\mathrm{B}_{1/2})$
or the fractional Stokes-transport system $(\mathrm{ST}_{1/2})$.
From the system $(\mathrm{ST}_{1/2})$, we find that the relation of $u$ and $\theta$ can be written as
\begin{align*}
  u =  \nabla \partial_2 \Lambda^{-3} \theta + (\Lambda^{-1} \theta) \,e_2,
\end{align*}
which enjoys the same scaling with the Biot-Savart law of the 2D Euler equations: $u= \nabla ^\perp \Lambda^{-2} \theta$.
It is well-known that the vorticity patch problem for the
2D Euler equations was solved by \cite{Chem93,bertozzi1993} as we already mentioned in the introduction. One may therefore wonder if these  techniques can be adapted to get new regularity results in the patch problem for the critical system $(\mathrm{ST}_{1/2})$.
However, there is a clear difference in the constitutive relation linking $\theta$ and $u$. Indeed, noticing that
\begin{align}\label{eq:nab-u-rel-crit}
  \nabla u =  \nabla^2 \partial_2 \Lambda^{-3} \theta + (\nabla\Lambda^{-1} \theta) \,e_2,\quad \textrm{with}\; \;\,
  \theta = \mathbf{1}_{D(t)},
\end{align}
we see that the operators in front of $\theta$ in \eqref{eq:nab-u-rel-crit} are  singular integral operators with odd kernels,
which is different from the case of even kernels in the vorticity patch problem. The even kernels have  additional cancellation effect and
this property plays an essential role in proving the key Lipschitz estimate of velocity $u$ as in the works \cite{Chem93,bertozzi1993}.
For further developments regarding the fine properties of the singular integral with even kernels one may see \cite{MOV09,MOV11}.
\end{remark}

Our second main result deals with the infinity Prandtl number limit of the patch solutions for the
2D Boussinesq-Navier-Stokes system $(\mathrm{B}_\alpha)$ in the torus $\mathbb{T}^2$. We rigorously justify the convergence to the patch solutions of
the 2D (fractional) Stokes-transport system $(\mathrm{ST}_\alpha)$.
In particular, we provide an indirect proof that the $C^{k+\gamma}$ ($k\in\mathbb{Z}^+$)
boundary regularity of the patch solutions to the system $(\mathrm{ST}_\alpha)$ is preserved globally in time. 
Specially, our theorem extends the result of Grayer II \cite{Gray23} to the regime $k\geq 3$.
\begin{theorem}\label{thm:limit}
Let $\alpha\in \left(\frac{1}{2},1\right]$, $\varepsilon \in (0,1]$, $\mathbf{D}=\TT^2$, $k\geq 1$ be an integer.
Suppose that $\theta_0(x) = \bar{\theta}_0(x) \mathbf{1}_{D_0}(x)$ satisfying $\int_{\mathbb{T}^2} \theta_0\,\dd x =0$
is the temperature patch initial data that fulfills the assumptions
in Theorem \ref{thm:main} and Remark \ref{rem:k=1}.
Assume that $u_0^\varepsilon\in H^1\cap \dot W^{1,p}(\mathbb{T}^2)$,
$(\partial_{W_0}u_0^\varepsilon,\cdots, \partial_{W_0}^{k-1}u_0^\varepsilon)\in W^{1,p}(\mathbb{T}^2)$
for some $p>\frac{2}{2\alpha-1}$, $\nabla\cdot u^\varepsilon_0=0$, and they converge to $u_0$ and
$(\partial_{W_0}u_0,\cdots,\partial_{W_0}^{k-1}u_0)$ in the corresponding norms.
Let $(u^\varepsilon, \theta^\varepsilon)$ be the unique global regular solution to the 2D Boussinesq-Navier-Stokes system $(\mathrm{B}_\alpha)$
constructed in Theorem \ref{thm:main}.

Then, as $\varepsilon\rightarrow 0$, up to an extraction of a subsequence, $(u^\varepsilon, \theta^\varepsilon)$
converges to the global unique weak solution $(u, \theta)$ which is solution of the (fractional) Stokes-transport system
$(\mathrm{ST}_\alpha)$, and $(u,\theta)$ satisfies that
\begin{equation}\label{eq:limit-targ}
  \theta(x, t) = \bar{\theta}_0(X_t^{-1}(x)) \mathbf{1}_{D(t)}(x), \quad
  \textrm{with}\;\; \partial D(t) \in L^{\infty}\big([0, T], C^{k+\gamma}\big),
\end{equation}
where $D(t)=X_t(D_0)$,  $X_t$ is the particle-trajectory generated by the velocity $u$
and $X^{-1}_t$ is its inverse.
\end{theorem}

To prove the persistence  of the temperature patch boundary in $C^{1+\gamma}$ and $C^{2+\gamma}$, it suffices to prove that $\varphi$ belongs to $L^\infty(0,T;C^{1+\gamma}(\mathbf{D}))$ and
$L^\infty(0,T;C^{2+\gamma}(\mathbf{D}))$, respectively. Note that the domain $D(t)=X_t(D_0)$ can be determined by the level-set function $\varphi(x,t)=\varphi_0(X_t^{-1}(x))$ which solves
\begin{equation}\label{eq:varphi}
  \partial_t\varphi + u\cdot\nabla\varphi = 0,~~~\varphi(0,x)=\varphi_0(x).
\end{equation}
Then, motivated by the works  \cite{HKR10,HKR11} which has been recently applied in  \cite{CMX22,KX22},
one  introduces an auxiliary quantity $\Gamma$,
which plays an important role in the analysis.
First, recall that the equation for the vorticity $\omega:= \partial_1u_2-\partial_2u_1$ is expressed as
\begin{equation}\label{eq:vorticity}
  \varepsilon \big(\partial_t\omega +  u\cdot\nabla\omega\big) + \Lambda^{2\alpha}\omega = \partial_1\theta.
\end{equation}
We rewrite it as
$\left(\varepsilon \partial_t + \varepsilon u \cdot \nabla + \Lambda^{2 \alpha}\right) \omega-\Lambda^{2 \alpha} \Theta=0$,
where $\partial_1 \theta=\Lambda^{2 \alpha} \Theta$, which  may be rewritten as
$\Theta=\mathcal{R}_{1-2 \alpha} \theta:= \partial_1 \Lambda^{-2 \alpha} \theta$.
Hence, applying the operator $\mathcal{R}_{1-2\alpha}$ to the temperature equation yields
\begin{align*}
  \varepsilon \partial_t \Theta  + \varepsilon \,u\cdot \nabla \Theta = -\varepsilon [\mathcal{R}_{1-2\alpha}, u\cdot \nabla] \theta.
\end{align*}
Thus, if we set $\Gamma:=\omega-\Theta = \omega  -\mathcal{R}_{1-2\alpha} \theta$, then
we obtain the following equation for $\Gamma$
\begin{equation}\label{eq:Gamma}
  \varepsilon \big( \partial_t \Gamma +  u \cdot \nabla \Gamma \big) + \Lambda^{2 \alpha} \Gamma
  = \varepsilon \left[\mathcal{R}_{1-2 \alpha}, u \cdot \nabla\right] \theta.
\end{equation}
Note that the vorticity equation \eqref{eq:vorticity} enjoys the same structure as the equation \eqref{eq:Gamma},
however the forcing term $\varepsilon \left[\mathcal{R}_{1-2 \alpha}, u \cdot \nabla\right] \theta$ in \eqref{eq:Gamma}
is more regular than the original term $\partial_1\theta$. In particular, it allows to prove the uniform  estimates with respect to $\varepsilon$.
Taking advantage of the Biot-Savart law and the relation $\omega = \Gamma+\mathcal{R}_{1-2\alpha}\theta$, we have
\begin{equation}\label{eq:u-exp1}
  u=\nabla^{\bot}\Lambda^{-2}\omega=\nabla^{\bot}\Lambda^{-2}\Gamma+\nabla^{\bot}\partial_1\Lambda^{-2-2\alpha}\theta.
\end{equation}

We first note that the $L^\infty_T(L^2(\mathbf{D}))$ norm of the velocity $u$ has an upper bound given by \eqref{eq:u-L2-es}
that is growing in $\frac{T}{\varepsilon}$ (with $\varepsilon =\frac{1}{\mathrm{Pr}}$) and this seems to be an obstruction to get uniform estimates with respect to  $\varepsilon$.
However, by using the good unknown $\Gamma$, we  can get nice {\it a priori} estimates of $\nabla u$ and $\Gamma$ which are uniform with respect to  $\varepsilon$ in some cases \eqref{eq:cases} (see  Propositions \ref{prop:ap-es1} and \ref{prop:ap-es2}). It is worth mentioning that the commutator estimates in Lemma \ref{lem:Rbeta-cm} play a crucial  role to get nice estimates and  in particular it allows us to deal with the forcing term in \eqref{eq:Gamma}. Since the commutator estimate \eqref{eq:R-1cm-es1} requires  a control on $\|u\|_{L^2}$  (unlike \eqref{eq:Rbeta-cm-es} and \eqref{eq:R-1cm-es2} where one only needs to control a higher order semi-norm of $u$ in the cases \eqref{eq:cases}), this  makes the case $\{\alpha =1,\,\mathbf{D}=\mathbb{R}^2\}$ a bit particular.
In this special case, we only prove the estimates of $\Gamma$ and $\nabla u$ with upper bounds depending on
$\frac{T}{\varepsilon}$. Importantly, one may notice that Propositions \ref{prop:ap-es1} and \ref{prop:ap-es2} are enough
to get the global well-posedness of temperature patch solutions for the system $(\mathrm{B}_\alpha)$
together with the global persistence of the regularity $C^{1+\gamma}$ of the patch boundary $\partial D(t)$, see Remarks \ref{rem:apes1}, \ref{rem:prop-ap-es2}
below or Remark \ref{rem:k=1}.
\vskip1mm

In order to prove that $\varphi\in L^\infty([0,T],C^{2+\gamma}(\mathbf{D}))$,
we introduce the tangential derivative $W=\nabla^{\perp}\varphi$ which solves the equation \eqref{eq:W},
and we focus on the $C^\gamma$ norm of the quantity $\nabla W$.
It remains to control the $L^1([0,T],C^\gamma)$ norm of the term $\partial_{W}\nabla u$,
where $\partial_W=W\cdot\nabla$.
Since we have \eqref{eq:u-exp1}, we shall prove the estimates of $\partial_W\Gamma$
and $\partial_W\theta$ separately.
Regarding the control of $\partial_W \Gamma$,
we need to apply the smoothing estimate \eqref{eq:TD-sm4} in the equation \eqref{eq:parW} of $\partial_W \Gamma$ in order to give a good dependence of the coefficient $\varepsilon = \frac{1}{\mathrm{Pr}}$. This is crucial especially after having noticed that there is a singular forcing term $\frac{1}{\varepsilon} [\Lambda^{2\alpha},W\cdot\nabla]\Gamma$ in \eqref{eq:parW}. By taking advantage of the commutator estimates in \eqref{eq:commEs} and Lemma \ref{lem:Rbeta-cm},
we can control the $L^1_T(B^{\gamma'}_{\infty,1})$ norm of $\partial_W \Gamma$ with some specific $\gamma'>\gamma$. 
This control is done in terms of the integral of $\|W\|_{W^{1,\infty}}$ multiplying some norms of $\Gamma$, which, in combination with the striated estimate \eqref{eq:str-es-mDphi} gives a good estimate of the $L^1_T(C^\gamma)$ norm of $\partial_W\nabla \nabla^\perp \Lambda^{-2}\Gamma$.
Concerning  the control of $\partial_W \theta$, we use a key property in Lemma \ref{lem:str-reg} that holds for the patch initial data,
and by using again the striated estimate \eqref{eq:str-es-mDphi}, we can get an upper bound of the
$L^1_T(C^\gamma)$ norm of $\partial_W \nabla \nabla^\perp \partial_1 \Lambda^{-2-2\alpha}\theta$.
Hence, collecting these estimates and \eqref{eq:u-exp1},
we get that $\|W(t)\|_{C^{1+\gamma}}$ is bounded by the time integral of $\|\nabla W\|_{C^\gamma}$ times  some norms of
$(\Gamma,\theta,\nabla u)$. Then, the wanted global estimate follows from Gr\"onwall's inequality,
and this allows to finish the proof of Theorem \ref{thm:main} when $k=2$.
Along the proof, we need to consider separately the cases $\alpha \in(\frac{1}{2},1)$ and
$\alpha=1$ as they require different approach. 
\vskip1mm

In the proof of the propagation of higher order regularity, namely the  $C^{k+\gamma}$-regularity of the patch boundary, following the technics of \cite{Chem91}, it suffices to show the striated estimate $\partial_W^{k-1}W\in L^\infty([0,T],C^\gamma(\mathbf{D}))$.
We use the induction method to prove it.
Assume that we already control the quantities $W$, $\nabla u$, and $\Gamma$ in the appropriate
$\mathcal{B}_{p,r,W}^{s, \ell}$-norms (see Definition \ref{def:gBesS}) as in \eqref{eq:ind-assum}
with $\ell \in\{1, \ldots, k-2\}$,
then our aim is to establish the corresponding estimates for the step $\ell+1$.
The procedure is analogous to the proof of the $C^{2+\gamma}$-persistence result. The higher-order striated estimates in Lemma \ref{lem:CMXstra-es} play an important role in the proof.
In order to get a control of $W$ in $L_T^{\infty}\big(\mathcal{C}_{W}^{\gamma+1, \ell}\big)$, using the equation of $\partial_W^{\ell} \nabla^2 W$
and  the striated estimate \eqref{eq:CMX2.13},
we see that the main task is to get a control the of $\nabla^2 u$ in $L_T^1\big(\mathcal{C}_{W}^{\gamma-1, \ell+1}\big)$.
Since we have \eqref{eq:u-exp1}, we deal with $\Gamma$ and $\theta$ separately. Indeed, by applying the smoothing estimate \eqref{eq:TD-sm4} of the transport-diffusion equation
and the induction assumption, we obtain a good estimate of $\Gamma$ (see \eqref{eq:Gam-Bgam'-el+1}) in the space
$L^1_T\big(\mathcal{B}_W^{\gamma', \ell+1}\big)$
for a specific $\gamma'>\gamma$. This can be used to bound the $L^1_T(\mathcal{C}^{\gamma-1,\ell+1}_W)$ norm of
$\nabla^2 \nabla^{\perp}\Lambda^{-2} \Gamma$. As for the control of $\theta$, by using the patch structure of $\theta$ and the striated estimates
\eqref{eq:str-reg}-\eqref{eq:str-reg2}, one can bound the
$L_T^1\big(\mathcal{C}_{W}^{\gamma-1, \ell+1}\big)$ norm of
$\nabla^2 \nabla^\perp \partial_1\Lambda^{-2-2\alpha} \theta$ as in \eqref{eq:the-str-es2}.
Gathering all these estimates and using Gr\"onwall's inequality, we find the desired uniform estimates \eqref{eq:ind-conc} in the $\ell+1$ level,
so that the induction scheme can be continued to give the final objective. This leads to the statement of Theorem \ref{thm:main} with general $k\geq 3$.
\vskip1mm

As far as the proof of Theorem \ref{thm:limit} is concerned, it is mainly based on the use of  the uniform estimates with respect to $\varepsilon$  obtained in
Theorem \ref{thm:main} and classical compactness arguments.
Although the $L^2$ energy estimate \eqref{eq:u-L2-es} of $u$ is not enough to provide a uniform bound in $\varepsilon$,
we can consider the zero mode and non-zero mode separately to show the uniform boundedness of
$u^\varepsilon$ in $\varepsilon$ in the  $L^\infty(0,T; H^1(\mathbb{T}^2))$  topology. By using the Aubin-Lions lemma,
we deduce the strong convergence of $\theta^\varepsilon$ as in \eqref{eq:the-str-conv},
so that by sending $\varepsilon = \frac{1}{\mathrm{Pr}} \rightarrow 0$
we can prove  $(u^\varepsilon, \theta^\varepsilon)$ converges to $(u,\theta)$ and  solves the Stokes-transport system $(\mathrm{ST_\alpha})$
in the sense of distribution. For a more general and precise statement one may see Proposition \ref{prop:weak-limit}. As well, by studying the strong convergence of  the particle-trajectory $X^{\varepsilon,\pm1}_t$,
level-set function $\varphi^\varepsilon$ and striated quantity $\partial_{W^\varepsilon} W^\varepsilon$
in the appropriate topology, we can conclude that the limit function $\theta$ preserves the patch structure
and the $C^{k+\gamma}$-regularity of the patch boundary $\partial D(t)$,
as claimed.
\vskip1mm

The structure of the paper is as follows. In Section~\ref{sec:pre}, we introduce some useful tools:
The Subsection \ref{subsec:str-es} is the introduction of some background on striated type Besov spaces
$\mathcal{B}_{p, r, \mathcal{W}}^{s, \ell}$ together with several estimates in striated spaces. In the subsection \ref{subsec:aux-lem},
we collect some useful intermediary  lemmas. Section~\ref{sec:C2gamma} is devoted to the proof of the
 persistence of the $C^{2+\gamma}$ regularity.  In the section~\ref{sec:Ckgamma},
we give the proof of the persistence of the regularity in $C^{k+\gamma}$ for any $k \geq 3$. In Section \ref{sec:limit}, we give the proof of Theorem \ref{thm:limit} which deals with the passage to the limit to infinity of the Prandtl number.
Finally, the last section~\ref{sec:apdx} is  the proof of the striated estimates \eqref{eq:CMX2.15} and \eqref{eq:str-es-mDphi} in Lemma~\ref{lem:CMXstra-es}.

\textbf{Notations.} The following notations will be used throughout this paper.

\begin{enumerate}[$\bullet$]
  \item $\mathbb{N}:= \{0,1,2,3,\cdots\}$, $\mathbb{Z}^+ := \{1,2,3,\cdots\}$
  and $\mathbb{R}_+ := (0,+\infty)$.
  \item Classically, $\mathcal{S}\left(\mathbb{R}^d\right)$ is the Schwartz class of rapidly decreasing $C^{\infty}$ functions
  and by $\mathcal{S}^{\prime}\left(\mathbb{R}^d\right)$ the space of tempered distributions which is the dual space of $\mathcal{S}\left(\mathbb{R}^d\right)$.
 \item  We also use several times the notation $\left\|\left(f_1, \ldots, f_n\right)\right\|_X:= \left\|f_1\right\|_X+\cdots+\left\|f_n\right\|_X$.
  \item For two operators $\mathcal{X}$ and $\mathcal{Y}$, the notation
  $[\mathcal{X}, \mathcal{Y}]:= \mathcal{X} \mathcal{Y}-\mathcal{Y} \mathcal{X}$ denotes the commutator operator.
\end{enumerate}

\section{Some definitions and lemmas}\label{sec:pre}

\subsection{Striated type Besov spaces and related estimates.}\label{subsec:str-es}
We first define the classical Littlewood-Paley decomposition and the definition of the Besov spaces  (see \cite{BCD11, LEM16}). The idea is that one can choose two nonnegative radial functions $\chi, \varphi \in C_c^{\infty}\left(\mathbb{R}^d\right)$
that are supported respectively in the ball $\left\{\xi \in \mathbb{R}^d:|\xi| \leq \frac{4}{3}\right\}$
and in the annulus $\left\{\xi \in \mathbb{R}^d: \frac{3}{4}\leq|\xi| \leq \frac{8}{3}\right\}$ such that
\begin{equation*}
  \chi(\xi)+\sum_{j \geq 0} \varphi\left(2^{-j} \xi\right)=1 \quad \text { for all } \xi \in \mathbb{R}^d.
\end{equation*}
For all tempered distribution $f$, the dyadic block operators $\Delta_j$ and $S_j$ are defined by
\begin{equation}\label{eq:Sj}
\begin{aligned}
  & \Delta_{-1} f=\chi(D) f= \overline{h} * f, \quad \Delta_j f=\varphi\left(2^{-j} D\right) f
  = h_j * f, \quad \forall j \in \mathbb{N}, \\
  & S_j f=\chi\left(2^{-j} D\right) f = \sum_{-1 \leq l \leq j-1} \Delta_l f
  = \overline{h}_j * f, \quad \forall j \in \mathbb{N},
\end{aligned}
\end{equation}
where $h_j(\cdot) := 2^{jd} h(2^j \cdot)$, $h:= \mathcal{F}^{-1} \varphi\in \mathcal{S}(\mathbb{R}^d)$,
$\overline{h}_j(\cdot):= 2^{j d}
\overline{h}(2^j \cdot)$, $\overline{h}:=\mathcal{F}^{-1} \chi\in \mathcal{S}(\mathbb{R}^d)$.

For all $f, g \in \mathcal{S}^{\prime}\left(\mathbb{R}^d\right)$, we have the following Bony's decomposition:
\begin{equation*}
  f g = T_f g+T_g f+R(f, g),
\end{equation*}
with
\begin{equation*}
  T_f g:= \sum_{q \in \mathbb{N}} S_{q-1} f \Delta_q g,
  \quad R(f, g):= \sum_{q \geq-1} \Delta_q f \widetilde{\Delta}_q g,
  \quad \widetilde{\Delta}_q:= \Delta_{q-1}+\Delta_q+\Delta_{q+1}.
\end{equation*}
In what follows, for a vector field $W: \mathbb{R}^d \rightarrow \mathbb{R}^d$,
we also use the notation $T_{W \cdot \nabla}$ to denote the operator $\sum_{q \in \mathbb{N}} S_{q-1} W \cdot \nabla \Delta_q$. \\

Now we introduce the Besov space $B_{p, r}^s\left(\mathbb{R}^d\right)$ and its striated version.
\begin{defn}\label{def:gBesS}
Let $s\in \RR$, $(p,r)\in [1,\infty]^2$. Denote by $B^s_{p,r}= B^s_{p,r}(\RR^d)$ the space of tempered distributions $f\in \mathcal{S}'(\RR^d)$ such that
\begin{align*}
  \|f\|_{B^s_{p,r}} := \big\| \big\{2^{qs}  \|\Delta_q f\|_{L^p}\big\}_{q\geq -1}  \big\|_{\ell^r}  < \infty.
\end{align*}
For all $\ell\in \NN$, $N\in \mathbb{Z}^+$ and a set of regular vector fields $\mathcal{W}= \{W_i\}_{1\leq i \leq N}$
with $W_i:\RR^d\rightarrow \RR^d$,
denote by $\bb^{s,\ell}_{p,r,\mathcal{W}}=\bb^{s,\ell}_{p,r,\mathcal{W}}(\RR^d)$ the space of tempered distributions
$f\in B^s_{p,r}(\RR^d)$ such that
\begin{align*}
  \|f\|_{\bb^{s,\ell}_{p,r,\mathcal{W}}} := \sum_{\lambda=0}^\ell \|\partial_{\mathcal{W}}^\lambda f\|_{B^s_{p,r}}
  = \sum_{\lambda=0}^\ell \sum_{\lambda_i\in \NN,\lambda_1 + \cdots + \lambda_N=\lambda}
  \|\partial_{W_1}^{\lambda_1}\cdots \partial_{W_N}^{\lambda_N} f\|_{B^s_{p,r}}  <\infty ;
\end{align*}
we also denote by $\BB^{s,\ell}_{p,r,\mathcal{W}}= \BB^{s,\ell}_{p,r,\mathcal{W}}(\RR^d)$ the set of tempered distributions $f\in B^s_{p,r}(\RR^d)$ such that
\begin{equation*}
\begin{split}
  \|f\|_{\BB^{s,\ell}_{p,r, \mathcal{W}}}  := & \sum_{\lambda=0}^\ell \|(T_{\mathcal{W}\cdot\nabla})^\lambda f\|_{B^s_{p,r}} \\
  = & \sum_{\lambda=0}^\ell \sum_{\lambda_i\in \NN,\lambda_1 +\cdots + \lambda_N = \lambda}
  \|(T_{W_1\cdot \nabla})^{\lambda_1}\cdots (T_{W_N\cdot\nabla})^{\lambda_N} f\|_{B^s_{p,r}} < \infty .
\end{split}
\end{equation*}
In particular, when $p=\infty$, we always use the following notations
\begin{equation}\label{eq:abbr}
\begin{split}
  \mathcal{C}^{s,\ell}_{\mathcal{W}} := \bb^{s,\ell}_{\infty,\infty,\mathcal{W}} , \quad
  \widetilde{\mathcal{C}}^{s,\ell}_{\mathcal{W}} := \BB^{s,\ell}_{\infty,\infty,\mathcal{W}}, \qquad
  \bb^{s,\ell}_{\mathcal{W}} :=  \bb^{s,\ell}_{\infty,1,\mathcal{W}},
  \quad  \BB^{s,\ell}_{\mathcal{W}} := \BB^{s,\ell}_{\infty,1,\mathcal{W}}.
\end{split}
\end{equation}
Besides, if $\mathcal{W}$ contains only one regular vector field $W$, i.e. $\mathcal{W}= \{W\}$,
we also denote
\begin{align}
  \bb^{s,\ell}_{p,r,W} & := \Big\{f\in B^s_{p,r}(\RR^d)\, \big|\, \|f\|_{\bb^{s,\ell}_{p,r,W}}
  := \sum_{\lambda=0}^\ell \|\partial_W^\lambda f\|_{B^s_{p,r}} < \infty \Big\}, \label{norm:BBsln2-2}\\
  \BB^{s,\ell}_{p,r,W} & := \Big\{f\in B^s_{p,r}(\RR^d)\, \big|\, \|f\|_{\BB^{s,\ell}_{p,r,W}}
  := \sum_{\lambda=0}^\ell \|(T_{W\cdot\nabla})^\lambda f\|_{B^s_{p,r}} < \infty \Big\}, \label{norm:BBsln-2}
\end{align}
and similar notations \eqref{eq:abbr} are used with $W$ in place of $\mathcal{W}$.
\end{defn}

We shall use the Chemin-Lerner mixed space-time Besov space which is denoted
$\widetilde{L}^\rho\left([0, T] , B_{p, r}^s\right)$ and is the set of tempered distribution $g$ such that
$\|g\|_{\widetilde{L}_T^\rho\left(B_{p, r}^s\right)} := \Big\|\Big(2^{q s}\|\Delta_q g\|_{L_T^\rho(L^p)}\Big)_{q \geq-1}\Big\|_{\ell^r}<\infty$. \\

In the above, the notations $\partial_{\mathcal{W}}=\mathcal{W} \cdot \nabla$ and $T_{\mathcal{W}\cdot \nabla}$
respectively denote the vector-valued operators $\left\{W_i \cdot \nabla\right\}_{1 \leq i \leq N}$
and $\left\{T_{W_i\cdot \nabla}\right\}_{1 \leq i \leq N}$,
and $\partial_{\mathcal{W}}^\lambda=\Big\{\partial_{W_1}^{\lambda_1} \cdots \partial_{W_N}^{\lambda_N}: \lambda_1+
\cdots+\lambda_N=\lambda, \lambda_i \in \mathbb{N}\Big\}$
and $\left(T_{\mathcal{W} \cdot \nabla}\right)^\lambda=\Big\{\left(T_{W_1 \cdot \nabla}\right)^{\lambda_1} \cdots
\left(T_{W_N \cdot \nabla}\right)^{\lambda_N}: \lambda_1+\cdots+\lambda_N= \lambda, \lambda_i \in \mathbb{N}\Big\}$
for all $\lambda \in \mathbb{N}$.

\begin{remark}\label{rem:Besov-torus}
  The above definition of Besov space $B^s_{p,r}(\mathbb{R}^d)$ and striated Besov space
$\mathcal{B}^{s,\ell}_{p,r,\mathcal{W}}(\mathbb{R}^d)$
can be extended to their counterparts  $B^s_{p,r}(\mathbb{T}^d)$ and
$\mathcal{B}^{s,\ell}_{p,r,\mathcal{W}}(\mathbb{T}^d)$ for distributions defined on $\mathbb{T}^d$.
Indeed, one can view a function $f$ on $\mathbb{T}^d$ as a 1-periodic function of $\mathbb{R}^d$
in all coordinate, i.e. $f(x+m)=f(x)$ for all $x\in\mathbb{R}^d$ and $m\in\mathbb{Z}^d$,
thus recalling \eqref{eq:Sj}, we have for all $x\in \mathbb{T}^d$, $j\in \mathbb{N}$,
\begin{align*}
  \Delta_j f (x) = \int_{\mathbb{R}^d} h_j(x-y) f(y) \dd y = \sum_{m\in \mathbb{Z}^d}
  \int_{\mathbb{T}^d + m} h_j(y) f(x-y)\dd y
  = \int_{\mathbb{T}^d } \mathbf{h}_j(y) f(x-y) \dd y,
\end{align*}
and
\begin{align*}
  S_j f(x) = \int_{\mathbb{R}^d} \overline{h}_j(y) f(x-y) \dd y
  = \int_{\mathbb{T}^d } \overline{\mathbf{h}}_j(y) f(x-y) \dd y
  = \overline{\mathbf{h}}_j * f(x),
\end{align*}
with
\begin{align*}
  \mathbf{h}_j(y) = \sum_{m\in\mathbb{Z}^d} h_j(y+m)
  = \sum_{m\in\mathbb{Z}^d} 2^{jd} h(2^j (y+m)),\quad
  \overline{\mathbf{h}}_j(y) = \sum_{m\in\mathbb{Z}^d} \overline{h}_j(y+m).
\end{align*}
From Poisson's summation formula, we see that $\mathbf{h}_j$ and $\overline{\mathbf{h}}_j$
have only discrete spectrum with
\begin{align*}
  \mathbf{h}_j(y) = \sum_{m\in \mathbb{Z}^d} \varphi(2^{-j} m) e^{2\pi i m\cdot y},
  \quad\textrm{and}\quad \overline{\mathbf{h}}_j(y) = \sum_{m\in\mathbb{Z}^d} \chi(2^{-j}m) e^{2\pi i m\cdot y}.
\end{align*}
\end{remark}

Some basic properties of the space $\mathcal{B}_{p, r, \mathcal{W}}^{s, \ell}$ are presented as follows.
\begin{lemma}Let $s, \widetilde{s} \in \mathbb{R}, \ell, \widetilde{\ell} \in \mathbb{N}, r, \widetilde{r} \in[1, \infty],~p \in[1, \infty]$, and $\mathcal{W}=\left\{W_i\right\}_{1 \leq i \leq N}$ be composed of regular vector fields $W_i: \mathbb{R}^d \rightarrow \mathbb{R}^d$. The function space $\mathcal{B}_{p, r, \mathcal{W}}^{s, \ell}$  satisfies that
\begin{equation*}
\begin{aligned}
  &\mathcal{B}_{p, r, \mathcal{W}}^{s, \ell} \subset \mathcal{B}_{p, r, \mathcal{W}}^{\widetilde{s}, \ell} \;\;\text { for } s \geq \widetilde{s},\qquad
  \mathcal{B}_{p, r, \mathcal{W}}^{s, \ell} \subset \mathcal{B}_{p, r, \mathcal{W}}^{s, \widetilde{\ell}} \;\;\text { for } \ell \geq \widetilde{\ell},\\
  &\mathcal{B}_{p, r, \mathcal{W}}^{s, \ell} \supset \mathcal{B}_{p, \widetilde{r}, \mathcal{W}}^{s, \ell}
  \;\; \text { for } r \geq \widetilde{r},
\end{aligned}
\end{equation*}
and
\begin{equation}\label{eq:stBes-fact}
  \|f\|_{\mathcal{B}_{p, r, \mathcal{W}}^{s, \ell+1}}
  =\big\|\partial_{\mathcal{W}}^{\ell+1} f\big\|_{B_{p, r}^s}
  +\|f\|_{\mathcal{B}_{p, r, \mathcal{W}}^{s, \ell}},
  \quad\|f\|_{\mathcal{B}_{p, r, \mathcal{W}}^{s, \ell+1}}
  =\left\|\partial_{\mathcal{W}} f\right\|_{\mathcal{B}_{p, r, \mathcal{W}}^{s, \ell}}+\|f\|_{B_{p, r}^s} .
\end{equation}
\end{lemma}
The following striated estimates dealing with the spaces $\mathcal{B}_{p, r, \mathcal{W}}^{s, \ell}$ will play a crucial role in the proof of the main results. The  proof of this lemma is provided in Section \ref{sec:apdx}.

\begin{lemma}\label{lem:CMXstra-es}
  Let $k \in \mathbb{N}$, $\rho \in(0,1)$, $N \in \mathbb{Z}^+$,
and $\mathcal{W}=\left\{W_i\right\}_{1 \leq i \leq N}$ be a set of regular divergence-free vector fields
$W_i: \mathbb{R}^d \rightarrow \mathbb{R}^d$ satisfying that
\begin{equation*}
  \|\mathcal{W}\|_{\mathcal{C}_{\mathcal{W}}^{1+\rho, k-1}}:= \sum_{\lambda=0}^{k-1}
  \big\|\partial_{\mathcal{W}}^\lambda \mathcal{W}\big\|_{C^{1+\rho}}
  <\infty .
\end{equation*}
Let $m(D) = \Lambda^\sigma m_0(D)$, $\sigma>-1$, and $m_0(D)$ be a zero-order pseudo-differential operator with
$m_0(\xi) \in C^{\infty}\left(\mathbb{R}^d \backslash\{0\}\right)$.
Assume that $u$ is a smooth divergence-free vector field of $\mathbb{R}^d$,
and $\phi: \mathbb{R}^d \rightarrow \mathbb{R}$ is a smooth function.
Then the following statements hold true.
\begin{enumerate}
\item For all $\epsilon\in(0,1)$ and $(p,r)\in [1,\infty]^2$, we have
\begin{equation}\label{eq:CMX2.13}
  \|u \cdot \nabla \phi\|_{\mathcal{B}_{p, r, \mathcal{W}}^{-\epsilon, k}}
  \leq C \min \bigg\{\sum_{\mu=0}^k\|u\|_{\mathcal{B}_{\mathcal{W}}^{0, \mu}}
  \|\nabla \phi\|_{\mathcal{B}_{p, r, \mathcal{W}}^{-\epsilon, k-\mu}},
  \sum_{\mu=0}^k\|u\|_{\mathcal{B}_{p, r, \mathcal{W}}^{-\epsilon, \mu}}
  \|\nabla \phi\|_{\mathcal{B}_{\mathcal{W}}^{0, k-\mu}}\bigg\} .
\end{equation}
\item For all $s\in(-1,1)$, $-1<\sigma+s<1$ and $(p,r)\in [1,\infty]^2$, we have
\begin{equation}\label{eq:CMX2.15}
  \|[m(D), u \cdot \nabla] \phi\|_{\mathcal{B}_{p, r, \mathcal{W}}^{s, k}}
  \leq C\left(\|\nabla u\|_{\mathcal{B}_{\mathcal{W}}^{0, k}}
  +\|u\|_{L^{\infty}}\right)\|\phi\|_{\mathcal{B}_{p, r, \mathcal{W}}^{s+\sigma, k}}.
\end{equation}
\item For all $s\in(-1,1)$, $-1<\sigma+s<1$ and $(p,r)\in [1,\infty]^2$, we have
\begin{equation}\label{eq:str-es-mDphi}
\begin{aligned}
  \|m(D) \phi\|_{\mathcal{B}_{p, r, \mathcal{W}}^{s, k+1}}
  \leq C\|\phi\|_{\mathcal{B}_{p, r, \mathcal{W}}^{s+\sigma, k+1}}
  + C\|\mathcal{W}\|_{\mathcal{B}_{\mathcal{W}}^{1, k}}
  \Big(\|\phi\|_{\mathcal{B}_{p, r, \mathcal{W}}^{s+\sigma, k}} + \mathbf{1}_{\{-1<\sigma \leq 0\}}
  \|\Delta_{-1} m(D) \phi\|_{L^p} \Big) ,
\end{aligned}
\end{equation}
where $\mathbf{1}_{\{-1<\sigma\leq 0\}}$ is the characteristic function of the set $\{-1<\sigma\leq 0\}$.
\end{enumerate}
In the above, $C>0$ depends  on $d,k,\epsilon,\sigma,s$
and $\|\mathcal{W}\|_{\mathcal{C}^{1+\rho,k-1}_{\mathcal{W}}}$
(when $k=0$ this norm vanishes).
\end{lemma}
In particular, for the special cases $k=0$ and $k=1$, the dependence on the lower order term $\|\mathcal{W}\|_{\mathcal{C}_{\mathcal{W}}^{1+\rho, k-1}}$ in the constant $C$ in Lemma~\ref{lem:CMXstra-es}
can be explicitly calculated. The corresponding striated estimates are stated as follows and they will play an important role in the proof as well.
\begin{lemma}\label{lem:CMX2.5}
Assume that $u$ is a smooth divergence-free vector field of $\mathbb{R}^d~(d \geq 2)$
and~$\mathcal{W}=\left\{W_i\right\}_{1 \leq i \leq N}$ ($N \in \mathbb{Z}^{+}$) is a set of smooth divergence-free vector fields.
Let $\phi: \mathbb{R}^d \rightarrow \mathbb{R}$ be a smooth function.
Let $m(D) = \Lambda^\sigma m_0(D)$, $\sigma>-1$, and $m_0(D)$ be a zero-order pseudo-differential operator with
$m_0(\xi) \in C^{\infty}\left(\mathbb{R}^d \backslash\{0\}\right)$.
Then the following statements hold true.
\begin{enumerate}
\item For all $\epsilon \in(0,1)$ and $(p, r) \in[1, \infty]^2$, there exists a constant $C=C(d, \epsilon)>0$ such that
\begin{equation}\label{eq:prodBes}
  \|u \cdot \nabla \phi\|_{B_{p, r}^{-\epsilon}}
  \leq C \min \left\{\|u\|_{B_{p, r}^{-\epsilon}}\|\nabla \phi\|_{L^{\infty}},\|u\|_{L^{\infty}}\|\nabla \phi\|_{B_{p, r}^{-\epsilon}}\right\},
\end{equation}
and
\begin{equation}\label{eq:paWprod-es}
  \left\|\partial_{\mathcal{W}}(u \cdot \nabla \phi)\right\|_{B_{p, r}^{-\epsilon}}
  + \left\|T_{\mathcal{W} \cdot \nabla}(u \cdot \nabla \phi)\right\|_{B_{p, r}^{-\epsilon}} \leq C \min \left\{A_1, A_2, A_3\right\},
\end{equation}
with
\begin{align*}
  A_1:= & \|u\|_{B_{p, r}^{-\epsilon}}\left\|\partial_{\mathcal{W}} \nabla \phi\right\|_{B_{\infty, 1}^0}
  +\left(\left\|\partial_{\mathcal{W}} u\right\|_{B_{p, r}^{-\epsilon}}+\|\mathcal{W}\|_{B_{\infty, 1}^1}
  \|u\|_{B_{p, r}^{-\epsilon}}\right)\|\nabla \phi\|_{B_{\infty, 1}^0}, \\
  A_2:=  & \|u\|_{B_{\infty, 1}^0}\left\|\partial_{\mathcal{W}} \nabla \phi\right\|_{B_{p, r}^{-\epsilon}}
  +\left(\left\|\partial_{\mathcal{W}} u\right\|_{B_{\infty, 1}^0}
  +\|\mathcal{W}\|_{B_{\infty, 1}^1}\|u\|_{B_{\infty, 1}^0}\right)\|\nabla \phi\|_{B_{p, r}^{-\epsilon}}, \\
  A_3:=  & \|u\|_{B_{\infty, 1}^0}\left(\left\|\partial_{\mathcal{W}} \nabla \phi\right\|_{B_{p, r}^{-\epsilon}}
  +\|\mathcal{W}\|_{B_{\infty, 1}^1}\|\nabla \phi\|_{B_{p, r}^{-\epsilon}}\right)
  +\left(\left\|\partial_{\mathcal{W}} u\right\|_{B_{p, r}^{-\epsilon}}+\|\mathcal{W}\|_{B_{\infty, 1}^1}
  \|u\|_{B_{p, r}^{-\epsilon}}\right)\|\nabla \phi\|_{B_{\infty, 1}^0}.
\end{align*}
\item For all $s \in(-1,1)$, $-1<\sigma+ s <1$ and $(p, r) \in[1, \infty]^2$, there exists a constant $C=C(d, s, \sigma)>0$ so that
\begin{equation}\label{eq:commEs}
  \|[m(D), u \cdot \nabla] \phi\|_{B_{p, r}^s} \leq C\|u\|_{W^{1, \infty}}\|\phi\|_{B_{p, r}^{s + \sigma}}.
\end{equation}
\item For all $s \in(-1,1)$, $-1<\sigma+ s <1$ and $(p, r) \in[1, \infty]^2$, there exists a constant $C=C(d, s, \sigma)>0$ so that
\begin{equation}\label{eq:paWmD-es}
  \left\|\partial_{\mathcal{W}}(m(D) \phi)\right\|_{B_{p,r}^s}
  \leq C\left\|\partial_{\mathcal{W}} \phi\right\|_{B_{p, r}^{s +\sigma}}+C\|\mathcal{W}\|_{W^{1, \infty}}\|\phi\|_{B_{p, r}^{s+\sigma}}.
\end{equation}
\end{enumerate}
\end{lemma}

\begin{proof}[Proof of Lemma~\ref{lem:CMX2.5}]
Estimates \eqref{eq:prodBes} and \eqref{eq:paWprod-es} are exactly the same as those in \cite[Lemma 2.5]{CMX22},
thus we only need to prove (ii) and (iii).

For the estimation of \eqref{eq:commEs}, using  Bony's decomposition, we have
\begin{align*}
  {[m(D), u \cdot \nabla] \phi=} & \sum_{j \in \mathbb{N}}\left[m(D), S_{j-1} u \cdot \nabla\right] \Delta_j \phi
  +\sum_{j \in \mathbb{N}}\left[m(D), \Delta_j u \cdot \nabla\right] S_{j-1} \phi \\
  & +\sum_{j \geq-1}\left[m(D), \Delta_j u \cdot \nabla\right] \widetilde{\Delta}_j \phi \\
  := &\, {I}_1 + {I}_2 + {I}_3.
\end{align*}
Noting that there exist $\widetilde{\psi}\in C^\infty_c(\mathbb{R}^d)$ supported in an annulus away from the origin and
$\widetilde{h} = \mathcal{F}^{-1}(m \widetilde{\psi}) = \mathcal{F}^{-1}(|\xi|^\sigma m_0 \widetilde{\psi})
\in \mathcal{S}\left(\mathbb{R}^d\right)$ such that
\begin{eqnarray}\label{eq:comm-formu1}
   \left[m(D), S_{j-1} u \cdot \nabla\right] \Delta_j \phi
  &=& \big[m(D)\widetilde{\psi}(2^{-j}D), S_{j-1} u \cdot \nabla\big] \Delta_j \phi\\
  &=& 2^{j(\sigma+d)} \int_{\mathbb{R}^d} \widetilde{h}(2^jy)\left(S_{j-1} u
  \left(x-y\right)-S_{j-1} u(x)\right) \cdot \nabla \Delta_j \phi\left(x-y\right)  \mathrm{d} y, \nonumber
\end{eqnarray}
we find that
\begin{equation*}
\begin{aligned}
  2^{q s}\left\|\Delta_q  {I}_1\right\|_{L^p}
  & \lesssim  2^{q s} \sum_{j \in \mathbb{N},|q-j| \leq 4}
  \left\|\Delta_q\left[m(D), S_{j-1} u \cdot \nabla\right] \Delta_j \phi\right\|_{L^p} \\
  & \lesssim  2^{qs} \sum_{j\in \NN, |j-q|\leq 4} 2^{j(\sigma -1)} \|\nabla S_{j-1} u\|_{L^\infty} \|\nabla \Delta_j \phi\|_{L^p} \\
  & \lesssim  c_q\|\nabla u\|_{L^{\infty}}\|\phi\|_{B_{p, r}^{s +\sigma}},
\end{aligned}
\end{equation*}
with $\left\{c_q\right\}_{q \geq-1}$ satisfying $\left\|c_q\right\|_{\ell^r}=1$.
For ${I}_2$, taking advantages of the following fact that (using Lemma~\ref{lem:m(D)}),
\begin{align}\label{eq:fact1a}
  \|\nabla m(D)\Delta_j\phi\|_{L^p} \leq C 2^{j(1+\sigma)} \|\Delta_j\phi\|_{L^p},\quad
 \forall p\in [1,\infty],\, \sigma>-1,\,j\geq -1,
\end{align}
we have, by setting $A:=2^{q s} \left\|\Delta_q {I}_2\right\|_{L^p}$
\begin{eqnarray*}
  A&=& C 2^{q s} \sum_{j \in \mathbb{N},|q-j| \leq 4}
  \big\|\Delta_q \left[m(D), \Delta_j u \cdot \nabla\right] S_{j-1} \phi\big\|_{L^p} \\
  &\leq& C 2^{q s} \sum_{j \in \mathbb{N},|q-j| \leq 4}
  \Big(\big\|\Delta_q m(D)\widetilde{\psi}(2^{-j}D)\big(\Delta_j u \cdot\nabla S_{j-1}\phi\big)\big\|_{L^p} +
  \big\|\Delta_q\big(\Delta_j u\cdot\nabla m(D)S_{j-1}\phi\big)\big\|_{L^p}\Big) \\
  &\leq& C  2^{q s} \sum_{j \in \mathbb{N},|q-j| \leq 4} \|\Delta_j u\|_{L^\infty}
  \Big( 2^{j\sigma} \left\|\nabla S_{j-1} \phi\right\|_{L^p} + \|\nabla m(D) S_{j-1}\phi\|_{L^p} \Big)  \\
  &\leq& C  \sum_{j \in \mathbb{N},|q-j| \leq 4} 2^{j(s-1)} \|\nabla \Delta_j u\|_{L^\infty}
  \bigg( 2^{j\sigma}\sum_{-1\leq j'\leq j-1} 2^{j'} \|\Delta_{j'}\phi\|_{L^p} +
  \sum_{-1\leq j'\leq j-1} 2^{j'(1+\sigma)} \|\Delta_{j'}\phi\|_{L^p}  \bigg)  \\
  &\leq& C \|\nabla u\|_{L^{\infty}} \sum_{j \in \mathbb{N},|j-q| \leq 4}
  \sum_{-1\leq j^{\prime} \leq j-1} \Big(2^{(j^{\prime}-j)(1-s-\sigma)} + 2^{(j'-j)(1-s)} \Big) 2^{j'(s+\sigma)} \left\|\Delta_{j^{\prime}} \phi\right\|_{L^p} \\
  &\leq& C c_q\|\nabla u\|_{L^{\infty}}\|\phi\|_{B_{p, r}^{s + \sigma}},
\end{eqnarray*}
where $\left\{c_q\right\}_{q \geq-1}$ is such that $\left\|c_q\right\|_{\ell^r}=1$.
The term ${I}_3$ can be decomposed as the following
\begin{equation*}
\begin{aligned}
  {I}_3 & = \sum_{j \geq -1} m(D) \operatorname{div}\big(\Delta_j u \,\widetilde{\Delta}_j \phi\big)
  - \sum_{j \geq -1}  \operatorname{div}\big(\Delta_j u\, m(D) \widetilde{\Delta}_j \phi\big) \\
  & := {I}_{3,1} + {I}_{3,2}.
\end{aligned}
\end{equation*}
For ${I}_{3,1}$ and ${I}_{3,2}$, in view of \eqref{eq:fact1a} and the discrete Young's inequality,
we infer that for all $s\in(-1,1)$, $-1<s+\sigma<1$,
\begin{align*}
  2^{q s} \left\|\Delta_q {I}_{3,1} \right\|_{L^p}
  & \leq C 2^{q s} \sum_{j \geq \max\{q-3, -1\}}
  \big\|\Delta_q m(D) \operatorname{div}\big(\Delta_j u \widetilde{\Delta}_j \phi\big) \big\|_{L^p} \\
  & \leq C \sum_{j \geq \max\{q-3,  -1\}} 2^{(q-j)(1+ s +\sigma)}
  2^j \left\|\Delta_j u\right\|_{L^{\infty}} 2^{j(s +\sigma)} \big\|\widetilde{\Delta}_j \phi\big\|_{L^p} \\
  & \leq C c_q\|u\|_{B_{\infty,\infty}^1} \|\phi\|_{B_{p,r}^{s +\sigma}}
  \leq C c_q\|u\|_{W^{1,\infty}} \|\phi\|_{B_{p,r}^{s +\sigma}},
\end{align*}
and, set $B:=2^{q s}\left\|\Delta_q {I}_{3,2}\right\|_{L^p} $
\begin{eqnarray*}
  B&\leq& C 2^{qs} \bigg(\sum_{j\geq \max\{q-3,2\}} \big\|\Delta_q \mathrm{div}\,\big(\Delta_j u\,
  m(D)\widetilde{\Delta}_j\phi\big)\big\|_{L^p}
  + \sum_{j\geq q-3, j\leq 2} \big\|\Delta_q \big(\Delta_j u\cdot\nabla m(D)
  \widetilde{\Delta}_j\phi \big)\big\|_{L^p}\bigg) \\
    &\leq& C \bigg( 2^{q(1+s)} \sum_{j \geq \max\{q-3,2\}}
  \left\|\Delta_j u\right\|_{L^{\infty}}
  2^{j\sigma}\big\| \widetilde{\Delta}_j \phi\big\|_{L^p}
  + \mathbf{1}_{\{-1\leq q\leq 5\}} \sum_{-1\leq j\leq 2} \|\Delta_j u\|_{L^\infty}
  \|\widetilde{\Delta}_j \phi\|_{L^p} \bigg) \\
    &\leq& C \sum_{j \geq \max\{q-3, 2\}} 2^{(q-j)(1+s)} 2^j\left\| \Delta_j u\right\|_{L^{\infty}}
  2^{j(s+ \sigma)} \big\|\widetilde{\Delta}_j \phi\big\|_{L^p}
  + \mathbf{1}_{\{-1\leq q\leq 5\}} \|u\|_{L^\infty} \|\phi\|_{B^s_{p,r}} \\
    &\leq& C \big(c_q + \mathbf{1}_{\{-1\leq q\leq 5\}} \big)
  \|u\|_{W^{1,\infty}} \|\phi\|_{B_{p,r}^{s +\sigma}},
\end{eqnarray*}
where $\left\{c_q\right\}_{q \geq-1}$ is such that $\left\|c_q\right\|_{\ell^r}=1$.
Collecting all the above controls gives the estimate \eqref{eq:commEs}.

Next we prove \eqref{eq:paWmD-es}.
Noting that $\partial_\mathcal{W} \big(m(D)\phi\big) = -[m(D), \mathcal{W}\cdot\nabla]\phi + m(D) \partial_\mathcal{W}\phi$,
and using \eqref{eq:commEs}, we have
\begin{align*}
  \|\partial_\mathcal{W} \big( m(D)\phi\big)\|_{B^s_{p,r}} & \leq \|[m(D),\mathcal{W}\cdot\nabla]\phi\|_{B^s_{p,r}}
  + \|m(D)\partial_\mathcal{W} \phi\|_{B^s_{p,r}} \\
  & \leq C \|\mathcal{W}\|_{W^{1,\infty}} \|\phi\|_{B^{s+\sigma}_{p,r}} + C \|\partial_\mathcal{W}\phi\|_{B^{s+\sigma}_{p,r}}
  + C \|\Delta_{-1} m(D)\partial_\mathcal{W} \phi\|_{L^p}.
\end{align*}
By applying Bony's decomposition and \eqref{eq:fact1a},
we obtain that
\begin{align}\label{eq:Del-1mDpaW}
  \|\Delta_{-1} m(D)\partial_\mathcal{W} \phi\|_{L^p}
  &\leq \sum_{0\leq j\leq 4} \|\Delta_{-1} m(D) \mathrm{div} (S_{j-1}\mathcal{W} \,\Delta_j \phi ) \|_{L^p}
  + \sum_{0\leq j\leq 4} \|\Delta_{-1} m(D) \mathrm{div}\, (\Delta_j \mathcal{W} \, S_{j-1}\phi)\|_{L^p} \nonumber \\
  & \quad + \sum_{j\geq -1} \big\|\Delta_{-1}m(D) \mathrm{div} \big(\Delta_j\mathcal{W}\, \widetilde{\Delta}_j \phi \big) \big\|_{L^p}
  \nonumber \\
  & \leq C \|\mathcal{W}\|_{L^\infty} \sum_{-1\leq j\leq 4} \|S_{j+1}\phi\|_{L^p} + C \sum_{j\geq -1} \|\Delta_j \mathcal{W}\|_{L^\infty}
  \|\widetilde{\Delta}_j\phi\|_{L^p} \nonumber \\
  & \leq C \|\mathcal{W}\|_{B^1_{\infty,\infty}} \|\phi\|_{B^{s+\sigma}_{p,r}} \bigg( 1 + \sum_{j\geq -1} 2^{-j(s+\sigma+1)}\bigg) \nonumber \\
  & \leq C \|\mathcal{W}\|_{W^{1,\infty}} \|\phi\|_{B^{s+\sigma}_{p,r}} .
\end{align}
Thus combining the above two estimates leads to \eqref{eq:paWmD-es}, as desired.
\end{proof}

The lemma below deals with the striated estimate of  patch-type initial data.
\begin{lemma}\label{lem:str-reg}
Let $\mathbf{D}$ be either $\mathbb{R}^2$
or $\mathbb{T}^2$.   Let $\alpha\in (\frac{1}{2},1]$, $k\geq 2$ and $0<\gamma\leq 2\alpha-1$.
Assume that $D_0 \subset \mathbf{D}$ is a bounded simply connected domain with boundary
$\partial D_0$ characterized by the level-set function $\varphi_0\in C^{k+\gamma}(\mathbf{D})$ for some $\gamma\in (0,2\alpha-1]$
if $\alpha\in (\frac{1}{2},1)$ and for some $\gamma\in(0,1)$ if $\alpha =1$.
Denote by $W_0 = \nabla^\perp \varphi_0$.
\begin{enumerate}[$(1)$]
\item
If $\alpha\in (\frac{1}{2},1]$, $0<\gamma <2\alpha-1$ and $\theta_0(x) = \bar{\theta}_0(x) \mathbf{1}_{D_0}(x)$ with $\bar{\theta}_0 \in C^{k-2\alpha+\gamma}(\overline{D_0})$, then we have
\begin{align}\label{eq:str-reg}
  \partial_{W_0}^{k-1} \theta_0(x) \in C^{\gamma-2\alpha+1}(\mathbf{D}).
\end{align}
\item
If $\alpha\in (\frac{1}{2},1)$, $\gamma=2\alpha-1$ and $\theta_0(x) = \bar{\theta}_0(x) \mathbf{1}_{D_0}(x)$ with $\bar{\theta}_0 \in C^{k-1+\widetilde{\gamma}}(\overline{D_0})$, $\widetilde{\gamma}>0$, then we have
\begin{align}\label{eq:str-reg2}
  \partial_{W_0}^{k-1} \theta_0(x) \in L^{\infty}(\mathbf{D}).
\end{align}
\end{enumerate}
\end{lemma}

\begin{proof}[Proof of Lemma \ref{lem:str-reg}]
We give the proof for $\mathbf{D}=\RR^2$ first. The proof of \eqref{eq:str-reg} in the case $0<\gamma<2\alpha-1$ is the same with \cite[Lemma 2.6]{CMX22}.
Thus we only need to sketch the proof of \eqref{eq:str-reg2}.

First note that Rychkov's extension theorem \cite{VSR1999} guarantees that there exists a function
$\widetilde{\theta}_0\in C^{k-1+\widetilde{\gamma}}(\RR^2)$ with the following restriction condition
$\widetilde{\theta}_0|_{D_0}=\theta_0$.
Then, it suffices to prove that $ \partial_{W_0}^{k-1} (\widetilde{\theta}_0\mathbf{1}_{D_0} )\in L^{\infty}(\mathbb{R}^2)$.
Since the vector field $W_0$ is tangential to the patch boundary $\partial D_0$, the operator $\partial_{W_0}^{k-1}$
commutes with the characteristic function $\mathbf{1}_{D_0}$, and thus we only need to prove
$\partial_{W_0}^{k-1}\widetilde{\theta}_0\in L^\infty (\RR^2)$.
In fact,
\begin{equation*}
\begin{aligned}
  \|\partial_{W_0}^{k-1}\widetilde{\theta}_0\|_{L^\infty} & \leq \|\partial_{W_0}^{k-1}\widetilde{\theta}_0\|_{C^{\widetilde{\gamma}}}
  \leq C\|W_0\|_{C^{\widetilde{\gamma}}} \|\nabla\partial_{W_0}^{k-2}\widetilde{\theta}_0\|_{C^{\widetilde{\gamma}}} \\
  &\le C\|W_0\|_{C^{1+\widetilde{\gamma}}} \left(\|\nabla\partial_{W_0}^{k-3}\widetilde{\theta}_0\|_{C^{\widetilde{\gamma}}}
  +\|\nabla^2\partial_{W_0}^{k-3}\widetilde{\theta}_0\|_{C^{\widetilde{\gamma}}}\right) \\
  &\le C\|W_0\|_{C^{k-2+\widetilde{\gamma}}} \left(\|\nabla\widetilde{\theta}_0\|_{C^{\widetilde{\gamma}}}+\|\nabla^2\widetilde{\theta}_0\|_{C^{\widetilde{\gamma}}}
  +\cdots+\|\nabla^{k-1}\widetilde{\theta}_0\|_{C^{\widetilde{\gamma}}}\right) \\
  &\le C\|W_0\|_{C^{k-2+\widetilde{\gamma}}}\|\widetilde{\theta}_0\|_{C^{k-1+\widetilde{\gamma}}}.
\end{aligned}
\end{equation*}
Hence \eqref{eq:str-reg2} is proved and we finish the proof of Lemma \ref{lem:str-reg} for $\mathbf{D}=\RR^2$.

As for the case $\mathbf{D}=\TT^2$, we define $\underline{\theta}_0$ in $\RR^2$ as $\underline{\theta}_0=\theta_0$
in a whole period, and vanishes in others. Then we treat $\underline{\theta}_0$ in the whole space case, it yields
that $\partial_{W_0}^{k-1} \underline{\theta}_0 \in C^{\gamma-2\alpha+1}(\mathbb{R}^2)$
and $ \partial_{W_0}^{k-1} \underline{\theta}_0 \in L^{\infty}(\mathbb{R}^2)$, respectively,
which implies the desired results that are \eqref{eq:str-reg} and \eqref{eq:str-reg2} where $\mathbf{D} =\mathbb{T}^2$.
\end{proof}

\subsection{Auxiliary lemmas}\label{subsec:aux-lem}

We refer to \cite[Proposition 3.1]{HKR11} for the following lemma.
\begin{lemma}\label{lem:m(D)}
Let $m(D) = \Lambda^\sigma m_0(D)$, $\sigma>0$, and $m_0(D)$ be a zero-order pseudo-differential operator with
$m_0(\xi) \in C^{\infty}\left(\mathbb{R}^d \backslash\{0\}\right)$, then we have that for all $p\in [1,\infty], q\geq-1, j\geq0$,
\begin{equation}\label{es:mD-Sj-Lp}
\begin{aligned}
  &\|m(D)S_{j}u\|_{L^p}\leq C2^{j\sigma}\|S_{j}u\|_{L^p},\\
  &\|m(D)\Delta_qu\|_{L^p}\leq C2^{q\sigma}\|\Delta_qu\|_{L^p}.
\end{aligned}
\end{equation}
\end{lemma}


We have the following estimates of commutators involving the operator $\mathcal{R}_{1-\beta}$.
\begin{lemma}\label{lem:Rbeta-cm}
Let $\mathbf{D}$ be either $\mathbb{R}^d$ or torus $\mathbb{T}^d$, $d\geq 2$.
Let $(p, r) \in[2, \infty] \times[1, \infty]$, $\mathcal{R}_{1-\beta} := \partial_1 \Lambda^{-\beta}$,
$\beta \in ( 1,2 ]$.
Assume that $u=(u_1,\cdots,u_d)$ is a smooth divergence-free vector field on $\mathbf{D}$
and $\phi$ is a smooth scalar function on $\mathbf{D}$.
\begin{enumerate}
\item
If $\mathbf{D}=\mathbb{R}^d$ and $\beta\in (1,2)$, we have that for all $s \in (\beta-2, 1)$,
\begin{align}\label{eq:Rbeta-cm-es}
  \left\|\left[\mathcal{R}_{1-\beta}, u \cdot \nabla\right] \phi\right\|_{B_{p, r}^s}
  \leq C_{s, \beta}\|\nabla u\|_{L^p}\left(\|\phi\|_{B_{\infty, r}^{s+1-\beta}}+\|\phi\|_{L^2}\right) ,
\end{align}
and for all $s \in (0,1)$,
\begin{equation}\label{eq:R-1cm-es1}
  \|[\mathcal{R}_{-1}, u\cdot\nabla] \phi\|_{B^s_{p,r}}
  \leq C_s \Big( \|\nabla u\|_{L^p} \|\phi\|_{B^{s - 1}_{\infty,r}}
  + \|u\|_{L^2} \|\phi\|_{L^2} \Big).
\end{equation}
\item
If $\mathbf{D} =\mathbb{T}^d$,
we have that \eqref{eq:Rbeta-cm-es} also holds 
and for all $s\in (0,1)$,
\begin{align}\label{eq:R-1cm-es2}
  \|[\mathcal{R}_{-1},u\cdot\nabla]\phi\|_{B^s_{p,r}} \leq C_s \|\nabla u\|_{L^p}
  \Big(\|\phi\|_{B^{s-1}_{\infty,r}} + \|\phi\|_{L^2} \Big).
\end{align}
\end{enumerate}
\end{lemma}

\begin{proof}[Proof of Lemma \ref{lem:Rbeta-cm}]
  $\mathbf{(i)}$ We only sketch the proof of \eqref{eq:Rbeta-cm-es}.
The proof of \eqref{eq:R-1cm-es1} is analogous with that of \eqref{eq:R-1cm-es2} given below, thus we omit the details.
By using Bony's decomposition, we may write
\begin{align*}
  [\mathcal{R}_{1-\beta},u\cdot\nabla]\phi & = \sum_{q\in\mathbb{N}} [\mathcal{R}_{1-\beta},S_{q-1}u\cdot\nabla]\Delta_q\phi
  + \sum_{q\in\mathbb{N}} [\mathcal{R}_{1-\beta},\Delta_q u\cdot\nabla]S_{q-1}\phi
  + \sum_{q\geq -1}[\mathcal{R}_{1-\beta},\Delta_q u\cdot\nabla]\widetilde{\Delta}_q\phi  \\
  & := \mathrm{I}_\beta + \mathrm{II}_\beta + \mathrm{III}_\beta.
\end{align*}
The control of $\mathrm{I}_\beta$ and $\mathrm{III}_\beta$ is analogous to the proof of $\mathrm{I}$ and $\mathrm{III}$
in Proposition 4.2 of \cite{WX12}, so that for all $s>\beta -2$,
\begin{align*}
  \|\mathrm{I}_\beta\|_{B^s_{p,r}} + \|\mathrm{III}_\beta\|_{B^s_{p,r}}
  \leq C \|\nabla u\|_{L^p} \big( \|\phi\|_{B^{s+1-\beta}_{\infty,r}} +  \|\phi\|_{L^2} \big).
\end{align*}
For $\mathrm{II}_\beta$, the control of the corresponding term $\mathrm{II}$ in Proposition 4.2 of \cite{WX12}
contains some error,  instead we estimate it as follows
\begin{eqnarray*}
  2^{js}\|\Delta_j \mathrm{II}_\beta\|_{L^p}  &\leq& C 2^{js} \sum_{q\in\mathbb{N},|q-j|\leq 4} \Big(\| \mathcal{R}_{1-\beta} (\Delta_q u\cdot\nabla S_{q-1}\phi)\|_{L^p}
  + \|\Delta_q u \cdot\nabla \mathcal{R}_{1-\beta} S_{q-1}\phi\|_{L^p} \Big) \\
  &\leq& C 2^{js} \sum_{q\in\mathbb{N},|q-j|\leq 4} \|\Delta_q u\|_{L^p} \Big( 2^{q(1-\beta)}
  \|\nabla S_{q-1} \phi\|_{L^\infty} + \|\nabla \mathcal{R}_{1-\beta} S_{q-1}\phi\|_{L^\infty} \Big) \\
  &\leq& C \sum_{|q-j|\leq 4} 2^{q(s-1)} \|\Delta_q\nabla u\|_{L^p}
  \bigg(2^{q(1-\beta)} \sum_{-1\leq q'\leq q-1} 2^{q'} \|\Delta_{q'}\phi\|_{L^\infty} \\
  && \hspace{2cm} \ + \sum_{-1\leq q'\leq q-1} 2^{q' (2-\beta)} \|\Delta_{q'}\phi\|_{L^\infty} \bigg) \\
  &\leq& C \|\nabla u\|_{L^p} \sum_{|q-j|\leq 4}
  \sum_{-1\leq q' \leq q-1} 2^{(q'-q)(1-s)} 2^{q'(s+1-\beta)} \|\Delta_{q'}\phi\|_{L^\infty}
\end{eqnarray*}
where in the fourth line we have used \eqref{es:mD-Sj-Lp}.
Discrete Young's inequality yields that for all $s<1$,
\begin{align*}
  \|\mathrm{II}_\beta\|_{B^s_{p,r}} \leq C \|\nabla u\|_{L^p} \|\phi\|_{B^{s+1-\beta}_{\infty,r}}.
\end{align*}
Gathering the above estimates lead to \eqref{eq:Rbeta-cm-es} if $\mathbf{D}=\mathbb{R}^d$. \\

$\mathbf{(ii)}$ We only give the proof of \eqref{eq:R-1cm-es2}
and the estimate \eqref{eq:Rbeta-cm-es} as  the proof can be adapted in a similar manner in the torus.
Since $[\mathcal{R}_{-1}, u\cdot \nabla ] \phi=[\mathcal{R}_{-1}, \left(u-\widehat{u}(0)\right)\cdot \nabla ] \phi$,
we can assume $\widehat{u}(0) = \int_{\mathbb{T}^d} u \dd x =0$ without loss of generality. Then, using Bony's decomposition
\begin{align}
  [\mathcal{R}_{-1}, u\cdot \nabla ] \phi
  & = \sum_{q\in \NN}[\mathcal{R}_{-1}, S_{q-1}u\cdot\nabla] \Delta_q \phi
  + \sum_{q\in \NN} [\mathcal{R}_{-1}, \Delta_q u \cdot\nabla] S_{q-1}\phi
  + \sum_{q\geq -1} [\mathcal{R}_{-1}, \Delta_q u\cdot\nabla] \widetilde{\Delta}_q \phi \nonumber \\
  & := {I}_1 + {I}_2 + {I}_3, \nonumber
\end{align}
where we have used the notations introduced in Remark \ref{rem:Besov-torus}.
For ${I}_1$, noting that there exists a bump function $\widetilde{\psi}\in C^\infty_c(\mathbb{R}^d)$
supported on an annulus of $\mathbb{R}^d$ away from the origin
such that
$${I}_1 = \sum_{q\in \mathbb{N}} [\mathcal{R}_{-1}\widetilde{\psi}
(2^{-q}D), S_{q-1}u\cdot\nabla]\Delta_q \phi,$$
and using  the fact that
\begin{align*}
  \mathcal{R}_{-1} \widetilde{\psi}(2^{-q}D) f(x) = \int_{\mathbb{R}^d} h^*_q (y) f(x-y)\dd y =
  \int_{\mathbb{T}^d} \mathbf{h}^*_q(y) f(x-y) \dd y,\quad \mathbf{h}^*_q(y) := \sum_{m\in\mathbb{Z}^d}h^*_q(y+m),
\end{align*}
with $h_q^*(y) := 2^{q(d-1)} h^*(2^q y)$ and $h^* = \mathcal{F}^{-1}\big(i \xi_1 |\xi|^{-2}\widetilde{\psi} \big)
\in \mathcal{S}(\mathbb{R}^d)$, we infer that
\begin{align*}
  {I}_1 & = \int_{\mathbb{T}^d} \mathbf{h}^*_q(y) \big(S_{q-1}u(x-y)
  - S_{q-1}u(x) \big)  \cdot\nabla \Delta_q\phi(x-y) \dd y  \\
  & = \int_{\mathbb{T}^d} \bigg(\sum_{m\in\mathbb{Z}^d}h_q^*(y+m)
  \big( S_{q-1}u(x-y-m) - S_{q-1}u(x)\big) \bigg)\cdot \nabla
  \Delta_q\phi(x-y) \dd y \\
  & = \int_{\mathbb{T}^d} \int_0^1 \bigg( \sum_{m\in\mathbb{Z}^d} h_q^*(y+m) (-y-m)\cdot
  \nabla S_{q-1}u(x-\tau y-\tau m) \bigg) \cdot \nabla \Delta_q\phi(x-y) \dd \tau \dd y.
\end{align*}
Thus, by using the Minkowski inequality  we find that for all $j\geq -1$,
\begin{align*}
  2^{js} \|\Delta_j {I}_1\|_{L^p} & \leq C 2^{js} \sum_{q\in\mathbb{N},|q-j|\leq 4}
  \|[\mathcal{R}_{-1}\widetilde{\psi}(2^{-q}D), S_{q-1}u\cdot\nabla]\Delta_q\phi\|_{L^p} \\
  & \leq C 2^{js} \sum_{|q-j|\leq 4} \int_{\mathbb{T}^d}
  \bigg(\sum_{m\in\mathbb{Z}^d} |h_q^*(y+m)| |y+m|\bigg) \dd y
  \|\nabla S_{q-1} u\|_{L^p} \|\nabla \Delta_q\phi\|_{L^\infty} \\
  & \leq C\|\nabla u\|_{L^p}
  \sum_{|q-j|\leq 4} 2^{q(1+s)} \|\Delta_q \phi\|_{L^\infty} \int_{\mathbb{R}^d} 2^{q(d-1)}|h^*(2^q y)| |y| \dd y \\
  & \leq C\|\nabla u\|_{L^p}
  \sum_{|q-j|\leq 4} 2^{q(s-1)} \|\Delta_q \phi\|_{L^\infty} ,
\end{align*}
which ensures that
\begin{align*}
  \|{I}_1\|_{B^s_{p,r}} \leq  C \|\nabla u\|_{L^p} \|\phi\|_{B^{s-1}_{\infty,r}}.
\end{align*}
For ${I}_2$, noting that
\begin{align*}
  {I}_2 = \sum_{q\in\mathbb{N}}\Big( \mathcal{R}_{-1} \widetilde{\psi}(2^{-q} D)\big(
  \Delta_q u \cdot\nabla S_{q-1}\phi\big) -\Delta_q u \cdot\nabla \mathcal{R}_{-1} S_{q-1}\phi \Big),
\end{align*}
we obtain
\begin{align*}
  2^{js} \|\Delta_j {I}_2\|_{L^p}
  & \leq C 2^{js} \sum_{q\in\mathbb{N},|q-j|\leq 4}
  \Big(\|\mathbf{h}^*_q * \big(\Delta_q u\cdot\nabla S_{q-1}\phi \big)\|_{L^p}
  + \|\Delta_q u\|_{L^p} \|\nabla \mathcal{R}_{-1}S_{q-1}\phi\|_{L^\infty} \Big) \\
  & \leq C 2^{js}\sum_{q\in\mathbb{N},|q-j|\leq 4} \|\Delta_q u\|_{L^p} \bigg(2^{-q} \|\nabla S_{q-1}\phi\|_{L^\infty}
  + \|\nabla \mathcal{R}_{-1}\Delta_{-1}\phi\|_{L^\infty}
  + \sum_{0\leq q'\leq q-1} \|\Delta_{q'} \phi\|_{L^\infty} \bigg) \\
  & \leq C \|\nabla u\|_{L^p} \sum_{|q-j|\leq 4} \bigg(2^{q(s-1)}\|\Delta_{-1}\phi\|_{L^2} + \sum_{0\leq q' \leq q-1}
  2^{(q-q')(s-1)} 2^{q'(s-1)} \|\Delta_{q'} \phi\|_{L^\infty} \bigg),
\end{align*}
which leads to that for all $s<1$,
\begin{align*}
  \|{I}_2\|_{B^s_{p,r}} \leq C \|\nabla u\|_{L^p} \big( \|\phi\|_{B^{s-1}_{\infty,r}} + \|\phi\|_{L^2} \big).
\end{align*}
For ${I}_3$, using the fact that $u$ is divergence-free, we split it as the following
\begin{align*}
  {I}_3 & = \sum_{q\geq 3} \mathcal{R}_{-1}\nabla\cdot\big(\Delta_q u\,\widetilde{\Delta}_q\phi \big)
  - \sum_{q\geq 3}\Delta_q u\cdot\nabla \mathcal{R}_{-1}\widetilde{\Delta}_q \phi
  + \sum_{-1\leq q\leq 2} [\mathcal{R}_{-1} \nabla\cdot, \Delta_q u] \widetilde{\Delta}_q \phi \\
  & := {I}_{3,1} + {I}_{3,2}  + {I}_{3,3}.
\end{align*}
For the first term ${I}_{3,1}$, since the operator $\mathcal{R}_{-1}\nabla$ is bounded on $L^p(\mathbb{T}^d)$
for $p\in[2,\infty)$
and $\mathcal{R}_{-1}\nabla \Delta_j$ ($j\in\mathbb{N}$) is bounded on $L^p(\mathbb{T}^d)$ for $p\in [2,\infty]$,
we infer that for $j= -1$ (where $\frac{2p}{p+2} = 2$ for $p=\infty$),
\begin{align*}
  2^{-s} \|\Delta_{-1} {I}_{3,1}\|_{L^\infty}
  \leq C \sum_{q\geq 3} \|\Delta_q u\,\widetilde{\Delta}_q\phi \|_{L^{\frac{2p}{p+2}}}
  \leq C \sum_{q\geq 3} \|\Delta_q u\|_{L^p} \|\widetilde{\Delta}_q \phi\|_{L^2}
  \leq C \|\nabla u\|_{L^p} \|\phi\|_{L^2},
\end{align*}
and for all $j\in\mathbb{N}$ and $s>0$,
\begin{align*}
  2^{js} \|\Delta_j {I}_{3,1}\|_{L^p}
  & \leq C 2^{js} \sum_{q\geq 3,q\geq j-3} \|\Delta_j \mathcal{R}_{-1}\nabla\cdot(\Delta_q u\,\widetilde{\Delta}_q \phi) \|_{L^p} \\
  & \leq C \sum_{q\geq 3} 2^{(j-q)s} 2^q\|\Delta_q u\|_{L^p} 2^{q(s-1)}\|\widetilde{\Delta}_q \phi\|_{L^\infty}
  \leq C c_j \|\nabla u \|_{L^p} \|\phi\|_{B^{s-1}_{\infty,r}},
\end{align*}
where $\{c_j\}_{j\in\mathbb{N}}$ is such that $\|c_j\|_{\ell^r} = 1$.
The estimation of ${I}_{3,2}$ is similar as that of ${I}_{3,1}$, and we have
\begin{align*}
  \|{I}_{3,2}\|_{B^s_{p,r}} \leq C \|\nabla u\|_{L^p} \big(\|\phi\|_{B^{s-1}_{\infty,r}} + \|\phi\|_{L^2} \big).
\end{align*}
For ${I}_{3,3}$, we use Poincar\'e's inequality,
\begin{align*}
  \|u\|_{L^p(\mathbb{T}^d)} = \Big\|u(\cdot) - \frac{1}{|\mathbb{T}^d|} \int_{\mathbb{T}^d}u(x)\dd x\Big\|_{L^p(\mathbb{T}^d)}
  \leq C \|\nabla u\|_{L^p(\mathbb{T}^d)},
\end{align*}
we find
\begin{align*}
  \|{I}_{3,3}\|_{B^s_{p,r}} & \leq C \sum_{-1\leq j\leq 6} \sum_{-1\leq q\leq 2}
  \Big( \|\Delta_j \mathcal{R}_{-1}\nabla\cdot (\Delta_q u\,\widetilde{\Delta}_q \phi)\|_{L^p}
  + \|\Delta_j (\Delta_q u \cdot\nabla \mathcal{R}_{-1} \widetilde{\Delta}_q \phi)\|_{L^p} \Big) \\
  & \leq C \sum_{-1\leq q\leq 2} \Big( \|\Delta_q u\,\widetilde{\Delta}_q\phi\|_{L^{\frac{2p}{p+2}}}
  + \|\Delta_q u\cdot \nabla \mathcal{R}_{-1}\widetilde{\Delta}_q \phi\|_{L^{\frac{2p}{p+2}}} \Big) \\
  & \leq C \sum_{-1\leq q\leq 2} \|\Delta_q u\|_{L^p} \|\widetilde{\Delta}_q \phi\|_{L^2}
  \leq C \|\nabla u\|_{L^p}  \|\phi\|_{L^2}.
\end{align*}
Therefore, collecting the above estimates yields the wanted estimate \eqref{eq:R-1cm-es2} in case of the torus.
\end{proof}

We refer to Lemma 6.10 of \cite{HKR10} for the following useful result.
\begin{lemma}\label{lem:HKR}
Let $\mathbf{D}$ be either $\mathbb{R}^d$ or torus $\mathbb{T}^d$, $d\geq 2$. Let $v$ be a smooth divergence-free vector field of $\mathbf{D}$ and $f$ be a smooth scalar function. Then, for all $p\in [1,\infty]$ and $q\geq -1$,
\begin{align*}
  \|[\Delta_q, v\cdot\nabla]f\|_{L^p}\le C \|\nabla v\|_{L^p}\|f\|_{B^{0}_{\infty,\infty}}.
\end{align*}
\end{lemma}

We recall the following regularity estimates of the transport equation (one can see \cite{BCD11} for a detailed proof).
\begin{lemma}\label{lem:CMX}
Let $\mathbf{D}$ be either $\mathbb{R}^d$ or torus $\mathbb{T}^d$, $d\geq 2$. Let $( p, r) \in[1, \infty]^2$ and $-1<s<1$. Assume that $u$ is a smooth divergence-free vector field of $\mathbf{D}$, and $\phi$ is a smooth function solving the transport equation
\begin{align*}
  \partial_t \phi+u \cdot \nabla \phi =f,\left.\quad \phi\right|_{t=0}(x)=\phi_0(x), \quad x \in \mathbb{R}^d,
\end{align*}
then there exists a constant $C=C(d, s)$ so that for all $t>0$,
\begin{equation}\label{eq:T-BesEs0}
  \|\phi\|_{L_t^{\infty}\left(B_{p, r}^s\right)} \leq C\left(\left\|\phi_0\right\|_{B_{p, r}^s}+\|f\|_{\widetilde{L}_t^1\left(B_{p, r}^s\right)}+\int_0^t\|\nabla u(\tau)\|_{L^{\infty}}\|\phi(\tau)\|_{B_{p, r}^s} \mathrm{~d} \tau\right),
\end{equation}
and
\begin{equation}\label{eq:T-BesEs}
  \|\phi\|_{L_t^{\infty}\left(B_{p, r}^s\right)} \leq C e^{C \int_0^t\|\nabla u\|_{L^{\infty}} d \tau}\left(\left\|\phi_0\right\|_{B_{p, r}^s}+\|f\|_{\widetilde{L}_t^1\left(B_{p, r}^s\right)}\right).
\end{equation}
\end{lemma}

We have the following regularity estimates of the transport-diffusion equation.
\begin{lemma}\label{lem:TD-sm}
Let $\mathbf{D}$ be either $\mathbb{R}^d$ or torus $\mathbb{T}^d$, $d\geq 2$. Let $\nu >0$, $(\rho, p, r) \in[1, \infty]^3$, $-1<s<1$, $0<\alpha\leq 1$.  Assume that $u$ is a smooth divergence-free vector field of $\mathbf{D}$, and $\phi$ is a smooth function solving the transport-diffusion equation
\begin{equation}\label{eq:TD}
  \partial_t \phi + u \cdot \nabla \phi + \nu \Lambda^{2\alpha}\phi =f,
  \left.\quad \phi\right|_{t=0}(x)=\phi_0(x), \quad x \in \mathbf{D}.
\end{equation}
Then, there exists a constant $C = C(d, s, \alpha)$ independent of $\nu$ so that for all $t>0$,
\begin{equation}\label{eq:TD-sm1}
\begin{aligned}
  \nu^{\frac{1}{\rho}} \|\phi\|_{\widetilde{L}_t^\rho\big(\dot B_{p, r}^{s+\frac{2\alpha}{\rho}}\big)}
  \leq C e^{C\int_0^t\|\nabla u(\tau)\|_{L^{\infty}} \dd \tau}
  \left(\|\phi_0\|_{\dot B^s_{p,r}} + \|f\|_{\widetilde{L}_t^1 (\dot B^s_{p,r})}\right),
\end{aligned}
\end{equation}
and
\begin{equation}\label{eq:TD-sm2}
\begin{aligned}
  \nu^{\frac{1}{\rho}} \|\phi\|_{\widetilde{L}_t^\rho\big(B_{p, r}^{s+\frac{2\alpha}{\rho}}\big)}
  \leq C(1+\nu t)^{\frac{1}{\rho}} e^{C \int_0^t\|\nabla u(\tau)\|_{L^{\infty}} \dd \tau}
  \left(\|\phi_0\|_{B^s_{p,r}}+\|f\|_{L_t^1(B^s_{p,r})}\right),
\end{aligned}
\end{equation}
and
\begin{equation}\label{eq:TD-sm4}
\begin{aligned}
  \nu^{\frac{1}{\rho}} \Big\|\Big(2^{q(s+\frac{2\alpha}{\rho})}\|\Delta_q\phi\|_{L_t^\rho(L^p)}\Big)_{q\in \mathbb{N}}\Big\|_{\ell^r}
  \leq C e^{C \int_0^t\|\nabla u(\tau)\|_{L^{\infty}} \dd \tau}
  \left(\|\phi_0\|_{B^s_{p,r}}+\|f\|_{L_t^1(B^s_{p,r})}\right).
\end{aligned}
\end{equation}

\end{lemma}

\begin{proof}[Proof of Lemma \ref{lem:TD-sm}]
  The smoothing estimate \eqref{eq:TD-sm1}   comes from \cite[Theorem 1.2]{MW09} (while for  the cases $\alpha\in (0,\frac{1}{2}]$ and $r=1$
it is done in \cite[Theorem 1.2]{HK07}).

We here only sketch the proof of \eqref{eq:TD-sm2} and \eqref{eq:TD-sm4}. For all $q\in \mathbb{N}$,
applying the frequency-localization operator $\Delta_q$ to the equation \eqref{eq:TD} gives
\begin{align*}
  \partial_t \Delta_q \phi + S_{q-1} u \cdot \nabla  \Delta_q \phi + \nu \Lambda^{2\alpha} \Delta_q \phi
  = \Delta_q f + R_q,
\end{align*}
with
\begin{align*}
  R_q := (S_{q-1} u - u)\cdot \nabla \Delta_q \phi - [\Delta_q, u\cdot\nabla] \phi .
\end{align*} Using \cite[Lemma 2.100]{BCD11},
we get
\begin{align*}
  2^{q s}\|R_q\|_{L^p} & \leq 2^{qs} \|(S_{q-1}u - u)\cdot \nabla \Delta_q \phi\|_{L^p} + 2^{q s} \|[\Delta_q,u\cdot\nabla]\phi\|_{L^p} \\
  & \leq C 2^{qs} \sum_{k\geq q-1} \|\Delta_k u\|_{L^\infty} 2^q \|\Delta_q \phi\|_{L^p} + C \,c_q \|\nabla u\|_{L^\infty} \|\phi\|_{B^s_{p,r}} \\
  & \leq C\, c_q \|\nabla u\|_{L^\infty} \|\phi\|_{B^s_{p,r}},
\end{align*}
where $\{c_q\}_{q\in \mathbb{N}}$ satisfies $\|c_q\|_{\ell^r} \leq 1$. Following the same strategy  as the proof of \cite[Eq. (3.10)]{MW09} and using the above estimate on $R_q$, we find that
there exists a small time $T'>0$  satisfying
\begin{align*}
  \int_0^{T'} \|\nabla u\|_{L^\infty}\dd \tau \leq C_0 \ll 1,
\end{align*}
so that for all $t\leq T'$,
\begin{align*}
  \nu^{\frac{1}{\rho}} 2^{q(s + \frac{2\alpha}{\rho})} \|\Delta_q \phi\|_{L^\rho_t (L^p)}
  \leq C \bigg( 2^{qs} \|\Delta_q \phi_0\|_{L^p} + 2^{q s} \|\Delta_q f\|_{L^1_t (L^p)}
  + \int_0^t c_q(\tau) \|\nabla u(\tau)\|_{L^\infty} \|\phi(\tau)\|_{B^s_{p,r}} \dd \tau\bigg).
\end{align*}
Thus using the embedding $L_t^1(B^s_{p,r})\hookrightarrow \widetilde{L}_t^1(B^s_{p,r})$, we have
\begin{align*}
  \nu^{\frac{1}{\rho}}\Big\|\Big( 2^{q(s+\frac{2\alpha}{\rho})} \|\Delta_q\phi\|_{L^\rho_t (L^p)}\Big)_{q\in \mathbb{N}}
  \Big\|_{\ell^r}
  \leq C \bigg(\|\phi_0\|_{B^s_{p,r}} + \|f\|_{L^1_t(B^s_{p,r})}
  + \int_0^t \|\nabla u\|_{L^\infty} \|\phi\|_{B^s_{p,r}} \dd \tau \bigg).
\end{align*}
On the other hand, applying the operator $\Delta_{-1}$ to the equation \eqref{eq:TD} gives
\begin{align*}
  \partial_t \Delta_{-1}\phi + u\cdot \nabla \Delta_{-1} \phi + \nu \Lambda^{2\alpha} \Delta_{-1}\phi
  = \Delta_{-1} f - [\Delta_{-1}, u\cdot\nabla]\phi.
\end{align*}
Taking advantage of the $L^p$-estimate of the transport-diffusion equation (see e.g. \cite[Proposition 6.2]{HK07}),
we infer that
\begin{align*}
  \|\Delta_{-1}\phi(t)\|_{L^p} & \leq \|\Delta_{-1}\phi_0\|_{L^p}
  + \|\Delta_{-1}f\|_{L^1_t(L^p)} + \int_0^t \|[\Delta_{-1}, u\cdot\nabla ]\phi(\tau)\|_{L^p}\dd \tau \\
  & \leq \|\Delta_{-1}\phi_0\|_{L^p}
  + \|\Delta_{-1}f\|_{L^1_t(L^p)} + C \int_0^t \|\nabla u(\tau)\|_{L^\infty} \|\phi(\tau)\|_{B^s_{p,r}}\dd \tau,
\end{align*}
where in the last line we have used \cite[Lemma 2.100]{BCD11} to deal with the commutator term.
Hence, combining the above estimates on high and low frequencies allow us to find that for all $t\leq T'$,
\begin{align*}
  \nu^{\frac{1}{\rho}}\|\phi\|_{\widetilde{L}^\rho_t (B^{s + \frac{2\alpha}{\rho}}_{p,r})}
  \leq C  (1+\nu t)^{\frac{1}{\rho}}\bigg(\|\phi_0\|_{B^s_{p,r}} + \|f\|_{L^1_t(B^s_{p,r})}
  + \int_0^t \|\nabla u(\tau)\|_{L^\infty} \|\phi(\tau)\|_{B^s_{p,r}} \dd \tau \bigg).
\end{align*}
By taking $\rho =\infty$, we find that for all $t\leq T'$ small enough,
\begin{align}\label{eq:phi-sminfty}
  \|\phi\|_{\widetilde{L}^\infty_t (B^{s}_{p,r})}
  \leq C \big(\|\phi_0\|_{B^s_{p,r}} + \|f\|_{L^1_t(B^s_{p,r})}\big),
\end{align}
which also implies
\begin{align}\label{eq:phi-smBes}
  \nu^{\frac{1}{\rho}} \|\phi\|_{\widetilde{L}^\rho_t(B^{s + \frac{2\alpha}{\rho}}_{p,r})}
  \leq C (1+\nu t)^{\frac{1}{\rho}} \Big(\|\phi_0\|_{B^s_{p,r}} + \|f\|_{L^1_t(B^s_{p,r})}\Big) ,
\end{align}
and
\begin{align}\label{eq:high-fre2}
  \nu^{\frac{1}{\rho}}\Big\|\Big( 2^{q(s+\frac{2\alpha}{\rho})} \|\Delta_q\phi\|_{L^\rho_t (L^p)}\Big)_{q\in \mathbb{N}}
  \Big\|_{\ell^r}
  \leq C \Big(\|\phi_0\|_{B^s_{p,r}} + \|f\|_{L^1_t(B^s_{p,r})}\Big).
\end{align}

Furthermore, for any $T>0$, we make a partition $\{T_i\}_{i=0}^M$ of the time interval $[0,T]$ so that
$\int_{T_i}^{T_{i+1}} \|\nabla u(\tau)\|_{L^\infty} \dd \tau \approx \frac{C_0}{2}$.
Then following the same ideas as the proof of \eqref{eq:phi-sminfty}, \eqref{eq:phi-smBes} and \eqref{eq:high-fre2}, we infer that
\begin{align}\label{eq:phi-sminfty2}
  \|\phi\|_{\widetilde{L}^\infty([T_i,T_{i+1}]; B^{s}_{p,r})}
  \leq C \Big(\|\phi(T_i)\|_{B^s_{p,r}} + \|f\|_{L^1([T_i,T_{i+1}]; B^s_{p,r})}\Big),
\end{align}
\begin{align*}
  \nu^{\frac{1}{\rho}}\|\phi\|_{\widetilde{L}^\rho([T_i,T_{i+1}]; B^{s + \frac{2\alpha}{\rho}}_{p,r})}
  \leq C(1+\nu T)^{\frac{1}{\rho}} \Big( \|\phi(T_i)\|_{B^s_{p,r}} + \|f\|_{L^1([T_i,T_{i+1}]; B^s_{p,r})}\Big),
\end{align*}
and
\begin{align*}
  \nu^{\frac{1}{\rho}}\Big\|\Big( 2^{q(s+\frac{2\alpha}{\rho})} \|\Delta_q\phi\|_{L^\rho([T_i,T_{i+1}]; L^p)}\Big)_{q\in \mathbb{N}}
  \Big\|_{\ell^r}
  \leq & C \Big(\|\phi(T_i)\|_{B^s_{p,r}} + \|f\|_{L^1([T_i,T_{i+1}]; B^s_{p,r})} \Big).
\end{align*}
By iterating \eqref{eq:phi-sminfty2}  $M$-times and using the fact
$M\approx \frac{1}{C_0} \int_0^T\|\nabla u\|_{L^\infty}\dd \tau$,
we deduce that
\begin{align*}
  \|\phi\|_{\widetilde{L}^\infty_T( B^{s}_{p,r})}
  & \leq C\bigg( \sum_{i=0}^{M-1}\|\phi(T_i)\|_{B^s_{p,r}} + \|f\|_{L^1_T (B^s_{p,r})}\bigg) \\
  & \leq C M \Big(\|\phi_0\|_{B^s_{p,r}} + \|f\|_{L^1_T(B^s_{p,r})} \Big) C^M
  + C  \|f\|_{L^1_T(B^s_{p,r})} \\
  & \leq C e^{C \int_0^T\|\nabla u\|_{L^\infty}\dd \tau}
  \Big(\|\phi_0\|_{B^s_{p,r}} + \|f\|_{L^1_T(B^s_{p,r})} \Big) ,
\end{align*}
which also yields
\begin{align*}
  \nu^{\frac{1}{\rho}}\|\phi\|_{\widetilde{L}^\rho_T(B^{s+\frac{2\alpha}{\rho}}_{p,r})}
  & \leq C(1+\nu T)^{\frac{1}{\rho}} \bigg( \sum_{i=0}^{M-1} \|\phi(T_i)\|_{B^s_{p,r}}
  + \|f\|_{L^1([0,T]; B^s_{p,r})}\bigg) \\
  & \leq C (1+\nu T)^{\frac{1}{\rho}} e^{C \int_0^T\|\nabla u\|_{L^\infty}\dd \tau}
  \Big(\|\phi_0\|_{B^s_{p,r}} + \|f\|_{L^1_T(B^s_{p,r})} \Big) ,
\end{align*}
and
\begin{align*}
  \nu^{\frac{1}{\rho}}\Big\|\Big( 2^{q(s+\frac{2\alpha}{\rho})} \|\Delta_q\phi\|_{L_T^\rho( L^p)}
  \Big)_{q\in \mathbb{N}}   \Big\|_{\ell^r}
  & \leq C e^{C \int_0^T\|\nabla u\|_{L^\infty}\dd \tau}  \Big(\|\phi_0\|_{B^s_{p,r}}
  + \|f\|_{L^1_T(B^s_{p,r})} \Big).
\end{align*}
Therefore, we complete the proof of \eqref{eq:TD-sm2}, \eqref{eq:TD-sm4}.
\end{proof}

The following compactness lemma plays an important role in the process of vanishing Prandtl number limit.
\begin{lemma}[Aubin-Lions lemma \cite{L1969}]\label{lem:Aubin-Lions}
Assume that $X_0\subseteq X\subseteq X_1$ are Banach spaces, and $X_0$ is compactly embedded in $X$ and
$X $ is continuously embedded in $X_1$. For all $1\leq p,q\leq \infty$, let
\begin{align*}
  V := \left\{u\in L_T^p(X_0)\;:\; \partial_tu\in L_T^q(X_1) \right\}.
\end{align*}
Then we have
\begin{enumerate}
\item if $p<\infty$, then $V$ is compactly embedded in $L_T^p(X)$;
\item if $p=\infty$ and $q>1$, then $V$ is compactly embedded in $C([0,T];X)$.
\end{enumerate}
\end{lemma}

%
%

\section{Persistence of the $C^{2+\gamma}$ boundary regularity}\label{sec:C2gamma}
This section is devoted to the proof of  the persistence of the $C^{2+\gamma}$ regularity of the temperature patch boundary
$\partial D(t)$ for some $0<\gamma\le2\alpha-1$ if $\alpha\in(\frac{1}{2},1)$ and for some $\gamma\in(0,1)$ if $\alpha=1$. In particular, we shall explicitly show the dependence of the \textit{a priori}
estimates on the coefficient $\varepsilon = \frac{1}{\mathrm{Pr}}$.

The $L^2$ energy estimate for the system $(\mathrm{B}_\alpha)$ is more or less classical
(e.g. see Proposition 4.1 of \cite{KX22} without the dependence on the Prandtl number),
and we have the following result.
\begin{proposition}\label{lem:eneEs}
Let $\varepsilon =\frac{1}{\mathrm{Pr}} \in (0,1]$, $\alpha\in (0,1]$. Let $\mathbf{D}$ be either $\mathbb{R}^2$
or $\mathbb{T}^2$. Suppose that $(u,\theta)$ is a smooth solution of 2D Boussinesq-Navier-Stokes system
$(\mathrm{B}_\alpha)$
with initial data $u_0\in L^2(\mathbf{D})$ and $\theta_0\in L^2\cap L^p(\mathbf{D})$, $p\in [1,\infty]$.
Then there exists an absolute number $C_0>0$ such that for all $t>0$,
\begin{align}
  \|\theta(t)\|_{L^p(\mathbf{D})}& = \|\theta_0\|_{L^p(\mathbf{D})}, \label{eq:the-MP} \\
  \|u(t)\|_{L^2(\mathbf{D})} + \frac{1}{\sqrt{\varepsilon}} \|\Lambda^\alpha u\|_{L^2_t\big(L^2(\mathbf{D})\big)}
  & \leq C_0 \big(\|u_0\|_{L^2} + \|\theta_0\|_{L^2} \big)
  \Big(\frac{t}{\varepsilon}+1 \Big) . \label{eq:u-L2-es}
\end{align}
\end{proposition}

\begin{proof}[Proof of Proposition \ref{lem:eneEs}]
  The conservation of the $L^p$ norm of $\theta$ in \eqref{eq:the-MP} follows from the transport equation.
By taking the dot product of the equation for $u$ in \eqref{BoussEq} with $u$ itself, and using an integration by parts, we get
\begin{align}\label{eq:u-eneEs1}
  \frac{\varepsilon}{2}\frac{\dd }{\dd t} \|u(t)\|_{L^2}^2 + \|\Lambda^\alpha u(t)\|_{L^2}^2
  \leq \Big|\int_{\mathbf{D}} \theta\, u_2 (x,t) \dd x \Big| \leq \|\theta_0\|_{L^2} \|u(t)\|_{L^2},
\end{align}
It follows that $\frac{\dd}{\dd t} \|u(t)\|_{L^2} \leq \frac{1}{\varepsilon} \|\theta_0\|_{L^2}$ and
\begin{align*}
  \|u(t)\|_{L^2(\mathbf{D})} \leq \|u_0\|_{L^2(\mathbf{D})} + \frac{t}{\varepsilon} \|\theta_0\|_{L^2(\mathbf{D})}.
\end{align*}
Inserting the above inequality into \eqref{eq:u-eneEs1} and integrating with respect to the time variable lead to
\begin{align*}
  \frac{1}{2}\|u(t)\|_{L^2}^2 + \frac{1}{\varepsilon} \|\Lambda^\alpha u\|_{L^2_t(L^2)}^2
  \leq \frac{1}{2}\|u_0\|_{L^2}^2 + \frac{t}{\varepsilon}\|\theta_0\|_{L^2} \Big(\|u_0\|_{L^2} + \frac{t}{\varepsilon} \|\theta_0\|_{L^2} \Big),
\end{align*}
which readily implies the desired bound \eqref{eq:u-L2-es}.
\end{proof}

The following result concerns the \textit{a priori} estimates for $u$ and
$\Gamma = \omega - \mathcal{R}_{1-2\alpha}\theta = \omega-\partial_1 \Lambda^{-2\alpha}\theta$
solving the system $(\mathrm{B}_\alpha)$ with $u_0\in H^1$, $\theta_0\in L^1\cap L^\infty$.
Note that Theorems 1.2, 5.2 and Proposition 5.4 of \cite{KX22} have provided similar estimate as \eqref{eq:u-Gam-es2},
but in order to clarify the dependence on the parameter $\varepsilon$, we here sketch the proof by using a slightly different argument.
\begin{proposition}\label{prop:ap-es1}
Let $\varepsilon = \frac{1}{\mathrm{Pr}}\in (0,1]$,
$\alpha\in(\frac{1}{2},1]$. Let $\mathbf{D} $ be either $\mathbb{R}^2$ or $\mathbb{T}^2$.
Suppose that $(u,\theta)$ is a smooth solution of 2D Boussinesq-Navier-Stokes equations $(\mathrm{B}_\alpha)$ with
\begin{itemize}
\item $u_0\in H^1(\mathbf{D})$,
\item $\nabla\cdot u_0=0$,
\item $\theta_0\in L^1\cap L^\infty(\mathbf{D})$.
\end{itemize}
Then, there exists a constant $C>0$ depending only on $\alpha$ and the norms of $(u_0,\theta_0)$ but independent of $\varepsilon$ such that for all $T>0$,  such that the following statements hold true.
\begin{enumerate}[(1)]
\item
If $\alpha\in (\frac{1}{2},1)$, we have
\begin{align}\label{eq:u-Gam-es2}
  \|(\nabla u,\Gamma)\|_{L^\infty_T(L^2)} + \|\Gamma\|_{L^1_T(B^1_{2,1})}
  + \left\|\Gamma\right\|_{\widetilde{L}_T^1 (B_{2, \infty}^{2\alpha})}
  + \|\nabla u\|_{\widetilde{L}^1_T(C^{2\alpha-1})} + \|\nabla u\|_{L^1_T(L^\infty)}
  \leq C e^{C T} .
\end{align}

\item If $\mathbf{D} = \mathbb{R}^2$, $\alpha=1$,
we have that for all $\gamma\in(0,1)$,
\begin{align}\label{eq:u-Gam-es3}
  \|\nabla u\|_{L^\infty_T(L^2)} + \|\nabla u\|_{L^1_T(L^\infty)} + \|\nabla u\|_{L^1_T(C^\gamma)}
  \leq C (1+T^3)\Big(1+ \sqrt{\tfrac{T}{\varepsilon}}\Big).
\end{align}
\item
If $\mathbf{D} = \mathbb{T}^2$, $\alpha=1$,
we have that for all $\gamma\in(0,1)$,
\begin{align}\label{eq:u-Gam-es4}
  \|(\nabla u,\Gamma)\|_{L^\infty_T(L^2)} + \|\Gamma\|_{L^1_T(B^1_{2,1})}
  + \left\|\Gamma\right\|_{\widetilde{L}_T^1 (B_{2, \infty}^2)}
  + \|\nabla u\|_{L^1_T(C^\gamma)}  \leq C e^{C T} .
\end{align}
\end{enumerate}
\end{proposition}

\begin{remark}\label{rem:apes1}
  Under the assumptions of Proposition \ref{prop:ap-es1},
we can show the global existence and uniqueness of solution $(\theta,u)$ to the
2D Boussinesq-Navier-Stokes system $(\mathrm{B}_\alpha)$
with $\varepsilon =\frac{1}{\mathrm{Pr}} \in (0,1]$ and $\alpha\in (\frac{1}{2},1]$,
which satisfies that for all $T>0$,
\begin{align}\label{eq:u-the-apri-H1}
  \theta\in L^\infty([0,T], L^1\cap L^\infty(\mathbf{D})), \quad
  u \in L^\infty([0,T],H^1(\mathbf{D})) \cap L^1([0,T], B^1_{\infty,1}(\mathbf{D})).
\end{align}
In particular, if $\theta_0$ is the temperature front initial data $\bar{\theta}_0\, \mathbf{1}_{D_0} $
with $\bar{\theta}_0\in L^\infty(\overline{D_0})$ and $D_0 \subset \mathbf{D}$
a bounded simply connected domain satisfying $\partial D_0 \in C^{1+\gamma}$, $\gamma\in (0,2\alpha-1)$,
then we have that \eqref{eq:the-exp}-\eqref{eq:reg-parD(t)} hold with $k=1$ and the same scope of $\gamma$.

Indeed, the existence part follows from a standard approximation process and the \textit{a priori}
estimates established in Propositions \ref{lem:eneEs} and \ref{prop:ap-es1}
(note that the fact that $u$ is controlled in $L^1_T(B^1_{\infty,1})$ in  \eqref{eq:u-the-apri-H1} can be easily obtained).
As far as the proof of the uniqueness is concerned, we refer to \cite{GGJ17} for the case $\alpha=1$
and to \cite{KX22} for the case $\alpha\in (\frac{1}{2},1)$.
Regarding the temperature patches, since $u$ is a divergence-free vector field which belongs to $L^1([0,T], W^{1,\infty}(\mathbf{D}))$ then using Cauchy-Lipschitz theorem, there exists a unique solution $X_t(\cdot) :\mathbf{D}\rightarrow \mathbf{D}$
to \eqref{eq:X} which is a bi-Lipschitzian measure-preserving homeomorphism,
thus the transport equation of $\theta$ ensures \eqref{eq:the-exp} holds,
and in combination with \eqref{eq:Xt-C1gam} and \eqref{eq:u-Gam-es2}, \eqref{eq:u-Gam-es4},
we have $\nabla X_t^{\pm 1} \in L^\infty_T(C^\gamma)$, $\gamma\in(0,2\alpha-1)$, which implies $\partial D(t) \in L^\infty_T(C^{1+\gamma})$.
\end{remark}

\begin{proof}[Proof of Proposition \ref{prop:ap-es1}]
Let us prove the first statement that is the point $(\textbf{1})$.
The proof of $(\textbf{1})$ works in either the whole space or the torus.
By computing the $L^2$ scalar product of the equation in $\Gamma$ \eqref{eq:Gamma} with $\Gamma$,
and using the commutator estimate \eqref{eq:Rbeta-cm-es} and the relation $\omega = \Gamma + \mathcal{R}_{1-2\alpha}\theta$,
we get
\begin{align*}
  \frac{1}{2} \frac{\dd }{\dd t} \|\Gamma(t)\|_{L^2}^2 + \frac{1}{\varepsilon} \|\Lambda^\alpha \Gamma(t)\|_{L^2}^2
  & \leq \|[\mathcal{R}_{1-2\alpha}, u\cdot \nabla]\theta\|_{L^2} \|\Gamma\|_{L^2} \nonumber \\
  & \leq C \|\omega\|_{L^2} \big( \|\theta\|_{B^{1-2\alpha}_{\infty,1}} + \|\theta\|_{L^2}\big) \|\Gamma\|_{L^2} \nonumber \\
  & \leq C \big(\|\Gamma\|_{L^2} +  \|\mathcal{R}_{1-2\alpha}\theta\|_{L^2}\big) \|\theta\|_{L^2\cap L^\infty}  \|\Gamma\|_{L^2} \nonumber \\
  & \leq C \|\Gamma\|_{L^2}^2 \big( 1 + \|\theta_0\|_{L^2\cap L^\infty} \big)
  + \|\theta_0\|_{L^1\cap L^\infty}^4,
\end{align*}
where in the last line we have used the fact that
\begin{align}\label{eq:Rbet-L2es}
  \|\mathcal{R}_{1-2\alpha} \theta(t)\|_{L^2} \leq C \|\theta(t)\|_{L^{\frac{1}{\alpha}}} \leq C \|\theta_0\|_{L^1\cap L^\infty}.
\end{align}
Gr\"onwall's inequality leads to
\begin{equation}\label{es:GamL2}
\begin{split}
  \|\Gamma\|_{L^\infty_T(L^2)}^2 + \frac{1}{\varepsilon} \|\Lambda^\alpha \Gamma\|_{L^2_T(L^2)}^2 &\leq\, C \big( \|\Gamma_0\|_{L^2}^2 + \|\theta_0\|_{L^1\cap L^\infty}^4 T \big)
  e^{C (1+\|\theta_0\|_{L^2\cap L^\infty})T}  \\
  &\leq\, C e^{CT}.
\end{split}
\end{equation}
In the above, $\Gamma_0:= \omega_0-\mathcal{R}_{1-
2 \alpha} \theta_0$ satisfies that
\begin{align*}
  \|\Gamma_0\|_{L^2}\leq \|\omega_0\|_{L^2} + \|\mathcal{R}_{1-2\alpha} \theta_0\|_{L^2}
  \leq \|u_0\|_{\dot H^1} + \|\theta_0\|_{L^{\frac{1}{\alpha}}} < \infty.
\end{align*}
Furthermore, by using \eqref{eq:Rbet-L2es} again, we find
\begin{align}\label{es:nab-uL2}
  \|\nabla u\|_{L^\infty_T(L^2)} \leq \|\omega\|_{L^\infty_T(L^2)}
  \leq \|\Gamma\|_{L^\infty_T(L^2)} + \|\mathcal{R}_{1-2\alpha}\theta\|_{L^\infty_T(L^2)}
  \leq C e^{CT}.
\end{align}

Now we consider the estimation of $\|\Gamma\|_{L^1_T(B^1_{2,1})}$ and $\|\Gamma\|_{\widetilde{L}^1_T(B^{2\alpha}_{2,\infty})}$.
For all $q\in \mathbb{N}$, applying the frequency localization operator $\Delta_q$ to the equation \eqref{eq:Gamma} yields
\begin{equation}\label{eq:Delta-q-Gamm}
\begin{aligned}
  \partial_t \Delta_q\Gamma + u \cdot \nabla \Delta_q\Gamma + \frac{1}{\varepsilon}\Lambda^{2 \alpha} \Delta_q\Gamma
  = - \left[\Delta_q, u \cdot \nabla\right] \Gamma +
  \Delta_q\big(\left[\mathcal{R}_{1-2 \alpha}, u \cdot \nabla\right] \theta\big)  := f_q .
\end{aligned}
\end{equation}
Taking the scalar product of the above equation with $\Delta_q \Gamma$, we get
\begin{align*}
  \frac{1}{2}\frac{\dd}{\dd t} \|\Delta_q\Gamma(t)\|_{L^2}^2 + \frac{c_0}{\varepsilon} 2^{2\alpha q} \| \Delta_q \Gamma(t)\|_{L^2}^2
  \leq \|\Delta_q \Gamma\|_{L^2} \|f_q\|_{L^2},
\end{align*}
with some $c_0>0$ absolute number. Integrating in the time variable leads to
\begin{align*}
  \|\Delta_q \Gamma(t)\|_{L^2} \leq e^{- \frac{c_0}{\varepsilon} 2^{2\alpha q} t} \|\Delta_q \Gamma_0\|_{L^2}
  + \int_0^t e^{- \frac{c_0}{\varepsilon} (t-\tau) e^{2\alpha q}} \|f_q(\tau)\|_{L^2} \dd \tau,
\end{align*}
and
\begin{align}\label{eq:Del-Gam-L2}
  \|\Delta_q\Gamma\|_{L^1_t(L^2)} \leq \frac{\varepsilon}{c_0} 2^{-2\alpha q} \|\Delta_q \Gamma_0\|_{L^2}
  + \frac{\varepsilon}{c_0} 2^{-2\alpha q} \|f_q\|_{L^1_t(L^2)}.
\end{align}
Let $N\in \mathbb{N}$ be a constant chosen later. Taking advantage of \eqref{es:GamL2}, \eqref{es:nab-uL2}
and Lemmas \ref{lem:Rbeta-cm}, \ref{lem:HKR}, we have
\begin{eqnarray}\label{eq:Gam-B121}
   \|\Gamma\|_{L^1_T(B^1_{2,1})} &=& \displaystyle\sum_{-1\leq q < N} 2^q \|\Delta_q \Gamma\|_{L^1_T(L^2)}
  + \sum_{q \geq N} 2^q \|\Delta_q \Gamma\|_{L^1_T(L^2)} \notag \\
  &\leq& C 2^N \|\Gamma\|_{L^1_T(L^2)} \notag \\
  && + C \sum_{q\geq N} 2^{q(1-2\alpha)} \Big( \|\Gamma_0\|_{L^2}+ \,\|[\Delta_q, u\cdot\nabla]\Gamma\|_{L^1_T(L^2)} + \|\Delta_q([\mathcal{R}_{1-2\alpha},u\cdot\nabla]\theta)\|_{L^1_T(L^2)} \Big)
  \nonumber \\
  &\leq& C 2^N  \|\Gamma\|_{L^1_T(L^2)} \notag\\
  &&+ \,C  2^{N(1-2\alpha)} \bigg( \|\Gamma_0\|_{L^2}  +  \|\nabla u\|_{L^\infty_T(L^2)}
  \Big(\|\Gamma\|_{L^1_T(B^0_{\infty,\infty})} +  \|\theta\|_{L^1_T(B^{1-2\alpha}_{\infty,\infty} \cap L^2)}\Big) \bigg) \nonumber \\
  &\leq& C e^{CT} 2^N + C 2^{N(1-2\alpha)} e^{CT} \|\Gamma\|_{L^1_T(B^1_{2,1})},
\end{eqnarray}
where in the last line we have used the continuous embedding $B^1_{2,1}(\mathbf{D})\hookrightarrow B^0_{\infty,\infty}(\mathbf{D})$
together with the fact that
\begin{align*}
  \|\theta\|_{L^1_T(B^{1-2\alpha}_{\infty,\infty}\cap L^2)} \leq
  \|\theta\|_{L^1_T(L^2\cap L^\infty)} \leq T \|\theta_0\|_{L^2\cap L^\infty}.
\end{align*}
By choosing $N\in \mathbb{N}$ so that $C 2^{N(1-2\alpha)} e^{CT} \approx \frac{1}{2}$, we infer that
\begin{align*}
  \|\Gamma\|_{L^1_T(B^1_{2,1})} \leq C e^{CT}.
\end{align*}
Repeating the above process, we find that
\begin{eqnarray}\label{eq:Gam-B2a2inf}
  \|\Gamma\|_{\widetilde{L}^1_T(B^{2\alpha}_{2,\infty})}
  &=&  \sum_{q\geq -1} 2^{2\alpha q} \|\Delta_q \Gamma\|_{L^1_T(L^2)} \nonumber \\
  &\leq& C \|\Delta_{-1}\Gamma\|_{L^1_T(L^2)} \nonumber \\
  && + C\,\sup_{q\in \mathbb{N}} \Big(\|\Delta_q \Gamma_0\|_{L^2}
  + \|[\Delta_q, u\cdot \nabla]\Gamma\|_{L^1_T(L^2)} + \|\Delta_q [\mathcal{R}_{1-2\alpha},u\cdot\nabla]\theta\|_{L^1_T(L^2)} \Big) \nonumber \\
  &\leq& \ C \ \|\Gamma\|_{L^1_T(L^2)}
   + \ C \  \|\Gamma_0\|_{L^2} +
  C \|\nabla u\|_{L^\infty_T(L^2)} \Big( \|\Gamma\|_{L^1_T(B^0_{\infty,\infty})} + \|\theta\|_{L^1_T(B^{1-2\alpha}_{\infty,2} \cap L^2)} \Big) \nonumber \\
  &\leq& Ce^{CT}.
\end{eqnarray}

Next, we want to prove that $\nabla u\in \widetilde{L}^1_T(C^{2\alpha-1})$. By using the identity
\begin{align}\label{eq:nab-u-rela}
  \nabla u = \nabla \nabla^\perp \Lambda^{-2} \omega = \nabla \nabla^\perp \Lambda^{-2} \Gamma + \nabla \nabla^\perp \partial_1 \Lambda^{-2-2\alpha} \theta,
\end{align}
we have
\begin{equation}\label{es:nab-uHold1}
\begin{split}
  \|\nabla u\|_{\widetilde{L}^1_T(C^{2\alpha-1})} & \leq \|\Delta_{-1} \nabla u\|_{L^1_T (L^\infty)}
  + \sup_{q\in \mathbb{N}} 2^{q(2\alpha-1)}\|\Delta_q \nabla u\|_{L^1_T(L^\infty)}  \\
  & \leq C \|\nabla u\|_{L^1_T(L^2)} + C \sup_{q\in \mathbb{N}} 2^{q(2\alpha -1)} \Big(\|\Delta_q \Gamma\|_{L^1_T(L^\infty)}
  + 2^{q(1-2\alpha)}\|\Delta_q \theta\|_{L^1_T(L^\infty)}\Big)  \\
  & \leq C \|\nabla u\|_{L^1_T(L^2)} + C \|\Gamma\|_{\widetilde{L}^1_T(B^{2\alpha}_{2,\infty})}
  + C \sup_{q\in \mathbb{N}} \|\Delta_q \theta\|_{L^1_T(L^\infty)} \\
  & \leq C e^{CT}.
\end{split}
\end{equation}
Since we have the following embedding $C^{2\alpha-1} \hookrightarrow B^0_{\infty,1} \hookrightarrow L^\infty$, then $\|\nabla u\|_{L_T^1L^\infty} \leq C e^{CT}$.
Hence gathering the above estimates completes the proof of \eqref{eq:u-Gam-es2}. \\

Let us prove the statement $(\textbf{2})$. From the vorticity equation \eqref{eq:vorticity} with $\alpha =1$ and the classical $L^2$ estimate, it is not difficult to see that
\begin{align*}
  \frac{\varepsilon}{2} \frac{\dd}{\dd t} \|\omega(t)\|_{L^2}^2 + \|\nabla \omega(t)\|_{L^2}^2
  \leq \Big|\int_{\mathbb{R}^2}\theta\, \partial_1\omega(x,t) \dd x \Big|
  \leq \frac{1}{2} \|\theta(t)\|_{L^2}^2 + \frac{1}{2} \|\nabla \omega(t)\|_{L^2}^2.
\end{align*}
By integrating in the time variable we obtain
\begin{align*}
  \|\omega(t)\|_{L^2}^2 + \frac{2}{\varepsilon} \|\nabla \omega(t)\|_{L^2}^2
  \leq \|\omega_0\|_{L^2}^2 + \frac{1}{\varepsilon} \|\theta_0\|_{L^2}^2 t,
\end{align*}
which also gives
\begin{align*}
  \|\nabla u\|_{L^\infty_T(L^2)} + \frac{1}{\sqrt{\varepsilon}} \|\nabla^2 u\|_{L^2}
  \leq 2 \big(\|\omega_0\|_{L^2} + \|\theta_0\|_{L^2} \big) \Big(1 + \sqrt{\tfrac{T}{\varepsilon}} \Big).
\end{align*}

Next we want to prove the estimate of $\|\omega\|_{L^1_T(B^0_{\infty,1})}$. Using the identity
$$\Gamma = \omega - \mathcal{R}_{-1}\theta = \omega - \partial_1 (-\Delta)^{-1}\theta,$$
together with Lemmas \ref{lem:Rbeta-cm} and \ref{lem:HKR},
we infer that for all $q\in \mathbb{N}$,
\begin{align*}
  \|[\Delta_q, u\cdot\nabla]\Gamma\|_{L^2} & \leq
  \|[\Delta_q,u\cdot\nabla](\mathrm{Id}-\Delta_{-1})\Gamma\|_{L^2}
  + \|\Delta_q(u\cdot\nabla \Delta_{-1}\Gamma)\|_{L^2}
  + \|u\cdot\nabla \Delta_q \Delta_{-1}\Gamma\|_{L^2} \\
  & \leq C \|\nabla u\|_{L^2} \|(\mathrm{Id} -\Delta_{-1})\Gamma\|_{B^0_{\infty,\infty}}
  + C \|u\|_{L^2} \|\nabla \Delta_{-1}\Gamma\|_{L^\infty} \\
  & \leq C \|\nabla u\|_{L^2} \big(\|\omega\|_{B^0_{\infty,\infty}} + \|\theta\|_{L^\infty} \big)
  + C \|u\|_{L^2} \big(\|\omega\|_{L^2} + \|\theta\|_{L^2} \big),
\end{align*}
and
\begin{align*}
  \|[\mathcal{R}_{-1},u\cdot\nabla]\theta\|_{L^2} \leq C \Big( \|\nabla u\|_{L^2}
  \left\Vert\theta\right\Vert_{B^{-1/2}_{\infty,1}} + \|u\|_{L^2} \|\theta\|_{L^2} \Big)
  \leq C \|u\|_{H^1} \|\theta\|_{L^2\cap L^\infty}.
\end{align*}
Hence, by noticing that the inequality \eqref{eq:Del-Gam-L2} with $\alpha=1$ still holds, and applying the above estimates and Proposition \ref{lem:eneEs}, we find that for some $N\in \mathbb{N}$ (to be chosen later), we have
\begin{align*}
  \|\omega\|_{L^1_T(B^0_{\infty,1})} & \leq \sum_{-1\leq q\leq N} \|\Delta_q \omega\|_{L^1_T(L^\infty)}
  + \sum_{q\geq N} \|\Delta_q \mathcal{R}_{-1} \theta \|_{L^1_T(L^\infty)}
  + \sum_{q\geq N} \|\Delta_q \Gamma\|_{L^1_T(L^\infty)} \\
  & \leq C\sum_{-\infty< q \leq N} 2^q \|\Delta_q \omega\|_{L^1_T(L^2)}
  + C \sum_{q\geq N} 2^{-q} \|\Delta_q \theta\|_{L^1_T(L^\infty)} \\
  & \quad + C \varepsilon \sum_{q\geq N} 2^{-q} \Big( \|\Delta_q \Gamma_0\|_{L^2}
  + \|[\Delta_q, u\cdot\nabla]\Gamma\|_{L^1_T(L^2)}
  + \|[\mathcal{R}_{-1},u\cdot\nabla]\theta\|_{L^1_T(L^2)}\Big) \\
  & \leq C 2^N T \Big(1+ \sqrt{\tfrac{T}{\varepsilon}} \Big) +
  C \varepsilon  2^{-N} \Big( \|(\omega_0,\theta_0)\|_{L^2}
  + \|\nabla u\|_{L^\infty_T(L^2)} \|\omega\|_{L^1_T(B^0_{\infty,\infty})} \Big) \\
  & \quad + C \varepsilon 2^{-N} \Big( \|\nabla u\|_{L^\infty_T(L^2)} \|\theta\|_{L^1_T(L^2\cap L^\infty)}
  + \|u\|_{L^\infty_T(L^2)} \big( \|\omega\|_{L^1_T(L^2)}
  + \|\theta\|_{L^1_T(L^2\cap L^\infty)}\big)\Big) \\
  & \leq C 2^N T \Big(1+ \sqrt{\tfrac{T}{\varepsilon}} \Big)
  + C \sqrt{\varepsilon} 2^{-N} \big(1+\sqrt{T} \big) \|\omega\|_{L^1_T(B^0_{\infty,\infty})}
  + C \varepsilon 2^{-N} (1+ T^2) \Big(1 + \sqrt{\tfrac{T}{\varepsilon}} \Big).
\end{align*}
Choosing $N\in \mathbb{N}$ such that $\max\{C\sqrt{\varepsilon} (1+\sqrt{T}),1\}2^{-N}\approx \frac{1}{2}$ allows us to write that
\begin{align*}
  \|\omega\|_{L^1_T(B^0_{\infty,1})} + \|\nabla u\|_{L^1_T(B^0_{\infty,1})}
  \leq C (1+T^2) \Big(1+\sqrt{\tfrac{T}{\varepsilon}}\Big),
\end{align*}
where $C>0$ depends only on the norms of $u_0$ and $\theta_0$, importantly,  $C>0$ is independent of $\varepsilon$.
Note that from the identity $\Gamma =\omega - \mathcal{R}_{-1}\theta$, we may write that
\begin{equation}\label{eq:Gam-B0inf-1}
\begin{split}
  \|\Gamma\|_{L^1_T(B^0_{\infty,1})} & \leq \|\omega\|_{L^1_T(B^0_{\infty,1})} + \|\mathcal{R}_{-1}\theta\|_{L^1_T(B^0_{\infty,1})} \\
  & \leq \|\omega\|_{L^1_T(B^0_{\infty,1})} + C \|\Delta_{-1}\mathcal{R}_{-1}\theta\|_{L^1_T(L^4)}
  + C \sum_{q\in\mathbb{N}} \|\Delta_q \mathcal{R}_{-1}\theta\|_{L^1_T(L^\infty)} \\
  & \leq \|\omega\|_{L^1_T(B^0_{\infty,1})} + C \|\theta\|_{L^1_T(L^1\cap L^\infty)}
  \leq C (1+T^2) \Big(1+\sqrt{\tfrac{T}{\varepsilon}}\Big).
\end{split}
\end{equation}

Similarly as above, we get that for all $\gamma\in (0,1)$,
\begin{eqnarray*}
   \|\omega\|_{L^1_T( B^\gamma_{\infty,1})}
  &\leq&  \|\Delta_{-1} \omega\|_{L^1_T(L^\infty)}
  + \sum_{q\in\mathbb{N}} 2^{q\gamma}\|\Delta_q \mathcal{R}_{-1}\theta\|_{L^1_T(L^\infty)}
  + \sum_{q\in\mathbb{N}} 2^{q\gamma} \|\Delta_q \Gamma\|_{L^1_T(L^\infty)} \\
  &\leq& C \| \Delta_{-1} \omega\|_{L^1_T(L^2)}
  + C \sum_{q\in \mathbb{N}} 2^{q(\gamma -1)} \|\Delta_q \theta\|_{L^1_T(L^\infty)}
  + C \sum_{q\in\mathbb{N}} 2^{q(\gamma +1)} \| \Delta_q \Gamma\|_{L^1_T(L^2)} \\
  &\leq& C  T \Big(1+\sqrt{\tfrac{T}{\varepsilon}} \Big) \\
  && \ +  \ C \varepsilon \sum_{q\in\mathbb{N}} 2^{q(\gamma-1)} \Big( \|\Delta_q \Gamma_0\|_{L^2}
  + \|[\Delta_q,u\cdot\nabla]\Gamma \|_{L^1_T(L^2)} + \|[\mathcal{R}_{-1},u\cdot\nabla]\theta\|_{L^1_T(L^2)}\Big) \\
  &\leq& C (1+ T^2) \Big(1+ \sqrt{\tfrac{T}{\varepsilon}} \Big) + C \sqrt{\varepsilon} \big(1+\sqrt{T} \big)
  \|\omega\|_{L^1_T(B^0_{\infty,\infty})} \\
  &\leq& C (1+ T^3) \Big(1+ \sqrt{\tfrac{T}{\varepsilon}} \Big).
\end{eqnarray*}
Collecting the above estimates leads to \eqref{eq:u-Gam-es3} and ends the proof of the statement $(\textbf{2})$. \\


It remains to prove the statement $(\textbf{3})$.
Consider \eqref{eq:Gamma} in the case $\alpha=1$. By taking the scalar product of the evolution equation in $\Gamma$ with $\Gamma$ itself,  and using the commutator estimate \eqref{eq:R-1cm-es2}, we get
\begin{eqnarray*}
  \frac{1}{2} \frac{\dd }{\dd t} \|\Gamma(t)\|_{L^2}^2 + \frac{1}{\varepsilon} \|\nabla \Gamma(t)\|_{L^2}^2
  &\leq& \|[\mathcal{R}_{-1}, u\cdot \nabla]\theta(t)\|_{L^2} \|\Gamma(t)\|_{L^2} \\
  &\leq& \|[\mathcal{R}_{-1}, u\cdot \nabla]\theta(t)\|_{B^{\frac{1}{2}}_{2,\infty}} \|\Gamma(t)\|_{L^2} \\
  &\leq& C \|\omega\|_{L^2} \big(\|\theta\|_{B^{-\frac{1}{2}}_{\infty,\infty}} + \|\theta\|_{L^2}\big) \|\Gamma(t)\|_{L^2} \\
  &\leq& C \big(\|\Gamma\|_{L^2} + \|\mathcal{R}_{-1}\theta\|_{L^2} \big) \|\theta_0\|_{L^2\cap L^\infty}  \|\Gamma(t)\|_{L^2} \\
  &\leq& C \big(\|\Gamma(t)\|_{L^2}  +  1 \big)\|\Gamma(t)\|_{L^2} \\
   &\leq& C \|\Gamma(t)\|_{L^2}^2 + C,
\end{eqnarray*}
where in the last line we have used that
\begin{eqnarray*}
  \|\mathcal{R}_{-1}\theta(t)\|_{L^2(\mathbb{T}^2)} \leq C \|\Lambda^{-1}\theta(t)\|_{L^4(\mathbb{T}^2)}
  \leq C \|\theta(t)\|_{L^{\frac{4}{3}}(\mathbb{T}^2)}
  \leq C \|\theta_0\|_{L^{\frac{4}{3}}(\mathbb{T}^2)}.
\end{eqnarray*}
Gr\"onwall's inequality and the fact that $\|\Gamma_0\|_{L^2} \leq \|\omega_0\|_{L^2} + \|\mathcal{R}_{-1}\theta_0\|_{L^2} \leq C $
allow us to state that, for all $t>0$, the following control holds
\begin{align*}
  \|\Gamma(t)\|_{L^2}^2 + \frac{1}{\varepsilon}\int_0^t \|\nabla \Gamma(\tau)\|_{L^2}^2\dd \tau
  \leq (\|\Gamma_0\|_{L^2}^2 + Ct\big) e^{Ct} \leq C e^{Ct}.
\end{align*}
It is easy to see that, using the identity $\omega = \Gamma + \mathcal{R}_{-1}\theta$, we may write that
\begin{align*}
  \|\nabla u(t)\|_{L^2} \leq \|\omega(t)\|_{L^2} \leq \|\Gamma(t)\|_{L^2}
  + \|\mathcal{R}_{-1}\theta(t)\|_{L^2} \leq C e^{Ct}.
\end{align*}

Then,  following the same lines as the proof of \eqref{eq:Gam-B121}, we find that for some $Q\in \mathbb{N}$ (that will be fixed later),
\begin{align*}
  & \|\Gamma\|_{L^1_T(B^1_{2,1})} = \sum_{-1\leq q < Q} 2^q \|\Delta_q \Gamma\|_{L^1_T(L^2)}
  + \sum_{q \geq Q} 2^q \|\Delta_q \Gamma\|_{L^1_T(L^2)} \notag \\
  & \leq C 2^Q \|\Gamma\|_{L^1_T(L^2)} +
  C \sum_{q\geq Q} 2^{-q} \Big( \|\Gamma_0\|_{L^2}  + \|[\Delta_q, u\cdot\nabla]\Gamma\|_{L^1_T(L^2)}
  + \|[\mathcal{R}_{-1},u\cdot\nabla]\theta\|_{L^1_T(L^2)} \Big) \nonumber \\
  & \leq C 2^Q  e^{CT}
  + C  2^{-Q} \bigg( \|\Gamma_0\|_{L^2}  +  \|\nabla u\|_{L^\infty_T(L^2)}
  \Big(\|\Gamma\|_{L^1_T(B^0_{\infty,\infty})} +  \|\theta\|_{L^1_T(B^{-1/2}_{\infty,\infty} \cap L^2)}\Big)
  \bigg) \nonumber \\
  & \leq C e^{CT} 2^Q + C 2^{-Q} e^{CT} \|\Gamma\|_{L^1_T(B^1_{2,1})}.
\end{align*}
By choosing $Q\in\mathbb{N}$ so that $C e^{CT} 2^{-Q} \approx \frac{1}{2}$, we find that
\begin{align*}
  \|\Gamma\|_{L^1_T(B^1_{2,1})} + \|\Gamma\|_{L^1_T(B^0_{\infty,1})} \leq C e^{CT}.
\end{align*}
Following the same idea as the proof of \eqref{eq:Gam-B2a2inf} and \eqref{es:nab-uHold1}, we infer that
\begin{align*}
  \|\Gamma\|_{\widetilde{L}^1_T(B^2_{2,\infty})} + \|\nabla u\|_{\widetilde{L}^1_T(B^1_{\infty,\infty})}
  \leq Ce^{CT}.
\end{align*}
Gathering the above estimates and the embedding $B^1_{\infty,\infty}(\mathbb{T}^2)\hookrightarrow C^\gamma(\mathbb{T}^2)$ ($0<\gamma<1$) lead to \eqref{eq:u-Gam-es4} and therefore the proof of the last statement $(\textbf{3})$, thus this completes the proof of Proposition \ref{prop:ap-es1}.
\end{proof}





The following result is concerned with the \textit{a priori} estimates for $(u,\Gamma)$ solving the system
$(\mathrm{B}_\alpha)$ with $u_0\in H^1\cap W^{1,p}$ and $\theta_0\in L^1\cap L^\infty$.

\begin{proposition}\label{prop:ap-es2}
Let $\varepsilon =\frac{1}{\mathrm{Pr}}\in (0,1]$, $\alpha\in(\frac{1}{2},1]$.
Let $\mathbf{D}$ be either $\mathbb{R}^2$ or $\mathbb{T}^2$.
Suppose that $(u,\theta)$ is the smooth solution for the 2D Boussinesq-Navier-Stokes system $(\mathrm{B}_\alpha)$
satisfying
\begin{itemize}
 \item $u_0\in H^1\cap\dot{W}^{1,p}(\mathbf{D})$, $2<p<\infty$,
 \item $\nabla\cdot u_0=0$,
\item $\theta_0\in L^1\cap L^{\infty}(\mathbf{D})$.
\end{itemize}
Then, for all $T>0$, there exists a constant $C>0$ depending on $\alpha$ and the norms of $(u_0,\theta_0)$
but independent of $\varepsilon$ such that the following statements hold.
\begin{enumerate}[(1)]
\item
If $\alpha\in(\frac{1}{2},1)$, then we have
\begin{align}\label{es:nabu-Gam2}
  \|(\nabla u,\Gamma)\|_{L^\infty_T(L^p)}
  + \|\Gamma\|_{\widetilde{L}^1_T(B^{2\alpha}_{p,\infty})} + \|\Gamma\|_{L^1_T (B^{2\alpha-1}_{\infty,1})}
  +  \|\nabla u\|_{L^1_T(C^{2\alpha-1})}
  \leq C e^{CT},
\end{align}
In particular, if additionally $\frac{2}{2\alpha-1}<p<\infty$, we also have
\begin{align}\label{eq:Gam-Lip-es}
  \|\Gamma\|_{L^1_T(B^1_{\infty,1})} \leq C e^{CT}.
\end{align}
\item If $\mathbf{D}=\RR^2$, $\alpha=1$, then for all $\gamma\in(0,1)$ we have
\begin{align}
  \|(\nabla u,\Gamma)\|_{L^\infty_T(L^p)} + \|\nabla u\|_{L^1_T(B^1_{\infty,\infty})}
  & \leq C (1+T^4) \Big(1+\tfrac{T}{\varepsilon}\Big), \quad  \label{es:nabu-Gam2-2a} \\
  \|\Gamma\|_{\widetilde{L}^1_T(B^{2}_{p,\infty})} + \|\Gamma\|_{L^1_T (B^{1}_{\infty,1})}
  & \leq C (1+ T^4)\Big(1+ \sqrt{\tfrac{T}{\varepsilon}} \Big). \label{es:nabu-Gam2-2}
\end{align}
\item
If $\mathbf{D}=\TT^2$, $\alpha=1$,
we have
\begin{align}\label{es:nabu-Gam3}
  \|(\nabla u,\Gamma)\|_{L^\infty_T(L^p)}
  + \|\Gamma\|_{\widetilde{L}^1_T(B^{2\alpha}_{p,\infty})} + \|\Gamma\|_{L^1_T (B^1_{\infty,1})}
  + \|\nabla u\|_{L^1_T(B^1_{\infty,\infty})}
  \leq C e^{CT}.
\end{align}
\end{enumerate}
\end{proposition}

\begin{remark}\label{rem:prop-ap-es2}
Under the assumptions of Proposition \ref{prop:ap-es2},
and if we additionally consider the temperature patch initial data
$\theta_0(x) = \bar{\theta}_0(x) \mathbf{1}_{D_0}(x)$ as in Theorem \ref{thm:main} with
\begin{align*}
\begin{cases}
  \partial D_0 \in C^{2\alpha},\; \bar{\theta}_0 \in L^\infty(\overline{D_0}), \;
  \quad & \textrm{if}\;\; \alpha\in (\frac{1}{2},1), \\
  \partial D_0\in W^{2,\infty}, \;\bar{\theta}_0\in C^\mu(\overline{D_0}),\;0<\mu<1,
  \quad & \textrm{if}\;\; \alpha =1,
\end{cases}
\end{align*}
we can prove that 
\begin{equation*}
\begin{cases}
  \partial D(t) \in L^\infty_T(C^{2\alpha}), \quad & \textrm{if}\;\; \alpha\in (\frac{1}{2},1), \\
  \partial D(t) \in L^\infty_T(W^{2,\infty}),\quad & \textrm{if}\;\; \alpha=1.
\end{cases}
\end{equation*}
In fact, if $\alpha\in (\frac{1}{2},1)$, the claim follows from the fact that $\nabla X_t^{\pm1} \in L^\infty_T(C^{2\alpha-1})$ which is a direct consequence of  \eqref{es:nabu-Gam2} and \eqref{eq:Xt-C1gam};
while if $\alpha=1$, the claim follows from the fact that $X_t^{\pm 1} \in L^\infty_T(W^{2,\infty})$
which is a consequence of the following estimate
\begin{align}\label{eq:W2inf-es}
  \|\nabla u\|_{L^1_T(W^{1,\infty}(\mathbf{D}))}
  \leq
  \begin{cases}
    C \exp \Big\{C (1+ T^4) \big(1+ \sqrt{\frac{T}{\varepsilon}}\big)\Big\}, \quad &\textrm{for}\;\; \alpha =1, \mathbf{D} =\mathbb{R}^2, \\
    C \exp \{ C e^{CT} \},\quad & \textrm{for}\;\; \alpha=1, \mathbf{D} = \mathbb{T}^2;
  \end{cases}
\end{align}
as for the proof of  \eqref{eq:W2inf-es}, we infer from \eqref{eq:nab-u-rela} that
\begin{align*}
  \|\nabla^2 u\|_{L^1_T(L^\infty)}
  & \leq \| \nabla^2 \nabla^\perp \Lambda^{-2} \Gamma \|_{L^1_T(L^\infty)}
  + \|\nabla^2 \nabla^\perp \partial_1 \Lambda^{-4} \theta\|_{L^1_T(L^\infty)} \\
  & \leq C \|\Gamma\|_{L^1_T(B^1_{\infty,1})} + \|\nabla^2 \nabla^\perp \partial_1 \Lambda^{-4} \theta\|_{L^1_T(L^\infty)};
\end{align*}
by following the geometric lemma in \cite{bertozzi1993} or  \cite[Sec. 3.2]{CMX22}, we conclude that $\|\nabla^2 \nabla^\perp \partial_1 \Lambda^{-4} \theta\|_{L^1_T(L^\infty)}$ and $\|\nabla u\|_{L^1_T(W^{1,\infty})}$
is controlled by \eqref{eq:W2inf-es}.
\end{remark}

\begin{proof}[Proof of Proposition \ref{prop:ap-es2}]
Let us first prove \textbf{(1)}. In this case $\alpha\in (\frac{1}{2},1)$, we prove \eqref{es:nabu-Gam2}-\eqref{eq:Gam-Lip-es}
in a unified approach without distinguishing $\mathbf{D}$ to be $\mathbb{R}^2$ or $\mathbb{T}^2$.

By multiplying both sides of the $\Gamma$-equation \eqref{eq:Gamma} with $|\Gamma|^{p-2}\Gamma(x,t)$
and integrating over the spatial variable, we use the integration by parts to get
\begin{align}\label{eq:Gam-Lp-es1a}
  \frac{1}{p} \frac{\dd}{\dd t} \|\Gamma(t)\|_{L^p}^p + \frac{1}{\varepsilon} \int_\mathbf{D} \Lambda^{2\alpha} \Gamma\,\, |\Gamma|^{p-2}
  \Gamma(x,t) \dd x \leq \|[\mathcal{R}_{1-2\alpha}, u\cdot\nabla]\theta \|_{L^p} \|\Gamma(t)\|_{L^p}^{p-1} .
\end{align}
The positivity lemma in \cite{CorC04} ensures that the term coming from dissipation $\Lambda^{2\alpha}\Gamma$ is nonnegative,
thus by applying Lemma \ref{lem:Rbeta-cm} and the Cader\'on-Zygmund theorem, we find
\begin{align*}
  \frac{\dd }{\dd t}\|\Gamma(t)\|_{L^p} & \leq \|[\mathcal{R}_{1-2\alpha},u\cdot\nabla]\theta(t)\|_{L^p} \\
  & \leq C \|\omega(t)\|_{L^p} \big(\|\theta(t)\|_{B^{1-2\alpha}_{\infty,1}} + \|\theta(t)\|_{L^2} \big) \\
  & \leq C \big(\|\Gamma(t)\|_{L^p} + \|\mathcal{R}_{1-2\alpha}\theta(t)\|_{L^p} \big) \|\theta(t)\|_{L^2\cap L^\infty} \\
  & \leq C \big(\|\Gamma(t)\|_{L^p} + \|\theta_0\|_{L^1\cap L^\infty} \big) \|\theta_0\|_{L^2\cap L^\infty},
\end{align*}
where in the last line we have used the following estimate
\begin{align}\label{eq:R1-2alp-the1}
  \|\mathcal{R}_{1-2\alpha} \theta(t)\|_{L^p} \leq C \|\Lambda^{1-2\alpha}\theta(t)\|_{L^p}
  \leq C \| \theta(t)\|_{L^{\frac{2p}{(2\alpha-1)p+2}}} \leq C \|\theta_0\|_{L^1\cap L^\infty}.
\end{align}
Gr\"onwall's inequality and the fact that $\|\Gamma_0\|_{L^p}\leq \|\omega_0\|_{L^p} + \|\mathcal{R}_{1-2\alpha}\theta_0\|_{L^p}\leq C$
imply that $\|\Gamma(t)\|_{L^p} \leq  e^{C t}$, thus, using the identity $\omega =\Gamma + \mathcal{R}_{1-2\alpha}\theta$, one finds
\begin{align*}
  \|(\Gamma,\nabla u)\|_{L^\infty_T(L^p)} \leq C \|(\Gamma,\omega)\|_{L^\infty_T(L^p)}
  \leq C \|\Gamma\|_{L^\infty_T(L^p)} + \|\mathcal{R}_{1-2\alpha}\theta\|_{L^\infty_T(L^p)} \leq C e^{CT}.
\end{align*}

By taking the scalar product, for all $q\in \mathbb{N}$, of the equation \eqref{eq:Delta-q-Gamm} with $\left|\Delta_q\Gamma\right|^{p-2} \Delta_q\Gamma(x,t)$
and then using the following estimate (see \cite{CMZ07})
\begin{align*}
  \int_\mathbf{D} \left(\Lambda^{2 \alpha} \Delta_q\Gamma\right)\left|\Delta_q\Gamma\right|^{p-2} \Delta_q\Gamma
  \mathrm{~d} x \geq c\, 2^{2 \alpha q}\left\|\Delta_q\Gamma\right\|_{L^p}^p,\quad \forall q\in \NN,
\end{align*}
for some $c>0$ independent of $q$, we obtain
\begin{align*}
  \frac{1}{p} \frac{\mathrm{d}}{\mathrm{d} t}\left\|\Delta_q\Gamma(t)\right\|_{L^p}^p +
  \frac{c}{\varepsilon} \, 2^{2 \alpha q}\left\|\Delta_q\Gamma(t)\right\|_{L^p}^p
  \leq\left\|\Delta_q\Gamma(t)\right\|_{L^p}^{p-1}\left\|f_q(t)\right\|_{L^p},
\end{align*}
which gives
\begin{align}\label{es:Deltaq}
  \left\|\Delta_q\Gamma(t)\right\|_{L^p} \leq e^{-\frac{c}{\varepsilon} t 2^{2 \alpha q}}\left\|\Delta_q\Gamma_0\right\|_{L^p}
  + \int_0^t e^{- \frac{c}{\varepsilon} (t-\tau) 2^{2 \alpha q}}\left\|f_q(\tau)\right\|_{L^p} \mathrm{d} \tau.
\end{align}
Taking the $L^1([0,T])$ norm, and using \eqref{eq:u-Gam-es2}, Lemmas~\ref{lem:Rbeta-cm} and \ref{lem:HKR},
we get that for all $q\in \mathbb{N}$,
\begin{align*}
  2^{2 \alpha q }\left\|\Delta_q\Gamma\right\|_{L_T^1 (L^p)}
  & \leq C  \left\|\Delta_q\Gamma_0\right\|_{L^p}+ C \int_0^T \left\|\left[\Delta_q, u \cdot \nabla\right]
  \Gamma \right\|_{L^p} \mathrm{d} \tau
  + C \int_0^T \left\|\Delta_q \left(\left[\mathcal{R}_{1-2 \alpha}, u \cdot \nabla\right] \theta\right)
  \right\|_{L^p} \mathrm{d} \tau \\
  & \leq C \left\|\Gamma_0\right\|_{L^p}  + C \|\nabla u\|_{L^\infty_T(L^p)} \left(\|\Gamma\|_{L^1_T(B^0_{\infty,\infty})}
  +\|\theta\|_{L^1_T(B^{1-2\alpha}_{\infty, \infty}\cap L^2)} \right)
  \leq  C e^{CT}.
\end{align*}
Hence, we have
\begin{align*}
  \|\Gamma\|_{\widetilde{L}^1_T(B^{2\alpha}_{p,\infty})} \leq C \|\Delta_{-1}\Gamma\|_{L^1_T (L^p)}
  + \sup_{q\in \mathbb{N}} 2^{2\alpha q}  \|\Delta_q \Gamma\|_{L^1_T(L^p)}  \leq C e^{CT}.
\end{align*}
Together with the continuous embedding $B^{2\alpha}_{p,\infty}\hookrightarrow B^{2\alpha-\frac{2}{p}}_{\infty,\infty}\hookrightarrow B^{2\alpha -1}_{\infty,1}$ ($p>2$) and $B^{2\alpha-\frac{2}{p}}_{\infty,\infty}\hookrightarrow B^1_{\infty,1}$
($p>\frac{2}{2\alpha-1}$),
the above inequality yields $\|\Gamma\|_{L^1_T(B^{2\alpha-1}_{\infty,1})} \leq Ce^{CT} $ and \eqref{eq:Gam-Lip-es}.

As for the estimate of $\nabla u$ in the space of $L^1_T(C^{2\alpha-1})= L^1_T(B^{2\alpha-1}_{\infty,\infty})$,
it suffices to follow the same lines as the proof of \eqref{es:nab-uHold1}
and we see that
\begin{equation}\label{eq:nab-u-C2alp-1}
\begin{split}
  \|\nabla u\|_{L^1_T(C^{2\alpha-1})} & \leq C \|\Delta_{-1}\nabla u\|_{L^1_T(L^\infty)} +
  C \Big\| \sup_{q\in \mathbb{N}}2^{q(2\alpha-1)} \|\Delta_q \nabla u\|_{L^\infty} \Big\|_{L^1_T} \\
  & \leq C \|\nabla u\|_{L^1_T(L^2)} + C \|\Gamma\|_{L^1_T( B^{2\alpha-1}_{\infty,\infty})}
  + C \|\theta\|_{L^1_T(B^0_{\infty,\infty})}
  \leq C e^{CT}.
\end{split}
\end{equation}
Hence collecting the above estimates gives \eqref{es:nabu-Gam2}-\eqref{eq:Gam-Lip-es} and therefore \textbf{(1)} is proved. \\

To prove the second statement, that is \textbf{(2)}, we consider the case $\mathbf{D}=\RR^2$ and $\alpha=1$.
We remark that \eqref{eq:Gam-Lp-es1a} now becomes
\begin{align*}
  \frac{1}{p} \frac{\dd}{\dd t} \|\Gamma(t)\|_{L^p}^p + \frac{p-1}{\varepsilon} \int_{\mathbb{R}^2} |\Gamma(x,t)|^{p-2}
  |\nabla\Gamma(x,t)|^2 \dd x \leq \|[\mathcal{R}_{-1}, u\cdot\nabla]\theta \|_{L^p} \|\Gamma(t)\|_{L^p}^{p-1},
\end{align*}
thus we use the embedding $B^{1-\frac{2}{p}}_{2,1}(\mathbb{R}^2)\hookrightarrow L^p(\mathbb{R}^2)$
and \eqref{eq:R-1cm-es1} to find
\begin{align}\label{eq:Gam-Lp-es1}
  \frac{\dd}{\dd t}\|\Gamma(t)\|_{L^p} &\leq \|[\mathcal{R}_{-1}, u\cdot\nabla]\theta(t)\|_{L^p}\\
  &\leq C \|[\mathcal{R}_{-1}, u\cdot\nabla]\theta(t) \|_{B^{1-2/p}_{2,1}} \nonumber  \\
  &\leq C \|\nabla u(t)\|_{L^2}\|\theta(t)\|_{B^{-2/p}_{\infty,1}}
  + C \|u(t)\|_{L^2}\|\theta(t)\|_{L^2} . \nonumber
\end{align}
Integrating in time and applying Proposition \ref{lem:eneEs} give that
\begin{align*}
  \|\Gamma\|_{L^\infty_T(L^p)} & \leq \|\Gamma_0\|_{L^p} + C  \big(\|\nabla u\|_{L^1_T(L^2)} + \|u\|_{L^1_T(L^2)} \big)
  \|\theta\|_{L^\infty_T(L^2\cap L^\infty)} \\
  & \leq \|\omega_0\|_{L^p} + \|\mathcal{R}_{-1}\theta_0\|_{L^p}
  + C \big( \sqrt{T} \|\nabla u\|_{L^2_T(L^2)} + T \|u\|_{L^\infty_T(L^2)}\big) \|\theta_0\|_{L^2\cap L^\infty} \\
  & \leq C (1+ T)\Big(1+ \tfrac{T}{\varepsilon}\Big),
\end{align*}
where in the above we also have used that $\|\mathcal{R}_{-1}\theta_0\|_{L^p} \leq C \|\theta_0\|_{L^{\frac{2p}{p+2}}} $.
Thus,
\begin{align*}
  \|\nabla u\|_{L^\infty_T(L^p)}\leq C \|\omega\|_{L^\infty_T(L^p)}
  \leq C \left(\|\Gamma\|_{L^\infty_T(L^p)} + \|\mathcal{R}_{-1} \theta\|_{L^\infty_T(L^p)} \right)
  \leq C (1+T)\Big(1+\tfrac{T}{\varepsilon}\Big).
\end{align*}

Then, we prove the estimate of $\|\Gamma\|_{\widetilde{L}^1_T(B^{2}_{p,\infty})}$.
We see that \eqref{es:Deltaq} with $\alpha=1$ still holds, and we integrate on the time interval $[0,T]$ to get that for all $q\in \mathbb{N}$,
\begin{align*}
  \|\Delta_q\Gamma\|_{L^1_t(L^p)} \leq \frac{\varepsilon}{c} 2^{-2 q} \|\Delta_q \Gamma_0\|_{L^p}
  + \frac{\varepsilon}{c} 2^{-2 q} \|f_q\|_{L^1_t(L^p)},
\end{align*}
for some $c>0$ independent of $q$ and $f_q$ given by \eqref{eq:Delta-q-Gamm}.
Thus by taking advantage of \eqref{eq:u-L2-es}, \eqref{eq:Gam-B0inf-1}, Lemmas \ref{lem:Rbeta-cm}, \ref{lem:HKR}
and the following estimate
\begin{align*}
  \|\Delta_{-1}\Gamma\|_{L^\infty_T(L^p)} & \leq \|\Delta_{-1}\omega\|_{L^\infty_T(L^p)}
  + \|\Delta_{-1}\mathcal{R}_{-1}\theta\|_{L^\infty_T(L^p)} \\
  & \leq C \|\omega\|_{L^\infty_T(L^2)} + C \|\mathcal{R}_{-1}\theta\|_{L^\infty_T(L^p)}
  \leq C \Big(1+ \sqrt{\tfrac{T}{\varepsilon}}\Big),
\end{align*}
we have
\begin{align}\label{eq:Gam-B2p-inf1}
  & \|\Gamma\|_{\widetilde{L}^1_T(B^{2}_{p,\infty})} \leq \|\Delta_{-1} \Gamma\|_{L^1_T(L^p)}
  + \sup_{q \in \mathbb{N}} 2^{2q} \|\Delta_q \Gamma\|_{L^1_T(L^p)} \nonumber \\
  & \leq C \|\Delta_{-1}\Gamma\|_{L^1_T(L^p)}
  + C \sup_{q\in \mathbb{N}} \Big( \|\Gamma_0\|_{L^p}
  + \varepsilon \|[\Delta_q, u\cdot\nabla]\Gamma\|_{L^1_T(L^p)}
  + \varepsilon \|\Delta_q([\mathcal{R}_{-1},u\cdot\nabla]\theta)\|_{L^1_T(L^p)} \Big) \nonumber \\
  & \leq C \|\Delta_{-1}\Gamma\|_{L^1_T(L^p)}  + C \|\Gamma_0\|_{L^p} \nonumber \\
  & \quad +  C \bigg( \varepsilon\|\nabla u\|_{L^\infty_T(L^p)}
  \Big(\|\Gamma\|_{L^1_T(B^0_{\infty,\infty})} +  \|\theta\|_{L^1_T(B^{-1/2}_{\infty,\infty} \cap L^2)}\Big)
  + \varepsilon\|u\|_{L_T^\infty L^2} \|\theta\|_{L_T^1 L^2}\bigg) \nonumber \\
  & \leq C T\Big(1+\sqrt{\tfrac{T}{\varepsilon}}\Big)  + C (1+ T^4) \Big(1+ \sqrt{\tfrac{T}{\varepsilon}} \Big)
  \leq C (1+ T^4) \Big( 1 + \sqrt{\tfrac{T}{\varepsilon}}\Big) .
\end{align}
As a result of the embedding $\widetilde{L}_T^1 (B^2_{p,\infty}) \hookrightarrow L_T^1(B^1_{\infty,1}) $,
we obtain the same upper bound as $\|\Gamma\|_{L_T^1B^1_{\infty,1}} $. Together with \eqref{eq:nab-u-C2alp-1} with $\alpha =1$,
we find
\begin{align}\label{eq:nab-u-B1inf}
  \|\nabla u\|_{L^1_T(B^1_{\infty,\infty})}
  \leq C\|\nabla u\|_{L^1_T(L^p)} + C \|\Gamma\|_{L^1_T(B^1_{\infty,\infty})} + C \|\theta\|_{L^1_T(B^0_{\infty,\infty})}
  \leq C (1+ T^4) \Big(1+ \tfrac{T}{\varepsilon}\Big).
\end{align}
Hence, gathering the above inequalities completes the proof of \eqref{es:nabu-Gam2-2a}-\eqref{es:nabu-Gam2-2}. \\

It remains to prove \textbf{(3)}. We treat the case $\mathbf{D} = \mathbb{T}^2$ and $\alpha=1$.
By using \eqref{eq:Gam-Lp-es1}, \eqref{eq:R1-2alp-the1} (with $\alpha=1$) and the commutator estimate \eqref{eq:R-1cm-es2}, we get
\begin{align*}
  \frac{\dd}{\dd t}\|\Gamma(t)\|_{L^p} &\leq \|[\mathcal{R}_{-1}, u\cdot\nabla]\theta(t)\|_{L^p} \\
  & \leq C \|[\mathcal{R}_{-1}, u\cdot\nabla]\theta(t)\|_{B^{1/2}_{p,\infty}}  \\
  & \leq C \|\omega(t)\|_{L^p} \big(\|\theta(t)\|_{B^{-1/2}_{\infty,\infty}} + \|\theta(t)\|_{L^2} \big) \\
  & \leq C \big( \|\Gamma(t)\|_{L^p} + \|\mathcal{R}_{-1}\theta(t)\|_{L^p} \big)
  \|\theta(t)\|_{L^2\cap L^\infty} \\
  & \leq C \big(\|\Gamma(t)|_{L^p} + \|\theta_0\|_{L^1\cap L^\infty} \big) \|\theta_0\|_{L^2\cap L^\infty}.
\end{align*}
Gr\"onwall's inequality ensures that
\begin{align*}
  \|\Gamma\|_{L^\infty_T(L^p)} \leq \big(\|\Gamma_0\|_{L^p} + CT \big) e^{CT} \leq C e^{CT}.
\end{align*}
Following the same idea as the proof of \eqref{eq:Gam-B2p-inf1}-\eqref{eq:nab-u-B1inf} and using \eqref{eq:u-Gam-es4}, we infer that
\begin{align*}
  \|\Gamma\|_{\widetilde{L}^1_T(B^2_{p,\infty})} + \|\Gamma\|_{L^1_T(B^1_{\infty,1})} + \|\nabla u\|_{L^1_T(B^1_{\infty,\infty})}
  \leq C e^{CT}.
\end{align*}
By collecting the above estimates we deduce \eqref{es:nabu-Gam3}.
\end{proof}

Our main result in this section is as follows.
\begin{proposition}\label{prop:C2gam}
Let $\varepsilon =\frac{1}{\mathrm{Pr}} \in (0,1]$, $\alpha\in (\frac{1}{2},1]$.
Let $\mathbf{D}$ be either $\mathbb{R}^2$ or $\mathbb{T}^2$.
Suppose that
\begin{itemize}
 \item $u_0\in H^1\cap\dot{W}^{1,p}(\mathbf{D})$,
\item $\partial_{W_0}u_0\in W^{1,p}(\mathbf{D})$, for some $p>\frac{2}{2\alpha-1}$,
\item $\nabla\cdot u_0=0$,
\item $\theta_0(x) = \bar{\theta}_0(x) \mathbf{1}_{D_0}(x)$,
\end{itemize}
\begin{align*}
  \bar{\theta}_0(x) \in
  \begin{cases}
    C^{2+\gamma-2\alpha}(\overline{D_0}),\quad & \textrm{for}\;\;\gamma\in(0,2\alpha-1), \alpha\in (\frac{1}{2},1], \\
    C^{1+\widetilde{\gamma}}(\overline{D_0}),\,\widetilde{\gamma}>0, \quad
    & \textrm{for}\;\; \gamma = 2\alpha-1,\alpha\in (\frac{1}{2},1),
  \end{cases}
\end{align*}
where $D_0\subset \mathbf{D}$ is a bounded simply connected domain
with boundary $\partial D_0\in C^{2+\gamma}$ for some $\gamma\in (0,2\alpha-1]$ if $\alpha\in (\frac{1}{2},1)$
and for some $\gamma\in(0,1)$ if $\alpha =1$.
Then there exists a unique global solution $(u,\theta)$
to the 2D Boussinesq-Navier-Stokes system $(\mathrm{B}_\alpha)$ which satisfies
\begin{equation*}
  \theta(x, t) = \bar{\theta}_0(X_t^{-1}(x)) \mathbf{1}_{D(t)}(x),
\end{equation*}
with
\begin{equation}\label{eq:reg-parD(t)-1}
  \partial D(t) \in L^{\infty}\left(0, T ; C^{2+\gamma}(\mathbf{D})\right),
\end{equation}
where $D(t)=X_t(D_0)$,  $X_t$ is the particle-trajectory generated by the velocity $u$
and $X^{-1}_t$ is its inverse. In particular, if either $\big\{\alpha\in (\frac{1}{2},1)\big\}$,
or $\big\{\mathbf{D}=\mathbb{T}^2$, $\alpha=1\big\}$, the result \eqref{eq:reg-parD(t)-1}
holds uniformly with respect to $\varepsilon$.
\end{proposition}

\begin{proof}[Proof of Proposition \ref{prop:C2gam}]
In view of Remark \ref{rem:apes1}, it suffices to show the \textit{a priori} estimate on
$L^\infty_T(C^{2+\gamma}(\mathbf{D}))$ of $\varphi$, which is the level-set function of $D(t)$
defined by \eqref{eq:varphi}. Let us denote by $W := \nabla^{\perp} \varphi$ the tangential vector field,  it satisfies
\begin{align}\label{eq:W}
  \partial_t W + u \cdot \nabla W = W \cdot \nabla u := \partial_W u, \quad W|_{t=0} = W_0,
\end{align}
and
\begin{align}\label{eq:nabW}
  \partial_t \nabla W+u \cdot \nabla(\nabla W)
  =\partial_W \nabla u+\nabla W \cdot \nabla u-\nabla u \cdot \nabla W,
  \quad \nabla W|_{t=0} = \nabla W_0.
\end{align}

 Thanks to \eqref{eq:T-BesEs0} and the product estimate $\|f\,g\|_{C^\gamma} \leq C \|f\|_{C^\gamma} \|g\|_{C^\gamma}$,
we obtain that
\begin{equation}\label{eq:nabW-Cgam}
\begin{aligned}
  \|\nabla W(t)\|_{C^\gamma} \leq & C \left\|\nabla W_0\right\|_{C^\gamma}
  + C \int_0^t\left\|\partial_W \nabla u(\tau) \right\|_{C^\gamma} \mathrm{d} \tau
  + C \int_0^t\|\nabla u(\tau)\|_{C^\gamma}\|\nabla W(\tau)\|_{C^\gamma} \mathrm{d} \tau  .
\end{aligned}
\end{equation}
The maximum principle of the equation \eqref{eq:W} also gives
\begin{align}\label{eq:W-Linf-es}
  \|W(t)\|_{L^\infty} \leq \|W_0\|_{L^\infty} + \int_0^t \|\nabla u(\tau)\|_{L^\infty} \|W(\tau)\|_{L^\infty} \dd \tau.
\end{align}

The main goal is to bound the $\partial_W \nabla u$ in $L^1_t(C^\gamma)$. Using the identity \eqref{eq:nab-u-rela},
we find that
\begin{align}\label{eq:parWu-Cgam}
  \|\partial_W \nabla u\|_{L^1_t (C^\gamma)} \leq \big\|\partial_W \nabla \nabla^{\perp} \Lambda^{-2} \Gamma\big\|_{L^1_t(C^\gamma)}
  + \big\|\partial_W \nabla \nabla^{\perp} \partial_1 \Lambda^{-2-2\alpha} \theta \big\|_{L^1_t(C^\gamma)}.
\end{align}
Below we split the proof into three parts according to the domain $\mathbf{D}$ and the value of $\alpha$.
\vskip1mm

\textbf{(1)} First, we deal with the case $\alpha\in(\frac{1}{2},1)$, and we prove the uniform estimates with respect to $\varepsilon$ regardless of the domain. Gr\"onwall's inequality and \eqref{eq:W-Linf-es}, \eqref{eq:u-Gam-es2} give that
\begin{align}\label{eq:W-Linf-es-alp}
  \|W(t)\|_{L^\infty} \leq \|W_0\|_{L^\infty} e^{\int_0^t \|\nabla u(\tau)\|_{L^\infty}\dd \tau} \leq C e^{\exp(Ct)}.
\end{align}
In order to control the term of $\Gamma$ in \eqref{eq:parWu-Cgam}, we prove an estimate of ~$\partial_W \Gamma=W \cdot \nabla \Gamma$.
From \eqref{eq:Gamma} and the fact that $\left[\partial_W, \partial_t+u \cdot \nabla\right]=0$,
we see that $\partial_W \Gamma$ solves the following equation:
\begin{equation}\label{eq:parW}
\begin{aligned}
  \partial_t\left(\partial_W \Gamma\right)+u \cdot \nabla\left(\partial_W \Gamma\right)
  + \frac{1}{\varepsilon} \Lambda^{2\alpha}\left(\partial_W \Gamma\right)
  = & \,\frac{1}{\varepsilon} [\Lambda^{2\alpha},W\cdot\nabla]\Gamma
  + \partial_W \big(\left[\mathcal{R}_{1-2\alpha}, u \cdot \nabla\right] \theta\big).
\end{aligned}
\end{equation}
According to the smoothing estimate \eqref{eq:TD-sm4} in Lemma~\ref{lem:TD-sm},
there exists a constant $C>0$ independent of $\varepsilon$ so that for all $2\alpha-1<\gamma^{\prime}<
\min\{1,4\alpha-2-\frac{2}{p}\}$ (recall that $\frac{2}{p}<2\alpha-1 $), therefore 
\begin{align*}
  \left\|\partial_W \Gamma\right\|_{L_t^1\big( B_{\infty, 1}^{\gamma^{\prime}}\big)}
  \leq  & \, \big\| \big( 2^{q\gamma'} \|\Delta_q (\partial_W \Gamma)\|_{L^1_t(L^\infty)}
  \big)_{q\in\mathbb{N}} \big\|_{\ell^1} +
  \|\Delta_{-1} (\partial_W \Gamma)\|_{L^1_t(L^\infty)} \\
  \leq & \,C\, \varepsilon\, e^{C\int_0^t\|\nabla u(\tau)\|_{L^\infty}\dd \tau}
  \bigg(\left\|\partial_{W_0} \Gamma_0\right\|_{B_{\infty, 1}^{\gamma^{\prime}-2\alpha}}
  + \frac{1}{\varepsilon} \int_0^t\left\|\left[\Lambda^{2\alpha}, W \cdot \nabla\right] \Gamma\right\|_{ B_{\infty, 1}^{\gamma^{\prime}-2\alpha}}\dd \tau \\
  & +\int_0^t\left\|\partial_W\left(\left[\mathcal{R}_{1-2\alpha}, u \cdot \nabla\right] \theta\right)\right\|_{B_{\infty, 1}^{\gamma^{\prime}-2\alpha}}
  \dd \tau\bigg)+ \|\Delta_{-1}\mathrm{div}\,(W\,\Gamma)\|_{L_t^1(L^\infty)}\\
  \leq &\, C  e^{C\int_0^t\|\nabla u(\tau)\|_{L^\infty}\dd \tau}
  \bigg(\varepsilon\left\|\partial_{W_0} \Gamma_0\right\|_{B_{\infty, 1}^{\gamma^{\prime}-2\alpha}}
  + \int_0^t\left\|\left[\Lambda^{2\alpha}, W \cdot \nabla\right] \Gamma\right\|_{ B_{\infty, 1}^{\gamma^{\prime}-2\alpha}}\dd \tau \\
  & +\varepsilon\int_0^t\left\|\partial_W\left(\left[\mathcal{R}_{1-2\alpha},
  u \cdot \nabla\right] \theta\right)\right\|_{B_{\infty, 1}^{\gamma^{\prime}-2\alpha}}
  \dd \tau\bigg) + C\int_0^t\|W(\tau)\|_{L^\infty}\|\Gamma(\tau)\|_{L^\infty}\dd \tau.
\end{align*}
Taking advantage of the identities 
\begin{align*}
  \Gamma_0=\omega_0+\mathcal{R}_{1-2\alpha} \theta_0
  =\omega_0+\partial_1 \Lambda^{-2\alpha} \theta_0,
  \quad \textrm{and}\quad
  \partial_{W_0} \partial_j u_0 = \partial_j (\partial_{W_0} u_0)-\partial_j W_0\cdot \nabla u_0,
\end{align*}
and using the product estimate \eqref{eq:prodBes},
we get
\begin{eqnarray*}
   \left\|\partial_{W_0} \Gamma_0\right\|_{ B_{\infty, 1}^{\gamma^{\prime}-2\alpha}} &\leq&
  \left\|\partial_{W_0} \nabla u_0\right\|_{ B_{\infty, 1}^{\gamma^{\prime}-2\alpha}}
  + \left\|\partial_{W_0} \mathcal{R}_{1-2\alpha} \theta_0\right\|_{B_{\infty, 1}^{\gamma^{\prime}-2\alpha}} \\
  &\lesssim& \left\|\partial_{W_0} u_0\right\|_{B_{\infty, 1}^{\gamma^{\prime}-2\alpha+1}}
  + C \left\|\nabla W_0\right\|_{L^{\infty}}
  \left\|\nabla u_0\right\|_{ B_{\infty, 1}^{\gamma^{\prime}-2\alpha}}
  + C \left\|W_0\right\|_{L^{\infty}}\left\|\nabla \mathcal{R}_{1-2\alpha} \theta_0\right\|_{ B_{\infty, 1}^{\gamma^{\prime}-2\alpha}} \\
  &\lesssim& C \left\|\partial_{W_0} u_0\right\|_{W^{1, p}}+C\left\|\varphi_0\right\|_{W^{2, \infty}}\left\|u_0\right\|_{W^{1, p}}
  + C\left\|\varphi_0\right\|_{W^{1, \infty}}\left\|\theta_0\right\|_{L^2\cap L^\infty} <\infty,
\end{eqnarray*}
where in the last line we have used the embedding that for all $0<\gamma'<4\alpha-2-\tfrac{2}{p}$,
\begin{align*}
  L^p\hookrightarrow B^0_{p,\infty}\hookrightarrow B^{\gamma^{\prime}+2-4\alpha}_{\infty,1} ,
  \quad W^{1,p} \hookrightarrow B^{\gamma'+3 -4\alpha}_{\infty,1} \hookrightarrow B_{\infty, 1}^{\gamma^{\prime}+1-2\alpha},
\end{align*}
and the following estimate
\begin{align*}
  \left\|\nabla \mathcal{R}_{1-2\alpha} \theta_0\right\|_{ B_{\infty, 1}^{\gamma^{\prime}-2\alpha}}
  \leq C \| \theta_0\|_{B^{\gamma^{\prime}+2-4\alpha}_{\infty,1}}
  \leq C\|\theta_0\|_{L^p} .
\end{align*}
Using  \eqref{eq:prodBes} (with $-1<\gamma^{\prime}-2\alpha<1-2\alpha<0$), Lemma \ref{lem:Rbeta-cm}
and \eqref{es:nabu-Gam2}, \eqref{eq:W-Linf-es}, together with the embedding $L^\infty\hookrightarrow B^{\gamma'+2-4\alpha+2/p}_{\infty,1}$,
we find that for all $0<\gamma^{\prime}<4\alpha-2-\frac{2}{p}$,
\begin{align}\label{eq:R-comm-Bes1}
  \left\|\nabla\left(\left[\mathcal{R}_{1-2\alpha}, u \cdot \nabla\right] \theta\right)
  \right\|_{L^1_t\big(B_{\infty, 1}^{\gamma^{\prime}-2\alpha}\big)}
  &\leq C \left\|\left[\mathcal{R}_{1-2\alpha}, u \cdot \nabla\right] \theta
  \right\|_{L^1_t\big(B_{p, 1}^{\gamma^{\prime}+1-2\alpha+\frac{2}{p}}\big)} \nonumber \\
  &\leq C\|\nabla u\|_{L^1_t(L^p)}\Big(
  \|\theta\|_{L^\infty_t\big(B^{\gamma^{\prime}+ 2-4\alpha+\frac{2}{p}}_{\infty,1}\big)}
  +\|\theta\|_{L^\infty_t(L^2)}\Big) \nonumber \\
  &\leq C \|\nabla u\|_{L^1_t(L^p)}\Big( \|\theta_0\|_{L^\infty}+\|\theta_0\|_{L^2}\Big) \leq C e^{Ct},
\end{align}
and
\begin{align*}
  \left\|\partial_W\left(\left[\mathcal{R}_{1-2\alpha}, u \cdot \nabla\right] \theta\right)
  \right\|_{L^1_t\big(B_{\infty, 1}^{\gamma^{\prime}-2\alpha}\big)}
  &\leq C\|W\|_{L^\infty_t(L^\infty)} \left\|\nabla\left[\mathcal{R}_{1-2\alpha}, u \cdot \nabla\right] \theta\right\|_{L^1_t\big(B_{\infty, 1}^{\gamma^{\prime}-2\alpha}\big)} \\
  & \leq C e^{\exp(Ct)} .
\end{align*}

According to \eqref{eq:commEs} in Lemma~\ref{lem:CMX2.5}, we deduce that for all $2\alpha-1<\gamma^{\prime}<1$,
\begin{align*}
  \left\|\left[\Lambda^{2\alpha}, W \cdot \nabla\right] \Gamma\right\|_{B_{\infty, 1}^{\gamma^{\prime}-2\alpha}}
  \leq C \| W\|_{W^{1,\infty}} \|\Gamma\|_{B_{\infty, 1}^{\gamma'}}.
\end{align*}
Hence,  gathering the estimates \eqref{es:nabu-Gam2}, \eqref{eq:W-Linf-es} and the above estimates, we find that
\begin{align*}
  \left\|\partial_W \Gamma\right\|_{L_t^1\left(B_{\infty, 1}^{\gamma^{\prime}}\right)}
  &\leq C e^{C \int_0^t\|\nabla u\|_{L^\infty}\dd \tau}\bigg(1 +
  \int_0^t\| W\|_{W^{1,\infty}}\|\Gamma\|_{B_{\infty, 1}^{\gamma'}}\dd \tau
  + \|W\|_{L_t^\infty(L^{\infty})} \|\nabla u\|_{L_t^1(L^p)}  \bigg )\\
  & \quad + C \|W\|_{L^\infty_t(L^\infty)} \|\Gamma\|_{L^1_t(L^\infty)} \\
  &\leq C e^{\exp (Ct)} \bigg(1+\int_0^t\|W(\tau)\|_{W^{1,\infty}} \|\Gamma(\tau)\|_{B^1_{\infty,1}}\dd \tau\bigg).
\end{align*}
Next, we want to get an estimate of $\partial_W\big(\nabla \nabla^{\perp} \Lambda^{-2} \Gamma\big).$ We have
\begin{eqnarray*}
  \big\|\partial_W\big(\nabla \nabla^{\perp} \Lambda^{-2} \Gamma\big)\big\|_{L_t^1\left(C^\gamma\right)}
  &\leq&   C \big\|\Delta_{-1}\partial_W\big(\nabla \nabla^{\perp} \Lambda^{-2} \Gamma\big)\big\|_{L_t^1\left(L^\infty\right)}
  +  C \big\|\nabla\partial_W\big(\nabla \nabla^{\perp}
  \Lambda^{-2} \Gamma\big)\big\|_{L_t^1\left( C^{\gamma-1}\right)} \\
  &\leq & C \big\|\Delta_{-1} \mathrm{div}\,\big(W\,\nabla \nabla^{\perp} \Lambda^{-2} \Gamma\big)\big\|_{L_t^1\left(L^\infty\right)}
  + C \big\|\nabla W\cdot\nabla\big(\nabla\nabla^{\perp} \Lambda^{-2} \Gamma\big)\big\|_{L_t^1\left( C^{\gamma-1}\right)} \\
  && \ + \ C\big\|\partial_W\big(\nabla^2 \nabla^{\perp} \Lambda^{-2} \Gamma\big)\big\|_{L_t^1\left( C^{\gamma-1}\right)}.
\end{eqnarray*}
The right-hand-side terms of the above inequality can be estimated as follows, thanks to \eqref{eq:u-Gam-es2} and \eqref{eq:W-Linf-es},
\begin{eqnarray*}
  \big\|\Delta_{-1}\partial_W\big(\nabla \nabla^{\perp} \Lambda^{-2} \Gamma\big)\big\|_{L_t^1\left(L^\infty\right)}
  &\leq&  C\|W\|_{L^\infty_t(L^\infty)} \|\nabla\nabla^\perp \Lambda^{-2}\Gamma\|_{L^1_t(L^2)} \\
  &\leq&  C e^{\exp(Ct)},
\end{eqnarray*}
and taking advantage of  \eqref{eq:prodBes} and Lemma~\ref{lem:m(D)},
\begin{eqnarray*}
  \big\|\nabla W\cdot\nabla\big(\nabla\nabla^{\perp} \Lambda^{-2} \Gamma\big)\big\|_{L^1_t(C^{\gamma-1})}
  &\leq &  C\int_0^t \|\nabla W(\tau)\|_{L^\infty} \|\nabla^2\nabla^{\perp} \Lambda^{-2} \Gamma(\tau)\|_{C^{\gamma-1}}  \dd \tau \\
  &\leq &  C \int_0^t \|\nabla W(\tau)\|_{L^\infty} \left(\|\Delta_{-1}\nabla^2\nabla^{\perp} \Lambda^{-2} \Gamma(\tau)\|_{L^{\infty}}
  + \| \Gamma(\tau)\|_{C^{\gamma}}\right) \dd \tau \\
  &\leq&  C \int_0^t \|\nabla W(\tau)\|_{L^\infty} \| \Gamma(\tau)\|_{C^{\gamma}} \dd \tau,
\end{eqnarray*}
then, since we have \eqref{eq:paWmD-es} (with $s=\gamma -1$, $\sigma=1$),
\begin{align*}
  \big\|\partial_W\big(\nabla^2 \nabla^{\perp} \Lambda^{-2} \Gamma\big)\big\|_{L^1_t(C^{\gamma-1})}
  \le C \big\|\partial_W\Gamma\big\|_{L^1_t(C^\gamma)} + C \int_0^t \|W(\tau)\|_{W^{1,\infty}}\|\Gamma(\tau)\|_{C^{\gamma}} \ \dd \tau.
\end{align*}
Collecting the above estimates yields
\begin{equation}\label{eq:parWGam-Bes}
\begin{aligned}
  \big\|\partial_W\big(\nabla \nabla^{\perp} \Lambda^{-2} \Gamma\big)\big\|_{L_t^1\left(C^\gamma\right)}
  & \leq C \left\|\partial_W \Gamma\right\|_{L_t^1\left(B_{\infty, 1}^{\gamma^{\prime}}\right)}
  + C \int_0^t \|W(\tau)\|_{W^{1,\infty}} \|\Gamma(\tau)\|_{C^\gamma} \dd \tau + Ce^{\exp{(Ct)}}\\
  & \leq C e^{\exp (Ct)} \bigg( 1 +  \int_0^t \|W(\tau)\|_{W^{1,\infty}}
  \|\Gamma(\tau)\|_{B^{1}_{\infty,1}} \dd \tau \bigg).
\end{aligned}
\end{equation}

For the estimation of the $\theta$ term in \eqref{eq:parWu-Cgam}, by using \eqref{eq:prodBes}, \eqref{eq:paWmD-es} (with $\sigma =2-2\alpha$, $s=\gamma-1\in (-1, 2\alpha-2]$)
and \eqref{eq:Rbet-L2es}, \eqref{eq:W-Linf-es}, we obtain
\begin{eqnarray}\label{eq:parWthe-Bes}
   \big\|\partial_W\big(\nabla \nabla^{\perp} \partial_1 \Lambda^{-2-2\alpha}
  \theta\big)\big\|_{L_t^1\left(C^\gamma\right)}  &\leq& C \left\|\Delta_{-1} \mathrm{div}\,\big(W\,\nabla \nabla^{\perp} \partial_1 \Lambda^{-2-2\alpha}
  \theta\big)\right\|_{L_t^1\left(L^\infty\right)} \nonumber \\
  && + \left\|\nabla\partial_W\big(\nabla \nabla^{\perp} \partial_1 \Lambda^{-2-2\alpha}
  \theta\big)\right\|_{L_t^1\left(C^{\gamma-1}\right)} \nonumber \\
  &\leq &  C \|W\|_{L_t^\infty(L^\infty)}\left\|\mathcal{R}_{1-2\alpha}
  \theta\right\|_{L_t^1\left(L^2\right)} \nonumber \\
  && +\,   C \int_0^t\|\nabla W(\tau)\|_{L^\infty}\|\nabla^2 \nabla^{\perp} \partial_1 \Lambda^{-2-2\alpha}
  \theta\|_{C^{\gamma-1}}\dd\tau \nonumber\\
  && +\,  C \left\|\partial_W\big(\nabla^2 \nabla^{\perp} \partial_1 \Lambda^{-2-2\alpha}
  \theta\big)\right\|_{L_t^1\left(C^{\gamma-1}\right)} \nonumber \\
  &\leq & C  \left\|\partial_W \theta\right\|_{L_t^1\left(C^{\gamma + 1-2\alpha}\right)}
  + C  \int_0^t\| W\|_{W^{1,\infty}} \|\theta\|_{L^2\cap L^\infty} \dd \tau \nonumber\\
  && + \,C  e^{\exp{(Ct)}},
\end{eqnarray}
where in the last line we also used the estimate
\begin{equation*}
  \left\Vert\nabla^2 \nabla^{\perp} \partial_1 \Lambda^{-2-2\alpha}
  \theta\right\Vert_{C^{\gamma-1}} \leq C  \|\Delta_{-1} \Lambda^{2-2\alpha} \theta\|_{L^2}
  + C \|\theta\|_{B^{\gamma+1-2\alpha}_{\infty,\infty}}
  \leq C \|\theta\|_{L^2\cap L^\infty}.
\end{equation*}
Since $\partial_W\theta$ solves the transport equation
\begin{align}\label{eq:par-W-the}
  \partial_t (\partial_W \theta) + u\cdot \nabla(\partial_W \theta) =0,\quad \partial_W \theta |_{t=0} = \partial_{W_0} \theta_0,
\end{align}
we apply Lemma \ref{lem:CMX} and \eqref{eq:u-Gam-es2} to show that for all $\gamma\in(0,2\alpha-1)$,
\begin{align}\label{eq:parWthe-Bes2}
  \|\partial_W \theta(t)\|_{C^{\gamma + 1 -2\alpha}} \leq C  e^{C \int_0^t \|\nabla u(\tau)\|_{L^\infty}
  \dd \tau} \|\partial_{W_0}\theta_0\|_{C^{\gamma+1-2\alpha}} \leq C  e^{\exp(C t)},
\end{align}
and for $\gamma =2\alpha -1$,
\begin{eqnarray}\label{eq:parWthe-Bes2-2}
  \Vert \partial_W \theta(t)\Vert_{B^{0}_{\infty,\infty}} &\leq &  C e^{C \int_0^t \Vert \nabla u(\tau)\Vert_{L^\infty}  \dd \tau} \Vert \partial_{W_0}\theta_0\Vert_{B^{0}_{\infty,\infty}}\nonumber \\
  &\leq&  C e^{C \int_0^t \|\nabla u(\tau)\|_{L^\infty}
   \dd \tau} \Vert\partial_{W_0}\theta_0\Vert_{L^\infty} \\
  &\leq & C e^{\exp(C t)}, \nonumber
\end{eqnarray}
where $\partial_{W_0} \theta_0 \in C^{\gamma + 1 -2\alpha}$ for $\gamma \in(0,2\alpha -1)$
and $\partial_{W_0} \theta_0 \in L^{\infty}$ for $\gamma =2\alpha-1$ in view of Lemma \ref{lem:str-reg}. \\

Therefore, by collecting the  estimates \eqref{eq:nabW-Cgam}, \eqref{eq:parWu-Cgam} and \eqref{eq:parWGam-Bes}-\eqref{eq:parWthe-Bes2-2}, we find
\begin{align*}
  & \|W(t)\|_{C^{1+\gamma}} + \left\|\partial_W \Gamma\right\|_{L_t^1\big(B_{\infty, 1}^{\gamma^\prime}\big)}
  + \left\|\partial_W \nabla u\right\|_{L_t^1\left(C^\gamma\right)} \\
  \leq & C e^{\exp (Ct)}\bigg(1+ \int_0^t\|W(\tau)\|_{C^{ 1+\gamma}}\left(\|\Gamma(\tau)\|_{B^{1}_{\infty,1}}
  + \|\theta(\tau)\|_{L^2\cap L^\infty}+\|\nabla u(\tau)\|_{C^\gamma}\right)\mathrm{d} \tau\bigg).
\end{align*}
Using Gr\"onwall's inequality combined with \eqref{es:nabu-Gam2}-\eqref{eq:Gam-Lip-es} (together with the embedding $C^{2\alpha-1}\hookrightarrow C^\gamma$,
$\gamma\in (0,2\alpha-1]$) guarantee that
\begin{equation*}
  \|W\|_{L_T^{\infty}\left(C^{1+\gamma}\right)}
  +\left\|\partial_W \Gamma\right\|_{L_T^1\left(B_{\infty, 1}^{\gamma^{\prime}}\right)}
  +\left\|\partial_W \nabla u\right\|_{L_T^1\left(C^\gamma\right)} \leq  C e^{\exp\,\exp(CT)},
\end{equation*}
where $C>0$ is independent of $\varepsilon$.
This implies $W\in L^{\infty}\left([0, T], C^{1+\gamma}\left(\mathbf{D}\right)\right)$  uniformly in $\varepsilon$
for all $0<\gamma\le2\alpha-1$ and $\alpha\in (\frac{1}{2},1)$, this ends the proof in the case $\alpha\in(\frac{1}{2},1)$.
\vskip1mm

\textbf{(2)} We deal with the case of $\alpha=1$ and $\mathbf{D}= \mathbb{R}^2$.
Note that \eqref{eq:W-Linf-es} and \eqref{eq:u-Gam-es3} imply
\begin{align}\label{eq:W-Linf-es-2}
  \|W(t)\|_{L^\infty} \leq \|W_0\|_{L^\infty} e^{\int_0^t \|\nabla u(\tau)\|_{L^\infty}\dd \tau}
  \leq C e^{CE_\varepsilon(t)},\quad \textrm{where}\; \;
  E_\varepsilon(t) := (1+T^4)\Big(1+ \sqrt{\tfrac{T}{\varepsilon}}\Big).
\end{align}
We want to get a control of $\partial_W \Gamma=W \cdot \nabla \Gamma$. Since $\partial_W \Gamma$ solves the following equation,
\begin{align*}
  \partial_t\left(\partial_W \Gamma\right)+u \cdot \nabla\left(\partial_W \Gamma\right)
  -\frac{1}{\varepsilon}\Delta\left(\partial_W \Gamma\right)
  & =-\frac{1}{\varepsilon}\left[\Delta, \partial_W\right] \Gamma
  +\partial_W\left(\left[\mathcal{R}_{-1}, u \cdot \nabla\right] \theta\right) \\
  & =-\frac{1}{\varepsilon}\Delta W \cdot \nabla \Gamma-\frac{2}{\varepsilon} \nabla W: \nabla^2 \Gamma
  + \partial_W\left(\left[\mathcal{R}_{-1}, u \cdot \nabla\right] \theta\right).
\end{align*}
Using the smoothing estimate \eqref{eq:TD-sm4} together with the product estimate \eqref{eq:prodBes}, we find that for all
$\gamma^\prime \in (1, \min \big\{\gamma+1, 2-\frac{2}{p}\big\})$,
\begin{align*}
  &\left\|\partial_W \Gamma\right\|_{L_t^1\big(B_{\infty, 1}^{\gamma^{\prime}}\big)} \\
  & \leq C\varepsilon e^{C\int_0^t\|\nabla u\|_{L^\infty}\dd\tau}
  \left(\left\|\partial_{W_0} \Gamma_0\right\|_{B_{\infty, 1}^{\gamma^{\prime}-2}}
  +\frac{1}{\varepsilon}\int_0^t\Big(\|\Delta W \cdot \nabla \Gamma\|_{B_{\infty, 1}^{\gamma^{\prime}-2}}
  +\left\|\nabla W: \nabla^2 \Gamma\right\|_{B_{\infty, 1}^{\gamma^\prime-2}}\Big)\dd \tau \right.\\
  &\quad \left.+ \left\|\partial_W\left(\left[\mathcal{R}_{-1}, u \cdot \nabla\right] \theta\right)
  \right\|_{L_t^1\left(B_{\infty, 1}^{\gamma^\prime-2}\right)}\right)
  + C  \left\|\Delta_{-1}\partial_W \Gamma\right\|_{L_t^1 (L^\infty)} \\
  & \leq C\varepsilon \,e^{C\int_0^t\|\nabla u\|_{L^\infty}\dd\tau}\left(\left\|\partial_{W_0} \Gamma_0
  \right\|_{B_{\infty, 1}^{\gamma^\prime-2}}
  +\frac{1}{\varepsilon}\int_0^t \Big(\|\Delta W\|_{B^{\gamma^\prime-2}_{\infty,1}}\| \nabla \Gamma\|_{L^\infty}
  + \|\nabla W\|_{L^\infty}\|\nabla^2\Gamma\|_{B_{\infty, 1}^{\gamma^\prime-2}}\Big) \dd \tau \right.\\
  &\quad \left.+ \left\|\partial_W\left(\left[\mathcal{R}_{-1}, u \cdot \nabla\right] \theta\right)
  \right\|_{L_t^1\left(B_{\infty, 1}^{\gamma^\prime-2}\right)}\right)
  + \left\|\Delta_{-1} \mathrm{div} \big(W\,\Gamma\big)\right\|_{L_t^1(L^\infty)} \\
  & \leq e^{C\int_0^t\|\nabla u\|_{L^\infty} \dd\tau}\left(\int_0^t\| W\|_{B^{\gamma^\prime}_{\infty,1}}
  \|\Gamma\|_{B_{\infty, 1}^{\gamma^\prime}} \dd \tau
  + \varepsilon \left\|\partial_W\left(\left[\mathcal{R}_{-1}, u \cdot \nabla\right] \theta\right)
  \right\|_{L_t^1\left(B_{\infty, 1}^{\gamma^\prime-2}\right)}
  +\varepsilon\left\|\partial_{W_0} \Gamma_0\right\|_{B_{\infty, 1}^{\gamma^\prime-2}}\right) \\
  & \quad + C \|W\|_{L^\infty_t(L^\infty)} \Vert \Gamma\Vert_{L^1_t(L^p)}.
\end{align*}
Using again that $\Gamma_0=\omega_0+\mathcal{R}_{-1} \theta_0=\omega_0+\partial_1 \Lambda^{-2} \theta_0$ and the embedding $W^{1, p} \subset B_{\infty, 1}^{\gamma^\prime-1}$ where  $\gamma^\prime$ is such that $0<\gamma^\prime-1<1-\frac{2}{p}$, we get
\begin{align*}
  \left\|\partial_{W_0} \Gamma_0\right\|_{B_{\infty, 1}^{\gamma^\prime-2}}
  & \leq\left\|\partial_{W_0} \nabla u_0\right\|_{B_{\infty, 1}^{\gamma^{\prime}-2}}
  + \left\|\partial_{W_0} \mathcal{R}_{-1} \theta_0\right\|_{B_{\infty, 1}^{\gamma^\prime-2}} \\
  & \leq C\left\|\partial_{W_0} u_0\right\|_{B_{\infty, 1}^{\gamma^\prime-1}}
  + C \left\|\nabla W_0\right\|_{L^{\infty}}\left\|\nabla u_0\right\|_{B_{\infty, 1}^{\gamma^\prime-2}}
  + C\left\|W_0\right\|_{L^\infty}\left\|\nabla \mathcal{R}_{-1} \theta_0\right\|_{B_{\infty, 1}^{\gamma^\prime-2}} \\
  & \leq C\left\|\partial_{W_0} u_0\right\|_{W^{1, p}} + C\left\|\varphi_0\right\|_{W^{2, \infty}}
  \left\|u_0\right\|_{W^{1, p}} + C\left\|\varphi_0\right\|_{W^{1, \infty}}\left\|\theta_0\right\|_{L^2 \cap L^{\infty}}<\infty .
\end{align*}
Then, using the estimates \eqref{eq:prodBes}, \eqref{eq:R-1cm-es1}, \eqref{es:nabu-Gam2-2a}, \eqref{eq:W-Linf-es-2} together with
the  embedding $L^\infty\hookrightarrow B^{\gamma'-2+ 2/p}_{\infty,1}$
for all $1<\gamma^\prime<2-\frac{2}{p}$, we deduce that
\begin{eqnarray*}
  \varepsilon\left\|\partial_W\left(\left[\mathcal{R}_{-1}, u \cdot \nabla\right] \theta\right)
  \right\|_{L_t^1\big(B_{\infty, 1}^{\gamma^\prime-2}\big)}
  &\leq & C  \varepsilon \|W\|_{L_t^\infty(L^\infty)} \left\|\left[\mathcal{R}_{-1}, u \cdot \nabla\right]
  \theta\right\|_{L_t^1 \big(B_{\infty, 1}^{\gamma^\prime -1}\big)} \\
  &\leq&  C \varepsilon e^{C E_\varepsilon(t)} \|[\mathcal{R}_{-1},u\cdot\nabla]\theta \|_{L^1_t\big(B^{\gamma'-1+\frac{2}{p}}_{p,1}\big)} \\
  &\leq&  C e^{CE_\varepsilon(t)} \Big(\|\nabla u\|_{L^1_t(L^p)} \|\theta\|_{L^\infty_t\big(B^{\gamma'-2+\frac{2}{p}}_{\infty,1}\big)}
  + \|u\|_{L^1_t(L^2)} \|\theta\|_{L^\infty_t(L^2)} \Big) \\
  &\leq& C  e^{C E_\varepsilon(t)}.
\end{eqnarray*}
Gathering the above estimates and \eqref{es:nabu-Gam2-2a}, \eqref{eq:W-Linf-es-2} yields
\begin{align*}
  \left\|\partial_W \Gamma\right\|_{L_t^1\left(B_{\infty, 1}^{\gamma^\prime}\right)}
  \leq C e^{CE_\varepsilon(t)} \left(1+\int_0^t\| W(\tau)\|_{B^{\gamma^\prime}_{\infty,1}}
  \|\Gamma(\tau)\|_{B_{\infty, 1}^{\gamma^\prime}}\dd \tau\right).
\end{align*}
Thanks to the striated estimate \eqref{eq:paWmD-es} (with $\sigma=0$), we also infer that
\begin{equation}\label{es:parWGa}
\begin{split}
  \big\|\partial_W\big(\nabla \nabla^{\perp} \Lambda^{-2} \Gamma\big)\big\|_{L_t^1\left(C^\gamma\right)}
  & \leq C\left\|\partial_W \Gamma\right\|_{L_t^1\left(C^\gamma\right)}
  + C\int_0^t\|W(\tau)\|_{W^{1, \infty}} \|\Gamma(\tau)\|_{C^\gamma}\dd\tau \\
  & \leq C e^{CE_\varepsilon(t)} \left(1+\int_0^t\| W(\tau)\|_{B^{\gamma^\prime}_{\infty,1}}
  \|\Gamma(\tau)\|_{B_{\infty, 1}^{\gamma^\prime}} \dd \tau\right) .
\end{split}
\end{equation}
Then, by following the same lines as proof of \eqref{eq:parWthe-Bes} and using \eqref{eq:W-Linf-es-2}, we  get
\begin{align}\label{es:ParWtheta0}
  & \big\|\partial_W\big(\nabla \nabla^\perp \partial_1 \Lambda^{-4} \theta\big)\big\|_{L_t^1\left(C^\gamma\right)} \nonumber \\
  &\leq C \big\|\Delta_{-1} \partial_W\big(\nabla \nabla^{\perp} \partial_1 \Lambda^{-4}
  \theta\big)\big\|_{L_t^1\left(L^\infty\right)}
  + C \big\|\nabla \partial_W\big(\nabla \nabla^{\perp} \partial_1 \Lambda^{-4} \theta\big)
  \big\|_{L_t^1\left(B_{\infty, \infty}^{\gamma-1}\right)} \nonumber \\
  &\leq C \|W\|_{L_t^{\infty}\left(L^{\infty}\right)}\|\theta\|_{L_t^1\left(L^2\right)}
  + C \int_0^t\|\nabla W\|_{L^{\infty}} \big\|\nabla^2 \nabla^{\perp} \partial_1 \Lambda^{-4}
  \theta\big\|_{B_{\infty, \infty}^{\gamma-1}} \dd\tau \nonumber \\
  & \quad + C \big\| \partial_W\big(\nabla^2 \nabla^{\perp} \partial_1 \Lambda^{-4} \theta\big)
  \big\|_{L_t^1\left(B_{\infty, \infty}^{\gamma-1}\right)} \nonumber \\
  &\leq C e^{CE_\varepsilon(t)} + C \int_0^t \|W(\tau)\|_{W^{1,\infty}} \|\theta(\tau)\|_{L^2\cap L^\infty}\dd\tau
  + C \left\|\partial_W \theta\right\|_{L_t^1\left(B_{\infty, \infty}^{\gamma-1}\right)}.
  \end{align}
From the equation \eqref{eq:par-W-the}, we apply Lemma \ref{lem:str-reg} and \eqref{eq:u-Gam-es3} to get that
\begin{eqnarray}\label{es:ParWtheta}
  \left\|\partial_W \theta(t)\right\|_{B_{\infty, \infty}^{\gamma-1}}
  \leq e^{C \int_0^t\|\nabla u\|_{L^{\infty}} \mathrm{d} \tau}\left\|\partial_{W_0} \theta_0\right\|_{C^{\gamma-1}}
  \leq C e^{C E_\varepsilon(t)} .
\end{eqnarray}
Therefore, taking advantages of \eqref{eq:nabW-Cgam}, \eqref{eq:parWu-Cgam}, \eqref{eq:W-Linf-es-2}-\eqref{es:ParWtheta}
and the embedding $C^{\gamma+1} \subset B_{\infty, 1}^{\gamma^{\prime}}$
for all $1<\gamma^\prime< \min\{\gamma+1,2-\frac{2}{p}\}$, we find
\begin{align*}
  &\|W(t)\|_{C^{\gamma+1}}+\left\|\partial_W \Gamma\right\|_{L_t^1\left(B_{\infty, 1}^{\gamma^\prime}\right)}
  + \left\|\partial_W \nabla u\right\|_{L_t^1\left(C^\gamma\right)} \\
  & \leq C e^{C E_\varepsilon(t)}\left(1+ \int_0^t \|W(\tau)\|_{C^{\gamma+1}}
  \Big(1+ \|\Gamma(\tau)\|_{B^{\gamma^\prime}_{\infty,1}}
  +\|\nabla u(\tau)\|_{C^\gamma}\Big)\dd\tau\right) .
\end{align*}
Together with the continuous embedding
$\widetilde{L}^1_t(B^2_{p,\infty})\hookrightarrow L^1_t(B^{\gamma'}_{\infty,1})$,
Gr\"onwall's inequality and the estimates \eqref{eq:u-Gam-es3}, \eqref{es:nabu-Gam2-2} guarantee that
\begin{eqnarray}\label{eq:W-C2+gam-2}
  \|W\|_{L_T^{\infty}\left(C^{1+\gamma}\right)}
  +\left\|\partial_W \Gamma\right\|_{L_T^1\left(B_{\infty, 1}^{\gamma^{\prime}}\right)}
  +\left\|\partial_W \nabla u\right\|_{L_T^1\left(C^\gamma\right)} \nonumber &\leq& C e^{C E_\varepsilon(T)} \nonumber\\
  &\times& \exp\Big\{e^{C E_\varepsilon(T)} \Big(T + \|\Gamma\|_{L^1_T(B^{\gamma'}_{\infty,1})}
  + \|\nabla u\|_{L^1_T(C^\gamma)} \Big) \Big\} \nonumber\\
  &\leq&  C e^{\exp(CE_\varepsilon(T))}
  \end{eqnarray}
where $C>0$ is independent of $\varepsilon$. This implies $W\in L^\infty([0,T], C^{1+\gamma}(\mathbb{R}^2))$ for all $0<\gamma<1$
and $\varepsilon\in (0,1]$.
\vskip1mm


(\textbf{3}) When $\alpha=1$ and $\mathbf{D}=\TT^2$, the proof is quite similar to the  case
$\alpha=1$ and $\mathbf{D}=\RR^2$.
The only difference lies on the use of Lemma \ref{lem:Rbeta-cm} and Propositions \ref{prop:ap-es1}, \ref{prop:ap-es2},
thus by repeating the process in the above by doing small modifications if necessary,
we conclude that \eqref{eq:W-C2+gam-2} holds with $E_\varepsilon(T)$ replaced by $\exp(CT)$.
Hence, we obtain that $W\in L^\infty(0,T;C^{1+\gamma}(\mathbb{T}^2))$ uniformly in $\varepsilon$ for all $0<\gamma<1$.
\vskip1mm

To sum-up, in terms of the notations \eqref{eq:abbr}-\eqref{norm:BBsln-2}, and using Propositions \ref{prop:ap-es1}, \ref{prop:ap-es2},
we have that for  $\alpha\in(\frac{1}{2},1)$ and for all $0<\gamma \le 2\alpha-1 < \gamma' < \min\{1, 4\alpha-2-\frac{2}{p}\}$,
\begin{eqnarray}\label{eq:str-es2-alp}
   \|W\|_{L_T^{\infty}\left(\mathcal{C}_W^{ 1 +\gamma,0}\right)}
  +\|\nabla u\|_{L_T^1\left(\mathcal{C}_W^{\gamma, 1}\right)}
  +\|\Gamma\|_{L_T^1\left(\mathcal{B}_W^{\gamma^{\prime},1}\right)} 
  &=& \|W\|_{L_T^{\infty}\left(C^{ 1+\gamma}\right)}+
  \left\|\left(\nabla u, \partial_W \nabla u\right)\right\|_{L_T^1\left(C^\gamma\right)} \nonumber\\
  && \ + \ \left\|\left(\Gamma, \partial_W \Gamma\right)\right\|_{L_T^1\left(B_{\infty, 1}^{\gamma^{\prime}}\right)} \\
  &\leq& C e^{\exp\,\exp (CT)} \nonumber,
\end{eqnarray}
whereas for $\alpha =1$ and for all $0<\gamma <1 <\gamma' <\min\{\gamma+1,2-\frac{2}{p}\}$,
\begin{equation}\label{eq:str-es2-1}
  \|W\|_{L_T^\infty\left(\mathcal{C}_W^{ 1 +\gamma,0}\right)}
  +\|\nabla u\|_{L_T^1\left(\mathcal{C}_W^{\gamma, 1}\right)}
  +\|\Gamma\|_{L_T^1\left(\mathcal{B}_W^{\gamma^\prime,1}\right)}
  \leq
  \begin{cases}
    C e^{\exp \exp(CT)},\quad & \textrm{for}\;\; \mathbf{D}= \mathbb{T}^2, \\
    C e^{\exp(C E_\varepsilon(T))}, \quad & \textrm{for}\;\; \mathbf{D} = \mathbb{R}^2,
  \end{cases}
\end{equation}
where $C>0$ is independent of $\varepsilon$ and $E_\varepsilon(t)$ is defined in \eqref{eq:W-Linf-es-2}.
This completes the proof of Proposition \ref{prop:C2gam}.
\end{proof}

\section{Persistence of the $C^{k+\gamma}$ boundary regularity with $k\geq 3$}\label{sec:Ckgamma}
In this section, under the assumptions of Theorem \ref{thm:main},
we want to prove that the regularity $C^{k+\gamma}$ of the boundary of initial patch temperature is preserved globally in time. More precisely, we want to prove that $\partial D(t) \in C^{k+\gamma}$ for all $k \geq 3$ and  for some $0<\gamma\le2\alpha-1$ if $\alpha\in(\frac{1}{2},1)$
and for some $\gamma\in(0,1)$ if $\alpha=1$.
In particular, the persistence of $C^{k+\gamma}$-boundary regularity is uniform with respect to $\varepsilon$
for the case that either $\{\alpha\in (\frac{1}{2},1)\}$ or
$\{\mathbf{D}=\mathbb{T}^2, \alpha =1\}$.

Recall that the regularity of the boundary of the patch temperature   $\partial D(t)$ is closely related with the striated regularity
of the (tangential) vector field $W=\nabla^{\perp} \varphi$ with $\nabla^{\perp}=(-\partial_2,\partial_1)^T$.
Indeed, according to \cite{Chem91} (see also \cite{LZ16}), we have that that for $k\geq 3$,
\begin{equation}\label{eq:targ1}
\begin{aligned}
  \partial D(t)\in L^\infty ([0,T], C^{k+\gamma}) \;
  \Longleftrightarrow \; \big(\partial_W^{k-1} W \big)(\cdot, t) \in L^{\infty}\left([0,T], C^\gamma(\mathbf{D})\right);
  \end{aligned}
\end{equation}
and, in particular, to prove the uniform persistence of the $C^{k+\gamma}$ regularity of the patch boundary
it suffices to prove \eqref{eq:targ1} uniformly in $\varepsilon$.

For this purpose, we shall prove that for all $k\geq 3$, $\frac{1}{2}<\alpha<1$, $0<\gamma\le 2\alpha-1$ and $\gamma^\prime$ such that
$ 2\alpha-1<\gamma'<\min\big\{1,4\alpha-2-\tfrac{2}{p},\gamma+2\alpha-1\big\},$
\begin{equation}\label{eq:targ2}
\begin{aligned}
  \|W\|_{L_T^{\infty}\big(\mathcal{C}_{W}^{\gamma+1, k-2}\big)}
  + \|\nabla u\|_{L_T^1\big(\mathcal{C}_{ W}^{\gamma, k-1}\big)}
  + \|\Gamma\|_{L_T^1\big(\mathcal{B}_W^{\gamma^{\prime}, k-1}\big)}  \leq H_{k-1}(T);
\end{aligned}
\end{equation}
whereas for all $k\geq 3$, $\alpha=1$, $0<\gamma<1,$ and
$ 1<\gamma^\prime<\min\big\{\gamma+1, 2 - \frac{2}{p}\big\}$,
\begin{equation}\label{eq:targ3}
\begin{split}
  \|W\|_{L_T^{\infty}\big(\mathcal{C}_{W}^{\gamma+1, k-2}\big)}
  +\|\nabla u\|_{L_T^1\big(\mathcal{C}_{ W}^{\gamma, k-1}\big)}
  + \|\Gamma\|_{L_T^1\big(\mathcal{B}_W^{\gamma^\prime, k-1}\big)}
  \leq
  \begin{cases}
    H_{k-1}(T),\quad & \textrm{for}\;\; \mathbf{D} = \mathbb{T}^2, \\
    H_{k-1}(E_\varepsilon(T)),\quad & \textrm{for} \;\;\mathbf{D}=\mathbb{R}^2,
  \end{cases}
\end{split}
\end{equation}
where $H_{k-1}(T)$ depends on $T$ but is independent of $\varepsilon$ and
$H_{k-1}(E_\varepsilon(T))$ depends on $E_\varepsilon(T)$.
If \eqref{eq:targ2} and \eqref{eq:targ3} are proved then a direct consequence of these controls is that
\begin{equation*}
\begin{aligned}
  \big\|\partial_W^{k-1} W\big\|_{L_T^\infty\left(C^\gamma\right)}
  = \big\|W \cdot \nabla \partial_W^{k-2} W\big\|_{L_T^\infty\left(C^\gamma\right)}
  & \leq C\|W\|_{L_T^\infty \left(C^\gamma\right)}
  \big\|\partial_W^{k-2} W\big\|_{L_T^\infty\left(C^{\gamma+1}\right)} \\
  & \leq C\|W\|_{L_T^\infty \left(C^\gamma\right)}
  \|W\|_{L_T^\infty (\mathcal{C}_W^{\gamma+1,k-2})} <\infty,
\end{aligned}
\end{equation*}
which corresponds to the desired result \eqref{eq:targ1}. Hence it suffices to prove that \eqref{eq:targ2} and \eqref{eq:targ3} hold.
In order to show the estimates \eqref{eq:targ2} and \eqref{eq:targ3}, we apply the induction method. \\

 First we deal with the estimate \eqref{eq:targ2} where $\mathbf{D}$ is  either $\mathbb{R}^2$
or $\mathbb{T}^2$. Assume that for some $\ell \in\{1, \ldots, k-2\}$, we have
\begin{equation}\label{eq:ind-assum}
\begin{aligned}
  \|W\|_{L_T^{\infty}\left(\mathcal{C}_{W}^{\gamma+1, \ell-1}\right)}
  +\|\nabla u\|_{L_T^1\left(\mathcal{C}_{W}^{\gamma, \ell}\right)}
  +\|\Gamma\|_{L_T^1\left(\mathcal{B}_W^{\gamma^{\prime}, \ell}\right)}
  \leq H_{\ell}(T),
\end{aligned}
\end{equation}
we want to prove that it also holds for $\ell$ replaced by $\ell+1$, that is,
\begin{equation}\label{eq:ind-conc}
\begin{aligned}
  \|W\|_{L_T^{\infty}\left(\mathcal{C}_{W}^{\gamma+1,\ell}\right)}+\|\nabla u\|_{L_T^1\left(\mathcal{C}_{W}^{\gamma, \ell+1}\right)}
  + \|\Gamma\|_{L_T^1\left(\mathcal{B}_W^{\gamma^{\prime}, \ell+1}\right)}
  \leq H_{\ell+1}(T) .
\end{aligned}
\end{equation}
The inductive statement is true for $\ell=1$, as a matter of fact, one  notices that \eqref{eq:str-es2-alp} is nothing but \eqref{eq:ind-assum} with $\ell=1$,
and also Lemma \ref{lem:CMXstra-es} and Lemmas \ref{lem:CMX5.1}, \ref{lem:CMX5.2} can be applied with $k= \ell$.
\vskip1mm
We first derive the estimation of the $L_T^\infty\big(C^{\gamma-1}\big)$-norm of $\partial_W^{\ell} \nabla^2 W$.
In view of \eqref{eq:nabW} and the fact that $\left[\partial_W, \partial_t+u \cdot \nabla\right]=0$, we have
\begin{equation}\label{eq:parWel-nab2W}
\begin{aligned}
  \partial_t\big(\partial_W^\ell \nabla^2 W\big) + u \cdot \nabla\big(\partial_W^\ell \nabla^2 W\big)
  = & \, \partial_W^{\ell+1} \nabla^2 u + 2 \partial_W^{\ell}\left(\nabla W \cdot \nabla^2 u\right)
  + \partial_W^{\ell}\left(\nabla^2 W \cdot \nabla u\right) \\
  & -\partial_W^{\ell}\left(\nabla^2 u \cdot \nabla W\right)
  -2 \partial_W^{\ell}\left(\nabla u \cdot \nabla^2 W\right) .
\end{aligned}
\end{equation}
Owing to Lemma~\ref{lem:CMX}, we find that for all $\gamma \in(0,2\alpha-1]$,
\begin{align}\label{eq:4.13}
  \big\|\partial_W^\ell \nabla^2 W(t)\big\|_{C^{\gamma-1}}
  & \leq C e^{C\int_0^t \|\nabla u\|_{L^\infty}\dd\tau} \left( \big\|\partial_{W_0}^\ell \nabla^2 W_0\big\|_{C^{\gamma-1}}
  + \int_0^t\big\|\partial_W^{\ell+1} \nabla^2 u(\tau)\big\|_{C^{\gamma-1}} \mathrm{d} \tau\right. \nonumber \\
  & \quad + \int_0^t \Big\|\Big(\partial_W^\ell\left(\nabla^2 W \cdot \nabla u\right),
  \partial_W^\ell \left(\nabla u \cdot \nabla^2 W\right) \Big)\Big\|_{C^{\gamma-1}} \mathrm{d} \tau \nonumber \\
  & \quad \left.+ \int_0^t\left\|\left(\partial_W^\ell \left(\nabla W \cdot \nabla^2 u\right),
  \partial_W^\ell \left(\nabla^2 u \cdot \nabla W\right)\right)\right\|_{C^{\gamma-1}} \mathrm{d} \tau\right).
\end{align}
Since $\partial D_0\in C^{k+\gamma}(\mathbf{D})$ implying $\varphi_0 \in C^{k+\gamma}\left(\mathbf{D}\right)$,
and by repeatedly  using \eqref{eq:prodBes}, one can get
\begin{align*}
  \big\|\partial_{W_0}^\ell \nabla^2 W_0\big\|_{C^{\gamma-1}}
  & \lesssim_{\|W_0\|_{L^\infty}} \|\nabla\partial_{W_0}^{\ell-1}\nabla^2W_0\|_{C^{\gamma-1}} \nonumber \\
  &\lesssim_{\|W_0\|_{W^{1,\infty}}}  \|\nabla\partial_{W_0}^{\ell-2}\nabla^2W_0\|_{C^{\gamma-1}}
  + \|\nabla^2\partial_{W_0}^{\ell-2}\nabla^2W_0\|_{C^{\gamma-1}}  \nonumber \\
  & \lesssim_{\|W_0\|_{W^{\ell-1,\infty}}} \|\nabla^3W_0\|_{C^{\gamma-1}}
  +\cdots + \|\nabla^{\ell +2} W_0\|_{C^{\gamma-1} } \nonumber \\
  & \lesssim_{\|\varphi_0\|_{W^{\ell,\infty}}} \left\|\varphi_0\right\|_{C^{\ell+2+\gamma}}
  \lesssim_{\|\varphi_0\|_{W^{k-2, \infty}}} \left\|\varphi_0\right\|_{C^{k+\gamma}}.
\end{align*}
In view of Lemma~\ref{lem:CMXstra-es} and the striated estimates \eqref{eq:CMX5.6}, \eqref{eq:CMX5.14},
\eqref{eq:CMX5.16}, the last two integrals on the right-hand side of \eqref{eq:4.13} can be treated as follows:
\begin{equation*}
\begin{aligned}
  & \int_0^t\left\|\left(\partial_W^{\ell}\left(\nabla W \cdot \nabla^2 u\right),
  \partial_W^{\ell}\left(\nabla^2 u \cdot \nabla W\right)\right)\right\|_{C^{\gamma-1}} \mathrm{d} \tau \\
  & \leq \,C\int_0^t\left\|\left(\nabla W \cdot \nabla^2 u,
  \nabla^2 u \cdot \nabla W\right)\right\|_{\mathcal{C}_{W}^{\gamma-1, \ell}} \mathrm{d} \tau \\
  & \le \, C \int_0^t\|\nabla W\|_{\mathcal{B}_W^{0, \ell}}
  \left\|\nabla^2 u\right\|_{\mathcal{C}_{ W}^{\gamma-1, \ell}} \mathrm{d} \tau
  \le C \int_0^t\|W(\tau)\|_{\mathcal{B}_W^{1, \ell}}
  \|\nabla u(\tau)\|_{\mathcal{C}_{W}^{\gamma, \ell}} \mathrm{d} \tau,
\end{aligned}
\end{equation*}
and
\begin{equation*}
\begin{aligned}
  & \int_0^t\left\|\left(\partial_W^{\ell}\left(\nabla^2 W \cdot \nabla u\right),
  \partial_W^{\ell}\left(\nabla u \cdot \nabla^2 W\right)\right)\right\|_{C^{\gamma-1}} \mathrm{d} \tau \\
  &\le C \int_0^t\left\|\left(\nabla^2 W \cdot \nabla u, \nabla u \cdot \nabla^2 W\right)
  \right\|_{\mathcal{C}_{W}^{\gamma-1,\ell}} \mathrm{~d} \tau \\
  & \le C \int_0^t\left\|\nabla^2 W\right\|_{\mathcal{C}_{W}^{\gamma-1, \ell}}
  \|\nabla u\|_{\mathcal{B}_W^{0, \ell}} \mathrm{d} \tau
  \le C \int_0^t\|W(\tau)\|_{\mathcal{C}_{W}^{\gamma+1, \ell}}
  \|\nabla u(\tau)\|_{\mathcal{C}_{W}^{\gamma, \ell}} \mathrm{d} \tau .
\end{aligned}
\end{equation*}
Now the main task is to control the second term on the right-hand side of \eqref{eq:4.13},  and it follows from equality \eqref{eq:u-exp1} that
\begin{equation}\label{eq:nab2u-str-es1}
\begin{aligned}
  \big\|\partial_W^{\ell+1} \nabla^2 u\big\|_{L_t^1\left(C^{\gamma-1}\right)}
  & \leq \big\|\partial_W^{\ell+1} \nabla^2 \nabla^{\perp} \Lambda^{-2}
  \Gamma \big\|_{L_t^1\left(C^{\gamma-1}\right)}
  + \big\|\partial_W^{\ell+1} \nabla^2 \nabla^{\perp} \partial_1
  \Lambda^{-2-2\alpha} \theta\big\|_{L_t^1\left(C^{\gamma-1}\right)} \\
  & \leq \big\|\nabla^2 \nabla^\perp \Lambda^{-2} \Gamma
  \big\|_{L_t^1\big(\mathcal{C}_W^{\gamma-1, \ell+1}\big)}
  + \big\|\nabla^2 \nabla^\perp \partial_1 \Lambda^{-2-2\alpha}
  \theta\big\|_{L_t^1\big(\mathcal{C}_W^{\gamma-1, \ell+1}\big)} .
\end{aligned}
\end{equation}
Taking advantage of \eqref{eq:str-es-mDphi} with $s=\gamma-1$ and $\sigma=2-2\alpha$, we get
\begin{equation*}
\begin{aligned}
  \big\|\nabla^2 \nabla^\perp \partial_1 \Lambda^{-2-2\alpha}
  \theta \big\|_{L_t^1\big(\mathcal{C}_W^{\gamma-1, \ell+1}\big)}
  \leq C \|\theta\|_{L_t^1\left(\mathcal{C}_{W}^{\gamma-2\alpha+1, \ell+1}\right)}
  + C \int_0^t \|W(\tau)\|_{\mathcal{B}_W^{1,\ell}}
  \|\theta(\tau)\|_{\mathcal{C}_W^{\gamma-2\alpha+1, \ell}}\mathrm{d} \tau.
\end{aligned}
\end{equation*}
Since $\partial_W^j\theta$ for all $j\in\{1,\cdots,\ell+1\}$ satisfies
\begin{align*}
  \partial_t(\partial_W^j\theta) + u\cdot\nabla(\partial_W^j\theta) = 0,
\end{align*}
we use \eqref{eq:T-BesEs} and Lemma~\ref{lem:str-reg} to infer that for all $\gamma\in(0,2\alpha-1)$,
\begin{equation*}
  \|\partial_W^j\theta\|_{L_t^\infty(C^{\gamma-(2\alpha-1)})}
  \leq Ce^{C\|\nabla u\|_{L_t^1(L^\infty)}}\|\partial_{W_0}^j\theta_0\|_{C^{\gamma-(2\alpha-1)}}\le Ce^{\exp (Ct)},
\end{equation*}
and for $\gamma=2\alpha-1$,
\begin{equation*}
\begin{aligned}
  \|\partial_W^j\theta\|_{L_t^\infty(B^0_{\infty,\infty})}
  & \leq C e^{C\|\nabla u\|_{L_t^1(L^\infty)}} \|\partial_{W_0}^j\theta_0\|_{B^0_{\infty,\infty}} \\
  & \leq C e^{C\|\nabla u\|_{L_t^1(L^\infty)}} \|\partial_{W_0}^j \theta_0\|_{L^\infty} \leq Ce^{\exp (Ct)}.
\end{aligned}
\end{equation*}
Summing the above inequalities over $j\in\{1,\cdots,\ell+1\}$ leads to that for all $\gamma\in(0,2\alpha-1]$,
\begin{equation}\label{eq:the-str-es}
  \|\theta\|_{L_t^\infty\big(\mathcal{C}^{\gamma-(2\alpha-1),\ell+1}_{W}\big)}
  \leq Ce^{\exp (Ct)}.
\end{equation}
Then, using \eqref{eq:the-str-es} gives
\begin{equation}\label{eq:the-str-es2}
  \|\nabla^2 \nabla^\perp \partial_1 \Lambda^{-2-2\alpha}
  \theta\|_{L_t^1\big(\mathcal{C}^{\gamma-1,\ell+1}_{W}\big)}
  \leq Ce^{\exp(Ct)} + Ce^{\exp(Ct)}\int_0^t\|W(\tau)\|_{\mathcal{B}_W^{1,\ell}}\dd \tau.
\end{equation}
For the first term of the right-hand side of \eqref{eq:nab2u-str-es1}, we use \eqref{eq:str-es-mDphi} (with $s=\gamma-1$, $\sigma=1$)
to deduce that
\begin{equation}\label{eq:Gam-str-es1}
  \big\|\nabla^2 \nabla^{\perp} \Lambda^{-2} \Gamma\big\|_{L_t^1\big(\mathcal{C}_{W}^{\gamma-1, \ell+1}\big)}
  \le C  \|\Gamma\|_{L_t^1\big(\mathcal{C}_{ W}^{\gamma,\ell+1}\big)}
  +C\int_0^t\|W(\tau)\|_{\mathcal{B}_W^{1, \ell}}\|\Gamma(\tau)\|_{\mathcal{C}_{W}^{\gamma, \ell}} \mathrm{d} \tau.
\end{equation}
In the following we consider the smoothing estimate of $\partial_W^{\ell+1} \Gamma$.
From \eqref{eq:Gamma} and the fact $$\big[\partial_W^{\ell+1}, \partial_t+u \cdot \nabla\big]=0,$$ we see that
\begin{equation}\label{eq:parWl+1Gam}
\begin{aligned}
  \partial_t\big(\partial_W^{\ell+1} \Gamma\big) + u \cdot \nabla
  \big(\partial_W^{\ell+1} \Gamma\big) + \frac{1}{\varepsilon} \Lambda^{2\alpha}\big(\partial_W^{\ell+1} \Gamma\big)
  & = \frac{1}{\varepsilon} \big[\Lambda^{2\alpha}, \partial_W^{\ell+1}\big] \Gamma
  + \partial_W^{\ell+1}\big(\left[\mathcal{R}_{1-2\alpha}, u \cdot \nabla\right] \theta\big) \\
  & := F_{\ell+1},
\end{aligned}
\end{equation}
where
\begin{equation}\label{eq:4.23}
\begin{aligned}
  \big[\Lambda^{2\alpha}, \partial_W^{\ell+1}\big] \Gamma
  =\big[\Lambda^{2\alpha}, \partial_W \big] \partial_W^\ell \Gamma
  + \partial_W \big(\big[\Lambda^{2\alpha}, \partial_W^\ell\big] \Gamma\big)
  = \sum_{j=0}^\ell \partial_W^j\big(\big[\Lambda^{2\alpha}, \partial_W\big] \partial_W^{\ell-j} \Gamma\big).
\end{aligned}
\end{equation}
According to \eqref{eq:TD-sm4} in Lemma~\ref{lem:TD-sm}, we infer that for all $2\alpha-1<\gamma^{\prime} <1$,
\begin{eqnarray}\label{eq:parGam-stra-es}
   \big\|\partial_W^{\ell+1} \Gamma\big\|_{L_t^1\big(B_{\infty, 1}^{\gamma^\prime}\big)} &\leq& \Big\| \Big(2^{q\gamma'} \big\|\Delta_q (\partial_W^{\ell+1} \Gamma)\big\|_{L^1_t(L^\infty)}
  \Big)_{q\in\mathbb{N}} \Big\|_{\ell^1}
  + \big\|\Delta_{-1} (\partial_W^{\ell+1} \Gamma) \big\|_{L^1_t (L^\infty)} \nonumber \\
  &\leq& C \varepsilon e^{C\int_0^t\|\nabla u(\tau)\|_{L^\infty}\dd\tau}
  \Big(\big\|\partial_{W_0}^{\ell+1} \Gamma_0\big\|_{B_{\infty, 1}^{\gamma^\prime-2\alpha}}
  + \left\|F_{\ell+1}\right\|_{L_t^1\big(B_{\infty, 1}^{\gamma^\prime-2\alpha}\big)}\Big) \nonumber \\
  && +\big\|\Delta_{-1}(\partial_W^{\ell+1}\Gamma) \big\|_{L_t^1(L^\infty)} \nonumber \\
  &\leq& C e^{\exp(C t)} \left(\varepsilon\big\|\partial_{W_0}^{\ell+1} \Gamma_0 \big\|_{B_{\infty, 1}^{\gamma^\prime-2\alpha}} + \big\| \big[\Lambda^{2\alpha}, \partial_W^{\ell+1}\big] \Gamma
  \big\|_{L_t^1\big(B_{\infty, 1}^{\gamma^\prime-2\alpha}\big)}\right. \nonumber \\
  && \left. + \ \varepsilon\big\|\partial_W^{\ell+1}\left(\left[\mathcal{R}_{1-2\alpha}, u \cdot \nabla\right]
  \theta\right)\big\|_{L_t^1\big(B_{\infty, 1}^{\gamma^\prime-2\alpha}\big)}\right) \nonumber \\
  && + \, C \int_0^t\|W(\tau)\|_{L^\infty} \|\partial_W^\ell\Gamma(\tau)\|_{L^\infty}\dd\tau.
\end{eqnarray}
Taking advantage of the relation $\Gamma_0=\omega_0-\mathcal{R}_{1-2\alpha} \theta_0$ and the equality
\begin{eqnarray}\label{eq:4.25}
  \big[\nabla, \partial_W^{\ell+1}\big] f & = &\left[\nabla, \partial_W\right] \partial_W^{\ell} f + \partial_W \big(\left[\nabla, \partial_W\right] \partial_W^{\ell-1} f\big)
  + \cdots +\partial_W^\ell \big(\left[\nabla, \partial_W\right] f\big) \nonumber \\
  &=&\sum_{j=0}^\ell \partial_W^j\big(\nabla W \cdot \nabla \partial_W^{\ell-j} f\big),
\end{eqnarray}
we apply Lemmas~\ref{lem:CMXstra-es},~\ref{lem:str-reg} and the striated estimates \eqref{eq:CMX5.6}, \eqref{eq:CMX5.14},
\eqref{eq:CMX5.16} to deduce that
\begin{align}\label{eq:4.26}
  &\quad \big\|\partial_{W_0}^{\ell+1} \Gamma_0\big\|_{B_{\infty, 1}^{\gamma^\prime-2\alpha}}
  \leq \big\|\partial_{W_0}^{\ell+1} \nabla u_0\big\|_{B_{\infty, 1}^{\gamma^\prime-2\alpha}}
  + \big\|\partial_{W_0}^{\ell+1} \mathcal{R}_{1-2\alpha} \theta_0\big\|_{B_{\infty, 1}^{\gamma^\prime-2\alpha}} \nonumber \\
  & \leq C \big\|\nabla \partial_{W_0}^{\ell+1} u_0\big\|_{B_{\infty, 1}^{\gamma^\prime-2\alpha}}
  + C \sum_{j=0}^\ell \big\|\nabla W_0 \cdot \nabla \partial_{W_0}^{\ell-j} u_0\big\|_{\mathcal{B}_{W_0}^{\gamma^\prime-2\alpha, j}}
  + C\big\|W_0 \cdot \nabla \mathcal{R}_{1-2\alpha} \theta_0\big\|_{\mathcal{B}_{W_0}^{\gamma^\prime-2\alpha, \ell}} \nonumber \\
  & \leq C \big\|\partial_{W_0}^{\ell+1} u_0\big\|_{W^{1, p}} + C\sum_{j=0}^\ell \left\|\nabla W_0\right\|_{\mathcal{B}_{W_0}^{0, j}}
  \big\|\nabla \partial_{W_0}^{\ell-j} u_0\big\|_{\mathcal{B}_{W_0}^{\gamma^\prime-2\alpha, j}}
  + C \left\|W_0\right\|_{\mathcal{B}_{W_0}^{0, \ell}}
  \left\|\nabla \mathcal{R}_{1-2\alpha} \theta_0\right\|_{\mathcal{B}_{W_0}^{\gamma^\prime-2\alpha, \ell}} \nonumber \\
  & \leq C \big\| \partial_{W_0}^{\ell+1} u_0 \big\|_{W^{1, p}} + C\left\|W_0\right\|_{\mathcal{B}_{W_0}^{1, \ell}}
  \left\|u_0\right\|_{\mathcal{B}_{W_0}^{\gamma^\prime-(2\alpha-1), \ell}}
  + C\left\|W_0\right\|_{\mathcal{B}_{W_0}^{0, \ell}} \Big(1+\left\|W_0\right\|_{\mathcal{B}_{W_0}^{1, \ell-1}}
  \Big) \left\|\theta_0\right\|_{\mathcal{B}_{W_0}^{\gamma^\prime+2-4\alpha, \ell}} \nonumber \\
  & \leq C \Big(1+\left\|W_0\right\|_{\mathcal{B}_{W_0}^{1, \ell}}\Big)
  \bigg(\sum_{j=0}^{\ell+1}\big\|\partial_{W_0}^j u_0\big\|_{W^{1, p}}
  + \sum\limits_{j=0}^\ell \big\|\partial_{W_0}^j\theta_0\big\|_{B^{\gamma'+2-4\alpha}_{\infty,1}}\bigg) \le C ,
\end{align}
where in the last three lines we have used the continuous embedding that
\begin{align*}
  W^{1,p}\hookrightarrow B^{\gamma'+3-4\alpha}_{\infty,1}\hookrightarrow B^{\gamma'+1-2\alpha}_{\infty,1},
  \quad~ L^\infty\hookrightarrow B^{\gamma'+2-4\alpha}_{\infty,1},
  \quad ~B^{\gamma-(2\alpha-1)}_{\infty,\infty}\hookrightarrow B^{\gamma'+2-4\alpha}_{\infty,1},
\end{align*}
valid for all
$$0<\gamma\le 2\alpha-1<\gamma'<\min\big\{1,\gamma+2\alpha-1,4\alpha-2-\tfrac{2}{p}\big\}.$$

For the commutator term $\big[\Lambda^{2\alpha}, \partial_W^{\ell+1}\big] \Gamma$ given by \eqref{eq:4.23},
it follows from \eqref{eq:CMX2.15} (with $s=\gamma'-2\alpha$, $\sigma=2\alpha$), \eqref{eq:CMX5.6} and \eqref{eq:CMX5.14} that
\begin{align}\label{eq:Lamd-2alp-com-Gam}
  \big\|\big[\Lambda^{2\alpha}, \partial_W^{\ell+1}\big] \Gamma\big\|_{L^1_t(B_{\infty, 1}^{\gamma^\prime-2\alpha})}
  & \leq \sum_{j=0}^\ell \big\|\partial_W^j \big([\Lambda^{2\alpha}, W\cdot\nabla] \partial_W^{\ell-j} \Gamma\big)
  \big\|_{L^1_t (B^{\gamma'-2\alpha}_{\infty,1})} \nonumber \\
  & \leq  C \sum_{j=0}^\ell \big\| [\Lambda^{2\alpha}, W\cdot\nabla]\partial_W^{\ell-j} \Gamma
  \big\|_{L^1_t \big(\mathcal{B}^{\gamma'-2\alpha,j}_W \big)} \nonumber \\
  & \leq  C \sum_{j=0}^\ell \int_0^t \Big(\|\nabla W(\tau)\|_{\mathcal{B}^{0,j}_W} + \|W(\tau)\|_{L^\infty} \Big)
  \|\partial_W^{\ell-j}\Gamma(\tau)\|_{\mathcal{B}^{\gamma',j}_W} \dd \tau \nonumber \\
  & \leq C \int_0^t \|W(\tau)\|_{\mathcal{B}_W^{1, \ell}} \|\Gamma(\tau)\|_{\mathcal{B}_W^{\gamma^\prime, \ell}}
  \dd \tau,
\end{align}
where $C>0$ depends on $\|W\|_{L^\infty_T(\mathcal{C}_W^{\gamma+1, \ell-1})}$ which is bounded by
$H_\ell(T)$.
For the second term of $F_{\ell+1}$ in \eqref{eq:parGam-stra-es}, by using the formula
$\nabla\left(\left[\mathcal{R}_{1-2\alpha}, u \cdot \nabla\right] \theta\right)
=\left[\nabla \mathcal{R}_{1-2\alpha}, u \cdot \nabla\right] \theta-(\nabla u) \cdot \nabla \mathcal{R}_{1-2\alpha} \theta$
and \eqref{eq:prodBes}, we see that
\begin{equation*}
\begin{aligned}
  & \, \varepsilon\big\|\partial_W^{\ell+1}\left(\left[\mathcal{R}_{1-2\alpha}, u \cdot \nabla\right]
  \theta\right)\big\|_{L_t^1\big(B_{\infty, 1}^{\gamma^{\prime}-2\alpha}\big)} \\
  & \le C \varepsilon \|W\|_{L_t^{\infty}\left(L^{\infty}\right)}\big\|\nabla \partial_W^\ell
  \big(\left[\mathcal{R}_{1-2\alpha}, u \cdot \nabla\right] \theta\big)
  \big\|_{L_t^1\big(B_{\infty, 1}^{\gamma^\prime-2\alpha}\big)} \\
  & \le C \varepsilon \big\|\partial_W^\ell \left(\left[\nabla \mathcal{R}_{1-2\alpha}, u \cdot \nabla\right]
  \theta\right)\big\|_{L_t^1\left(B_{\infty, 1}^{\gamma^{\prime}-2\alpha}\right)}
  + C \varepsilon \big\|\partial_W^{\ell}\left(\nabla u \cdot \nabla \mathcal{R}_{1-2\alpha} \theta\right)
  \big\|_{L_t^1\big(B_{\infty, 1}^{\gamma^\prime-2\alpha}\big)} \\
  & \quad + C \varepsilon \big\|\big[\nabla, \partial_W^\ell\big]\left(\left[\mathcal{R}_{1-2\alpha}, u \cdot \nabla\right]
  \theta\right)\big\|_{L_t^1\big(B_{\infty, 1}^{\gamma^\prime-2\alpha}\big)} \\
  & := N_1 + N_2 + N_3 .
\end{aligned}
\end{equation*}
To estimate $N_1$, we use  \eqref{eq:CMX2.15} in Lemma~\ref{lem:CMXstra-es}, the induction  assumption \eqref{eq:ind-assum}
together with the estimate \eqref{eq:the-str-es} (with the embedding
$\mathcal{C}^{\gamma-2\alpha+1,\ell+1}_W \hookrightarrow \mathcal{B}^{\gamma'+2-4\alpha,\ell}_W$), we find
\begin{equation*}
\begin{aligned}
  N_1 &\leq C \varepsilon \left\|\left[\nabla \mathcal{R}_{1-2\alpha}, u \cdot \nabla\right]
  \theta\right\|_{L_t^1\big(\mathcal{B}_W^{\gamma^{\prime}-2\alpha, \ell}\big)}\\
  &\leq C \Big(\|\nabla u\|_{L_t^1\big(\mathcal{B}_W^{0, \ell}\big)}
  +\varepsilon\| u\|_{L_t^1\left(L^{\infty}\right)} \Big)
  \|\theta\|_{L_t^{\infty}\big(\mathcal{B}_W^{\gamma^{\prime}+2-4\alpha, \ell}\big)}\\
  &\leq C,
\end{aligned}
\end{equation*}
where in the last line we have used the following estimate (which is a consequence of  \eqref{eq:u-L2-es} and \eqref{eq:u-Gam-es2})
\begin{align*}
  \varepsilon \| u\|_{L_t^1(L^\infty)}
  \leq C \int_0^t\|\varepsilon u\|_{L^2}^{\frac{1}{2}}\|\varepsilon\nabla u\|_{L^\infty}^{\frac{1}{2}}\dd\tau
  \leq C \Big( \|\varepsilon u\|_{L_t^1(L^2)}+\|\varepsilon\nabla u\|_{L_t^1(L^\infty)}\Big) \leq C,
\end{align*}
and $C>0$ independent of $\varepsilon$ depends on $H_{\ell}(T)$ with $H_{\ell}(T)\geq C e^{\exp\exp(CT)}$.
By applying \eqref{eq:CMX2.13}, \eqref{eq:str-es-mDphi},  \eqref{eq:ind-assum}
and \eqref{eq:the-str-es}, the second term $N_2$ can be estimated as
\begin{align*}
  N_2 & \leq C \varepsilon \left\|\nabla u \cdot \nabla \mathcal{R}_{1-2\alpha}
  \theta\right\|_{L_t^1\big(\mathcal{B}_W^{\gamma^{\prime}-2\alpha, \ell}\big)} \nonumber \\
  & \leq C\varepsilon \|\nabla u\|_{L_t^1\big(\mathcal{B}_W^{0, \ell}\big)}
  \left\|\nabla \mathcal{R}_{1-2\alpha} \theta\right\|_{L_t^{\infty}\big(\mathcal{B}_W^{\gamma^{\prime}-2\alpha, \ell}\big)} \nonumber \\
  & \leq C \varepsilon \|\nabla u\|_{L_t^1\big(\mathcal{B}_W^{0, \ell}\big)}
  \Big(\|\theta\|_{L_t^{\infty}\big(\mathcal{B}_W^{\gamma^{\prime}+2-4\alpha, \ell}\big)}
  + \|W\|_{L_t^\infty \big(\mathcal{B}_W^{1, \ell-1}\big)}
  \|\theta\|_{L_t^{\infty} \big(\mathcal{B}_W^{\gamma^{\prime}+2-4\alpha, \ell-1}\big)}\Big) \nonumber \\
  & \leq C.
\end{align*}
For $N_3$, we use \eqref{eq:4.25}, \eqref{eq:ind-assum}, Lemma~\ref{lem:CMXstra-es},
together with the estimates of $N_1$, $N_2$, one obtains that
\begin{align*}
  N_3 & \leq \varepsilon \sum_{i=0}^{\ell-1}\left\|\partial_W^i\left(\nabla W \cdot \nabla \partial_W^{\ell-1-i}
  \left[\mathcal{R}_{1-2\alpha}, u \cdot \nabla\right] \theta\right)\right
  \|_{L_t^1\big(B_{\infty, 1}^{\gamma^{\prime}-2\alpha}\big)} \\
  & \le C \varepsilon \sum_{i=0}^{\ell-1}\left\|\nabla W \cdot \nabla \partial_W^{\ell-1-i}\left[\mathcal{R}_{1-2\alpha},
  u \cdot \nabla\right] \theta\right\|_{L_t^1\big(\mathcal{B}_W^{\gamma^{\prime}-2\alpha, i}\big)} \\
  & \le C \varepsilon \sum_{i=0}^{\ell-1}\|\nabla W\|_{L_t^{\infty}\left(\mathcal{B}_W^{0, i}\right)}
  \left\|\nabla \partial_W^{\ell-1-i}\left[\mathcal{R}_{1-2\alpha}, u \cdot \nabla\right]
  \theta\right\|_{L_t^1\big(\mathcal{B}_W^{\gamma^{\prime}-2\alpha, i}\big)} \\
  & \le C \varepsilon \|W\|_{L_t^{\infty}\big(\mathcal{B}_W^{1, \ell-1}\big)} \sum_{i=0}^{\ell-1}
  \left\|\nabla \partial_W^{\ell-1-i}\left[\mathcal{R}_{1-2\alpha}, u \cdot \nabla\right]
  \theta\right\|_{L_t^1\big(\mathcal{B}_W^{\gamma^{\prime}-2\alpha, i}\big)} \\
  & \le C \varepsilon \left\|\nabla\left(\left[\mathcal{R}_{1-2\alpha}, u \cdot \nabla\right] \theta\right)
  \right\|_{L_t^1\left(\mathcal{B}_W^{\gamma^{\prime}-2\alpha, \ell-1}\right)}
  + C \varepsilon \sum_{i=0}^{\ell-2}\left\|\big[\nabla, \partial_W^{\ell-1-i}\big]\left(\left[\mathcal{R}_{1-2\alpha},
  u \cdot \nabla\right] \theta\right)\right\|_{L_t^1\big(\mathcal{B}_W^{\gamma^{\prime}-2\alpha, i}\big)} \\
  & \leq C+C \varepsilon \sum_{i=0}^{\ell-2} \sum_{j=0}^{\ell-2-i}\left\|\partial_W^j\left(\nabla W \cdot \nabla
  \partial_W^{\ell-2-i-j}\left(\left[\mathcal{R}_{1-2\alpha}, u \cdot \nabla\right]
  \theta\right)\right)\right\|_{L_t^1\big(\mathcal{B}_W^{\gamma^{\prime}-2\alpha, i}\big)} \\
  & \leq C+C \varepsilon \sum_{0 \leq i+j \leq \ell-2}\left\|\nabla W \cdot \nabla
  \partial_W^{\ell-2-i-j}\left(\left[\mathcal{R}_{1-2\alpha}, u \cdot \nabla\right]
  \theta\right)\right\|_{L_t^1\big(\mathcal{B}_W^{\gamma^{\prime}-2\alpha, i+j}\big)} \\
  & \leq C+C \varepsilon \sum_{0 \leq i+j \leq \ell-2}\left\|\nabla \partial_W^{\ell-2-i-j}\left(\left[\mathcal{R}_{1-2\alpha},
  u \cdot \nabla\right] \theta\right)\right\|_{L_t^1\big(\mathcal{B}_W^{\gamma^{\prime}-2\alpha, i+j}\big)},
\end{align*}
where $C>0$ is independent of $\varepsilon$ and depends on $H_{\ell}(t)$ with
$H_{\ell}(T) \geq H_{\ell-1}(T) \geq C e^{\exp\exp(C T)}$.
By repeating the above process and using \eqref{eq:R-comm-Bes1} we end up with
\begin{equation*}
\begin{split}
  N_3 \leq  C+C\varepsilon \left\|\nabla\left(\left[\mathcal{R}_{1-2\alpha}, u \cdot \nabla\right]
  \theta\right)\right\|_{L_t^1\big(B_{\infty, 1}^{\gamma^{\prime}-2\alpha}\big)}
  \leq C.
\end{split}
\end{equation*}
Hence, it follows from \eqref{eq:Gam-str-es1}, \eqref{eq:parGam-stra-es}, \eqref{eq:4.26}, \eqref{eq:Lamd-2alp-com-Gam}
and the above estimates on $N_1$-$N_3$ that
\begin{equation}\label{eq:Gam-Bgam'-el+1}
\begin{aligned}
  \|\Gamma\|_{L_t^1\big(\mathcal{B}_W^{\gamma^{\prime}, \ell+1}\big)}
  =\,  \big\|\partial_W^{\ell+1} \Gamma \big\|_{L_t^1\big(B_{\infty, 1}^{\gamma^{\prime}}\big)}
  + \|\Gamma\|_{L_t^1\big(\mathcal{B}_W^{\gamma^\prime, \ell}\big)}
  \leq C\left(1+\int_0^t\|W(\tau)\|_{\mathcal{B}_W^{1, \ell}}
  \|\Gamma(\tau)\|_{\mathcal{B}_W^{\gamma^{\prime}, \ell}} \mathrm{d} \tau\right),
\end{aligned}
\end{equation}
and (recall that $0<\gamma <\gamma'$)
\begin{equation}\label{eq:Gam-str-es2}
  \big\|\nabla^2 \nabla^{\perp} \Lambda^{-2} \Gamma\big\|_{L_t^1\big(\mathcal{C}_{W}^{\gamma-1, \ell+1}\big)}
  \le C  + C\int_0^t\|W(\tau)\|_{\mathcal{B}_W^{1, \ell}}\|\Gamma(\tau)\|_{\mathcal{B}_{W}^{\gamma', \ell}} \mathrm{d} \tau.
\end{equation}
Then, using \eqref{eq:4.13}, \eqref{eq:the-str-es2} and \eqref{eq:Gam-str-es2}, we find that for all $(\gamma,\gamma')$ such that
$$0<\gamma\leq 2\alpha-1 < \gamma'<\min \big\{1, 4\alpha-2-\tfrac{2}{p},\gamma+2\alpha-1\big\},$$
\begin{equation}\label{eq:4.34}
\begin{aligned}
  \big\|\partial_W^\ell \nabla^2 W(t)\big\|_{C^{\gamma-1}}
  + \big\|\partial_W^{\ell+1} \nabla^2 u\big\|_{L_t^1\left(C^{\gamma-1}\right)}
  \leq C\int_0^t\left(\|\Gamma\|_{\mathcal{B}_W^{\gamma^{\prime}, \ell}}
  +\|\nabla u\|_{\mathcal{C}_{ W}^{\gamma, \ell}}+1\right)\|W\|_{\mathcal{C}_{W}^{\gamma+1, \ell}}\mathrm{d} \tau+C,
\end{aligned}
\end{equation}
where $C>0$ is independent of $\varepsilon$ and depends on $H_{\ell}(T)\geq C e^{\exp\exp{(CT)}}$.
By using \eqref{eq:4.23}, \eqref{eq:4.25} and Lemma~\ref{lem:CMXstra-es} and Lemma~\ref{lem:CMX5.2}, we infer that
\begin{eqnarray*}
  \big\|\big[\nabla^2, \partial_W^\ell\big] W(t)\big\|_{C^{\gamma-1}}
  &\leq& \sum_{j=0}^{\ell-1}\big\|\partial_W^j\big(\nabla^2 W \cdot \nabla \partial_W^{\ell-1-j} W\big)
  \big\|_{L^\infty_t(C^{\gamma-1})} \\
  && \ + \ 2 \sum_{j=0}^{\ell-1}\big\|\partial_W^j\big(\nabla W \cdot\nabla^2 \partial_W^{\ell-1-j} W\big)\big\|_{L^\infty_t(C^{\gamma-1})} \\
  &\leq& \sum_{j=0}^{\ell-1}\big\|\nabla^2 W \cdot \nabla \partial_W^{\ell-1-j} W\big\|_{L^\infty_t\big(\mathcal{C}_W^{\gamma-1, j}\big)}
  +  2\sum_{j=0}^{\ell-1}\big\|\nabla W \cdot \nabla^2 \partial_W^{\ell-1-j} 
  W\big\|_{L^\infty_t\big(\mathcal{C}_W^{\gamma-1, j}\big)} \\
  &\leq& C \sum_{j=0}^{\ell-1} \bigg(\left\|\nabla^2 W\right\|_{L^\infty_t(\mathcal{C}_{W}^{\gamma-1, j})}
  \big\|\nabla \partial_W^{\ell-1-j} W\big\|_{L^\infty_t(\mathcal{B}_W^{0, j})}  \\
  && \ + \   \|\nabla W\|_{L^\infty_t(\mathcal{B}_W^{0, j})}
  \big\|\nabla^2 \partial_W^{\ell-1-j} W\big\|_{L^\infty_t\big(\mathcal{C}_{W}^{\gamma-1, j}\big)}\bigg) \\
  &\leq& C \|W\|_{L^\infty_t\big(\mathcal{C}_{W}^{\gamma+1, \ell-1}\big)} \|W\|_{L^\infty_t\big(\mathcal{B}_W^{1, \ell-1}\big)}
  \Big(1+ \|W\|_{L^\infty_t(\mathcal{B}^{1,\ell-1}_W)} \Big) \\
  &\leq& C,
\end{eqnarray*}
and
\begin{equation*}
\begin{aligned}
  \big\|\big[\nabla, \partial_W^{\ell+1}\big] \nabla u\big\|_{L_t^1\left(C^{\gamma-1}\right)}
  & \leq \sum_{j=0}^{\ell}\big\|\partial_W^j\big(\nabla W \cdot \nabla \partial_W^{\ell-j} \nabla u\big)
  \big\|_{L_t^1\left(C^{\gamma-1}\right)}  \\
  & \leq \sum_{j=0}^\ell \big\|\nabla W \cdot \nabla \partial_W^{\ell-j} \nabla u
  \big\|_{L_t^1\big(\mathcal{C}_{W}^{\gamma-1, j}\big)} \\
  & \leq C \sum_{j=0}^{\ell} \int_0^t\|\nabla W\|_{\mathcal{B}_W^{0, j}}
  \big\|\nabla \partial_W^{\ell-j} \nabla u\big\|_{\mathcal{C}_W^{\gamma-1, j}} \ d \tau  \\
  & \leq C \int_0^t\|W(\tau)\|_{\mathcal{B}_W^{1, \ell}}\|\nabla u(\tau)\|_{\mathcal{C}_{W}^{\gamma, \ell}} \ d \tau .
\end{aligned}
\end{equation*}
Consequently, taking advantage of the following estimates
\begin{align*}
  \|W(t)\|_{\mathcal{C}_{W}^{\gamma+1, \ell}}
  & \leq \big\|\partial_W^\ell W(t) \big\|_{C^{\gamma+1}} + \|W(t)\|_{\mathcal{C}_{W}^{\gamma+1, \ell-1}} \nonumber \\
  & \leq C \big\|\nabla^2 \partial_W^\ell W(t) \big\|_{C^{\gamma-1}}
  + \big\|\Delta_{-1} \partial_W^\ell W(t)\big\|_{L^{\infty}} + C \nonumber \\
  & \le C \big\|\partial_W^\ell \nabla^2 W(t)\big\|_{C^{\gamma-1}}
  + \big\|\big[\nabla^2, \partial_W^\ell\big] W(t)\big\|_{C^{\gamma-1}}
  + \|W(t)\|_{L^\infty} \big\|\partial_W^{\ell-1} W(t)\big\|_{L^\infty} + C \nonumber \\
  & \le C\big\|\partial_W^{\ell} \nabla^2 W(t) \big\|_{C^{\gamma-1}} + C, \nonumber
\end{align*}
and
\begin{eqnarray*}
   \|\nabla u\|_{L_t^1\big(\mathcal{C}_W^{\gamma, \ell+1}\big)}
  &\leq& \big\|\partial_W^{\ell+1} \nabla u\big\|_{L_t^1(C^\gamma)}
  + \|\nabla u\|_{L_t^1\big(\mathcal{C}_W^{\gamma, \ell}\big)} \nonumber \\
  &\leq& C \big\|\nabla \partial_W^{\ell+1} \nabla u \big\|_{L_t^1\left(C^{\gamma-1}\right)}
  +\big\|\Delta_{-1} \partial_W^{\ell+1} \nabla u\big\|_{L_t^1\left(L^{\infty}\right)}+C \nonumber \\
  &\leq& C \big\|\partial_W^{\ell+1} \nabla^2 u\big\|_{L_t^1\left(C^{\gamma-1}\right)}
  + C\big\|\big[\nabla, \partial_W^{\ell+1}\big] \nabla u\big\|_{L_t^1\left(C^{\gamma-1}\right)} \nonumber \\
  && +\, C\|W\|_{L_t^{\infty}\left(L^{\infty}\right)}\big\|\partial_W^\ell \nabla u\big\|_{L_t^1\left(L^{\infty}\right)}+C \nonumber \\
  &\leq& C \big\|\partial_W^{\ell+1} \nabla^2 u\big\|_{L_t^1\left(C^{\gamma-1}\right)}
  + C\int_0^t\|W(\tau)\|_{\mathcal{B}_W^{1, \ell}}\|\nabla u(\tau)\|_{\mathcal{C}_{W}^{\gamma,\ell}} \mathrm{d} \tau+C,
\end{eqnarray*}
and in combination with the estimates \eqref{eq:Gam-Bgam'-el+1}, \eqref{eq:4.34}, we deduce that for all $0<\gamma\leq 2\alpha-1$
and $2\alpha-1<\gamma^\prime<\min \big\{1,4\alpha-2-\frac{2}{p},\gamma+2\alpha-1\big\},$
\begin{equation*}
\begin{aligned}
  & \|W(t)\|_{\mathcal{C}_{W}^{\gamma+1, \ell}}+\|\nabla u\|_{L_t^1\left(\mathcal{C}_{W}^{\gamma, \ell+1}\right)}
  + \|\Gamma\|_{L_t^1\left(\mathcal{B}_W^{\gamma^{\prime}, \ell+1}\right)} \\
  & \leq C\int_0^t\left(\|\Gamma(\tau)\|_{\mathcal{B}_W^{\gamma^{\prime}, \ell}}
  +\|\nabla u(\tau)\|_{\mathcal{C}_{ W}^{\gamma, \ell}} + 1\right)\|W(\tau)\|_{\mathcal{C}_{W}^{\gamma + 1, \ell}}\mathrm{d} \tau+C,
\end{aligned}
\end{equation*}
where $C>0$ depends on $H_{\ell}(T)$ but is independent of $\varepsilon$. Gr\"onwall's inequality and assumption \eqref{eq:ind-assum} guarantee that
\begin{equation*}
\begin{aligned}
  & \|W\|_{L_T^{\infty}\big(\mathcal{C}_W^{\gamma+1, \ell}\big)} + \|\nabla u\|_{L_T^{1}\big(\mathcal{C}_W^{\gamma, \ell+1}\big)}
  + \|\Gamma\|_{L_T^1\big(\mathcal{B}_W^{\gamma^\prime, \ell+1}\big)} \\
  & \leq C \exp \bigg\{C\|\nabla u\|_{L_T^1\big(\mathcal{C}_{W}^{\gamma, \ell}\big)}
  + C\|\Gamma\|_{L_T^1\big(\mathcal{B}_W^{\gamma^{\prime}, \ell}\big)}+C T\bigg\} \leq  H_{\ell+1}(T),
\end{aligned}
\end{equation*}
which corresponds to \eqref{eq:ind-conc}, as desired. Therefore, the estimate \eqref{eq:targ2} is proved.
\vskip1mm

It remains to prove  the estimate \eqref{eq:targ3}. One may follow the same steps as the proof of the estimate \eqref{eq:targ2} up to possibly minor modifications. The main issue is the estimate of $\|[\Delta,\partial_W^{\ell+1}]\Gamma\|_{L^1_t(B^{\gamma'-2}_{\infty,1})}$ but it was already done in \cite{CMX22}  (the notation must be adapted as $\gamma'$ used here corresponds to $\gamma'+1$). For the proof of \eqref{eq:targ3} in the case $\varepsilon =1$ and $\mathbf{D}=\mathbb{R}^2$  we also refer to \cite{CMX22}.
Hence, the proof of Theorem \ref{thm:main} is completed.

\section{Infinite Prandtl number limit in the torus $\mathbb{T}^2$}\label{sec:limit}

This section is devoted to the proof of Theorem \ref{thm:limit} concerning the passage to the limit when the Prandtl number $\mathrm{Pr}$
goes to infinity.
Before that, we present a general convergence result of the system $(\mathrm{B}_\alpha)$
without assuming the temperature patch structure.

\begin{proposition}\label{prop:weak-limit}
For each $\varepsilon \in (0,1]$, consider $(u^\varepsilon, \theta^\varepsilon)$
a global regular solution of the Boussinesq-Navier-Stokes system $(\mathrm{B}_\alpha)$ defined on $\TT^2$ for $\alpha \in (\frac{1}{2},1]$
with uniformly bounded initial data $(u_0^\varepsilon, \theta_0^\varepsilon) \in (H^1 \times L^\infty)(\TT^2)$,
$\nabla\cdot u_0^\varepsilon=0$ and $\int_{\mathbb{T}^2} \theta_0^\varepsilon \dd x =0$.
We suppose that the initial data converge to
$(u_0, \theta_0) \in (H^1 \times L^\infty)(\TT^2)$ as $\varepsilon\rightarrow 0$.

Then, as $\varepsilon\rightarrow 0$, up to extraction of a subsequence, $(u^\varepsilon, \theta^\varepsilon)$
converges to a global unique weak solution $(u, \theta)\in L^\infty([0,T], H^1(\mathbb{T}^2)) \times L^\infty([0,T]\times \mathbb{T}^2)$
of the (fractional) Stokes-transport system $(\mathrm{ST}_\alpha)$ given by \eqref{eq:StokTrans}.
\end{proposition}

\begin{remark}
Concerning the existence and uniqueness of the global weak solution for the (fractional) Stokes-transport system \eqref{eq:StokTrans},
one  refers to the work \cite{Gray23} (see Corollary 7) for the case $\alpha=1$  and
 to \cite{Cobb23} (see Theorem 1.5) for the case $\alpha\in\left(\frac{1}{2},1\right)$.
Note that from the assumption we have $\int_{\TT^2}\theta_0\dd x=0$,
which is the compatibility condition for the system \eqref{eq:StokTrans}.
\end{remark}

\begin{proof}[Proof of Proposition~\ref{prop:weak-limit}]
Since $\widehat{\theta_0^\varepsilon}(0) = \int_{\mathbb{T}^2} \theta_0^\varepsilon(x)\ dx=0$,
from the equation verified by $\theta$ in \eqref{BoussEq} we find
\begin{align*}
  \int_{\mathbb{T}^2} \theta^\varepsilon(x,t) \ dx = \int_{\mathbb{T}^2} \theta^\varepsilon_0(x) \ dx =0,\quad \forall t>0.
\end{align*}
Integrating the evolution equation in $u$ in \eqref{BoussEq} over the spatial variable gives
\begin{align*}
  \varepsilon\frac{\dd}{\dd t} \int_{\mathbb{T}^2} u^\varepsilon(x,t)\ dx = \int_{\mathbb{T}^2} \theta^\varepsilon(x,t) e_2\dd x =0,\quad
  \textrm{that is}, \quad \int_{\mathbb{T}^2} u^\varepsilon(x,t) \ dx = \int_{\mathbb{T}^2} u^\varepsilon_0(x) \ dx.
\end{align*}
Thanks to Poincar\'e's inequality, one gets
\begin{align*}
  \|u^\varepsilon(t)\|_{H^1(\TT^2)}
  &\leq \|\nabla u^\varepsilon(t)\|_{L^2}
  + \Big\|u^\varepsilon(\cdot,t) - \frac{1}{|\mathbb{T}^2|}\int_{\mathbb{T}^2} u^\varepsilon(x,t)\dd x
  \Big\|_{L^2} + C \Big|\int_{\mathbb{T}^2} u^\varepsilon(x,t)\dd x\Big| \\
  & \leq C \left\|\nabla u^\varepsilon(t)\right\|_{L^2(\TT^2)}
  + C \Big|\int_{\mathbb{T}^2} u^\varepsilon_0(x)\dd x\Big|.
\end{align*}
According to Propositions~\ref{lem:eneEs}, \ref{prop:ap-es1},
we have that $u^\varepsilon$ is uniformly bounded in $L_T^\infty(H^1(\TT^2))$ and
$\theta^\varepsilon$ is uniformly bounded in $L_T^\infty(L^\infty(\TT^2))$.
Thus, up to extraction of subsequences, we have the weak-$*$ convergence, as $\varepsilon\rightarrow 0$,
\begin{align}
  & u^\varepsilon \rightharpoonup^* u,~~~ \text{in}~~~ L^\infty_T\big(H^1(\TT^2)\big), \label{eq:u-weak-conv} \\
  &\theta^\varepsilon\rightharpoonup^* \theta,~~~ \text{in}~~~L_T^\infty \big(L^\infty(\TT^2)\big). \nonumber 
\end{align}
It follows from the lower semicontinuity of weak limit that
$u\in L^\infty_T(H^1(\mathbb{T}^2))$ and $\theta\in L^\infty_T(L^\infty(\mathbb{T}^2))$.
From the equation $\partial_t \theta^\varepsilon = -\mathrm{div}\,(u^\varepsilon\,\theta^\varepsilon)$ and
\begin{align*}
  \|\mathrm{div}\,(u^\varepsilon\,\theta^\varepsilon)\|_{L^\infty_T(H^{-1})} \leq C \|u^\varepsilon\,\theta^\varepsilon\|_{L^\infty_T(L^2)}
  \leq C \|u^\varepsilon\|_{L^\infty_T(L^2)} \|\theta^\varepsilon\|_{L^\infty_T(L^\infty)} \leq C,
\end{align*}
we know that $\partial_t \theta^\varepsilon$ is uniformly  bounded in $L^\infty_T(H^{-1}(\mathbb{T}^2))$.
Since $L^2(\mathbb{T}^2) \hookrightarrow H^{-\frac{1}{2}}(\mathbb{T}^2)$ is compact, we use the Aubin-Lions lemma or the Rellich compactness theorem (see Lemma \ref{lem:Aubin-Lions} or \cite{LEM16})
to infer the strong convergence:
\begin{align}\label{eq:the-str-conv}
  \theta^\varepsilon \rightarrow \theta,\quad \textrm{in}\;\; C\big([0,T], H^{-\frac{1}{2}}(\mathbb{T}^2)\big).
\end{align}

Now we can pass the limit to show that $(u,\theta)$ solves the (fractional) Stokes-transport system \eqref{eq:StokTrans}.
We only need to show the convergence of the nonlinear terms, 
since the linear ones can be dealt with in a standard way.
For all $\zeta\in \mathcal{S}(\mathbb{T}^2\times[0,T])$, we have that as $\varepsilon \rightarrow 0$,
\begin{align*}
  \varepsilon \Big|\int_0^T \int_{\mathbb{T}^2} \big(u^\varepsilon \cdot\nabla u^\varepsilon\big) \zeta\dd x \dd t \Big|
  \leq \varepsilon \|u^\varepsilon\|_{L^\infty_T(L^4)} \|\nabla u^\varepsilon\|_{L^\infty_T(L^2)} \|\zeta\|_{L^1_T(L^4)}
  \leq C \varepsilon \|u^\varepsilon\|_{L^\infty_T(H^1)}^2 \|\zeta\|_{L^1_T(L^4)} \rightarrow 0,
\end{align*}
and
\begin{align}\label{eq:converg-nonl}
  & \Big|\int_0^T \int_{\mathbb{T}^2} \mathrm{div}\, (u^\varepsilon\,\theta^\varepsilon)\, \zeta \dd x \dd t -
  \int_0^T \int_{\mathbb{T}^2} \mathrm{div}\, (u\,\theta)\, \zeta \dd x \dd t  \Big| \nonumber \\
  & \leq \Big|\int_0^T \int_{\mathbb{T}^2} (u^\varepsilon - u)\cdot \nabla \zeta \, \theta \, \dd x \dd t\Big|
  + \Big|\int_0^T \int_{\mathbb{T}^2} u^\varepsilon \cdot \nabla \zeta\,(\theta^\varepsilon -\theta) \dd x \dd t \Big| \nonumber \\
  & \leq \Big|\int_0^T \int_{\mathbb{T}^2} (u^\varepsilon - u)\cdot \nabla \zeta \, \theta \, \dd x \dd t\Big|
  + \|u^\varepsilon\cdot \nabla \zeta\|_{L^1_T(H^{\frac{1}{2}})} \|\theta^\varepsilon -\theta\|_{L^\infty_T(H^{-\frac{1}{2}})} \rightarrow 0,
\end{align}
where the convergence in \eqref{eq:converg-nonl} follows from \eqref{eq:u-weak-conv}, \eqref{eq:the-str-conv}
and the estimate that $$\|u^\varepsilon\cdot\nabla \zeta\|_{L^1_T(H^{\frac{1}{2}})} \leq \|u^\varepsilon\cdot\nabla \zeta\|_{L^1_T(H^1)}
\leq C\|u^\varepsilon\|_{L^\infty_T(H^1)} \|\zeta\|_{L^1_T(W^{2,\infty})} \leq C.$$
\end{proof}
\begin{coro}\label{cor:strong-limit}
Under the assumptation of Proposition~\ref{prop:weak-limit}, one can get $\theta^\varepsilon$~and~$\theta$ in Proposition~\ref{prop:weak-limit} satisfies $\theta^\varepsilon\rightarrow\theta$ in $L_T^r(L^q(\TT^2))$ for all $1\leq q,r<\infty$.
\end{coro}
\begin{proof}[Proof of Corollary~\ref{cor:strong-limit}]
Since $\theta^\varepsilon$ is uniformly bounded in $L_T^\infty(L^\infty(\TT^2))$ and $\theta\in L_T^\infty(L^\infty(\TT^2))$, thanks to the interpolation and embedding properties of Lebesgue spaces, we only need to show $\theta^\varepsilon\rightarrow\theta$ in $L_T^2(L^2(\TT^2))$. Since $(u, \theta)\in L^\infty([0,T], H^1(\mathbb{T}^2)) \times L^\infty([0,T]\times \mathbb{T}^2)$ is a global unique weak solution of the (fractional) Stokes-transport system~\eqref{eq:StokTrans}, one gets $\|\theta\|_{L^2(\TT^2)}=\|\theta_0\|_{L^2(\TT^2)}$ and $\|\theta\|_{L^2([0,T]\times\TT^2)}=T^{1/2}\|\theta_0\|_{L^2(\TT^2)}$. Similarly we have $\|\theta^\varepsilon\|_{L^2([0,T]\times\TT^2)}=T^{1/2}\|\theta_0^\varepsilon\|_{L^2(\TT^2)}$. Since $\theta_0^\varepsilon$ converges to $\theta_0$ in $L^\infty(\TT^2)\subset L^2(\TT^2)$, one finds that $\|\theta^\varepsilon\|_{L^2([0,T]\times\TT^2)}$ converges to $\|\theta\|_{L^2([0,T]\times\TT^2)}$ as $\varepsilon\rightarrow 0$. Taking advantages of $\theta^\varepsilon\rightharpoonup \theta$ in $L^2([0,T]\times\TT^2)$, one  gets that $\theta^\varepsilon\rightarrow\theta$ in $L_T^2(L^2(\TT^2))$.
\end{proof}
Now we give the proof of Theorem \ref{thm:limit}.

\begin{proof}[Proof of Theorem \ref{thm:limit}]
Let $X_t^\varepsilon$ be the particle-trajectory generated by the velocity $u^\varepsilon$, that is, $X^\varepsilon_t$
solves
\begin{align}\label{eq:Xt-epsilon}
  \frac{\dd}{\dd t} X_t^\varepsilon(y) = u^\varepsilon\big(X_t^\varepsilon(y),t\big),\quad
  X_t^\varepsilon(y) |_{t=0} = y.
\end{align}
Denote by $X^{\varepsilon,-1}_t$ the inverse map of $X_t^\varepsilon$, then $X^{\varepsilon,-1}_t$  satisfies
\begin{align}\label{eq:Xt-1esps}
  X_t^{\varepsilon,-1}(x) = x - \int_0^t u^\varepsilon\big( \tau, X^\varepsilon_\tau
  \circ X_t^{\varepsilon,-1}(x)\big) \dd \tau.
\end{align}
Thanks to Proposition \ref{prop:ap-es2}, we have the following uniform estimates,
$\nabla u^\varepsilon \in L^1_T(C^{2\alpha -1}(\mathbb{T}^2))$ for $\alpha\in (\frac{1}{2},1)$
and $\nabla u^\varepsilon \in L^1_T(B^1_{\infty,\infty}(\mathbb{T}^2))$ for $\alpha =1$.
Since $$\widehat{u^\varepsilon}(0,t) = \int_{\mathbb{T}^2} u^\varepsilon(x,t)\dd x
= \int_{\mathbb{T}^2} u^\varepsilon_0 (x)\dd x \leq C_0 \|u_0^\varepsilon\|_{L^2}$$ for all $t\in [0,T]$,
we get that $u^\varepsilon \in L^1_T(C^{2\alpha}(\mathbb{T}^2))$ uniformly in $\varepsilon$ if $\alpha\in (\frac{1}{2},1)$ and
$u^\varepsilon\in L^1_T(C^{1+\overline{\gamma}}(\mathbb{T}^2))$, for all $\overline{\gamma}\in (0,1)$
uniformly in $\varepsilon$ if $\alpha=1$.
Therefore, using  \eqref{eq:Xt-C1gam} we find that the system \eqref{eq:Xt-epsilon} has a unique solution
$X_t^\varepsilon(\cdot):\mathbb{T}^2 \rightarrow \mathbb{T}^2$ on $[0,T]$ which is a measure-preserving homeomorphism satisfying
that $X_t^\varepsilon$ and its inverse $X_t^{\varepsilon, -1}$ belong to $L^\infty_T(C^{2\alpha}(\mathbb{T}^2))$ if $\alpha\in (\frac{1}{2},1)$
and belong to $L^\infty_T(C^{1+\overline{\gamma}}(\mathbb{T}^2))$, $\overline{\gamma} \in (0,1)$
if $\alpha=1$.
Besides, from the equation of $\theta$ in \eqref{BoussEq}, it holds that
\begin{align}\label{eq:the-epsil}
  \theta^\varepsilon(x,t) = \bar{\theta}_0\big(X_t^{\varepsilon,-1}(x)\big)
  \mathbf{1}_{D^\varepsilon(t)}(x),\quad\textrm{with}\;\; D^\varepsilon(t) = X_t^\varepsilon(D_0).
\end{align}
By using the equations \eqref{eq:Xt-epsilon} and \eqref{eq:Xt-1esps}, we also know that
$\partial_t X_t^{\varepsilon,\pm 1}  \in L^\infty_T(H^1(\mathbb{T}^2))$ uniformly in $\varepsilon$,
thus Aubin-Lions lemma  guarantees that there exist
$X_t(\cdot) : \mathbb{T}^2 \rightarrow \mathbb{T}^2$ and its inverse
$X_t^{-1}(\cdot):\mathbb{T}^2 \rightarrow \mathbb{T}^2$
such that, as $\varepsilon \rightarrow 0$ and up to the extraction of a subsequence,
\begin{align*}
  X_t^{\varepsilon,\pm 1} \rightarrow X_t^{\pm 1},\quad \textrm{in}\;\;
  C\big([0,T]; C^{1+\widetilde\gamma}(\mathbb{T}^2)\big),\; \widetilde{\gamma} \in (0, 2\alpha-1).
\end{align*}
Moreover, $X_t(\cdot)$ is a measure-preserving homeomorphism that solves the limit equation \eqref{eq:X}
in the sense of distribution, and also $X_t^{\pm1} \in L^\infty_T(C^{2\alpha}(\mathbb{T}^2))$ if $\alpha\in
(\frac{1}{2},1)$ and $X^{\pm1}_t \in L^\infty_T(C^{1+\overline{\gamma}}(\mathbb{T}^2))$,
$\forall \overline{\gamma}\in(0,1)$
if $\alpha=1$. Passing to the limit in \eqref{eq:the-epsil}, we recover that
$\theta(x,t)$ which solves the first equation in \eqref{eq:StokTrans} satisfies
\begin{align*}
  \theta(x,t) = \bar{\theta}_0\big(X_t^{-1}(x)\big) \mathbf{1}_{D(t)}(x),\quad \textrm{with}
  \;\; D(t) = X_t(D_0).
\end{align*}
If $k =1$, the regularity of $X_t^{\pm1}$ implies the global persistence of the regularity $C^{1+\gamma}$ of $\partial D(t)$,
that is, \eqref{eq:limit-targ} with $k=1$ holds.
If  $k\geq 2$, then recalling that $\varphi_0(x)\in C^{k+\gamma}(\mathbb{T}^2)$ satisfying \eqref{eq:D_0}
is the level-set characterization of $D_0$, we have that the domain $D^\varepsilon(t)$ can be characterized by
$\varphi^\varepsilon(x,t) = \varphi_0\big(X^{\varepsilon,-1}_t(x)\big)$ satisfying that
\begin{align*}
  \partial_t\varphi^\varepsilon + u^\varepsilon\cdot\nabla\varphi^\varepsilon =0, \quad
  \varphi^\varepsilon|_{t=0}(x) = \varphi_0(x).
\end{align*}
Since $\varphi^\varepsilon$ is bounded in $L_T^\infty(C^{2+\gamma}(\TT^2))$
(see \eqref{eq:str-es2-alp} and \eqref{eq:str-es2-1})
and $u^\varepsilon$ is controlled in $L_T^\infty(H^1(\TT^2))$,
one gets that $\partial_t\varphi^\varepsilon$ belongs to $L_T^\infty(H^1(\TT^2))$ uniformly in $\varepsilon$.
Since $C^{2+\gamma}(\TT^2)\hookrightarrow C^{2+\gamma_1}(\TT^2)$ ($\gamma_1<\gamma$) is compact,
Lemma~\ref{lem:Aubin-Lions} yields that, up to the extraction of a subsequence,
\begin{align}\label{eq:varphi-conv}
  \varphi^\varepsilon\rightarrow \varphi,\quad \textrm{in}\;\; C\big([0,T],C^{2 + \gamma_1}(\TT^2)\big),
  0< \gamma_1<\gamma.
\end{align}
By letting $\varepsilon\rightarrow 0$ in the equation of $\varphi^\varepsilon$, we find that $\varphi(x,t) = \varphi_0(X_t^{-1}(x))$ is the level-set characterization of the domain $D(t)$
and it solves the equation (in the sense of distribution)
\begin{align*}
  \partial_t\varphi + u\cdot\nabla\varphi = 0,\quad \varphi|_{t=0}(x) = \varphi(x).
\end{align*}
Besides, it follows from the weak-$*$ limit of $\varphi^\varepsilon$ that $\varphi\in L_T^\infty(C^{2+\gamma}(\TT^2))$,
and this proves that the global persistence of the $C^{2+\gamma}$ regularity of $\partial D(t)$.

For the case $k>2$, in light of \eqref{eq:targ1} and \eqref{eq:targ3}, we have $\partial_{W^\varepsilon}^{k-1}
W^\varepsilon\in L^\infty_T(C^\gamma(\mathbb{T}^2))$ uniformly in $\varepsilon$,
with $W^\varepsilon := \nabla^\perp \varphi^\varepsilon$.
From the equation \eqref{eq:W} and the fact that $[\partial_{W^\varepsilon}, \partial_t
+ u^\varepsilon \cdot\nabla] =0$, we get
\begin{align*}
  \partial_t \big(\partial_{W^\varepsilon}^{k-1} W^\varepsilon \big) + u^\varepsilon \cdot \nabla
  \big(\partial_{W^\varepsilon}^{k-1} W^\varepsilon \big) = \partial_{W^\varepsilon}^{k-1}
  \big(W^\varepsilon \cdot\nabla u^\varepsilon \big).
\end{align*}
The uniform estimates \eqref{eq:targ1} and \eqref{eq:targ3} together with \eqref{eq:prodBes} and the striated estimate \eqref{eq:CMX2.13} imply that
\begin{align*}
  \|u^\varepsilon \cdot\nabla \partial_{W^\varepsilon}^{k-1} W^\varepsilon\|_{L^1_T(C^{\gamma-1})}
  \leq C \|u^\varepsilon\|_{L^1_T(L^\infty)} \|\partial_{W^\varepsilon}^{k-1} W^\varepsilon\|_{L^\infty_T(C^\gamma)}
  \leq C,
\end{align*}
and
\begin{align*}
  \|\partial_{W^\varepsilon}^{k-1} (W^\varepsilon\cdot \nabla u^\varepsilon)\|_{L^1_T(C^{\gamma -1})}
  \leq \|W^\varepsilon \cdot \nabla u^\varepsilon\|_{L^1_T(\mathcal{C}^{\gamma-1,k-1}_{W^{\varepsilon}})}
  \leq C \|W^\varepsilon\|_{L^\infty_T(\mathcal{B}^{0,k-1}_{W^\varepsilon})}
  \|\nabla u^\varepsilon\|_{L^1_T(C^{\gamma-1,k-1}_{W^\varepsilon})} \leq C,
\end{align*}
where $C>0$ is independent of $\varepsilon$, thus it follows that $\partial_t \big(\partial_{W^\varepsilon}^{k-1}
W^\varepsilon\big) \in L^1_T(C^{\gamma-1}(\mathbb{T}^2))$ uniformly in $\varepsilon$.
Using the Aubin-Lions lemma ensures that, up to an extraction of a subsequence,
\begin{align}\label{eq:Weps-conv}
  \partial_{W^\varepsilon}^{k-1} W^\varepsilon \rightarrow f_k ,\quad \textrm{in}\;\; L^2([0,T], C^{\gamma_2}(\mathbb{T}^2)),
  \, 0<\gamma_2<\gamma.
\end{align}
We claim that $f_k = \partial_W^{k-1} W$ with $W:= \nabla^\perp \varphi$.
Indeed, if $k=3$, it follows from \eqref{eq:varphi-conv} that $f_3 = \partial_W^2 W$.
Now assume that $f_\ell = \partial_W^{\ell -1} W$ for each $\ell \in \{3,\cdots, k-1\}$,
we shall show that $f_{\ell+1} = \partial_W^\ell W$.
Noting that by using \eqref{eq:varphi-conv} and \eqref{eq:Weps-conv},
we find that for all $\zeta \in \mathcal{S}(\mathbb{T}^2\times[0,T])$,
\begin{align*}
  \int_0^T \int_{\mathbb{T}^2} \big(\partial_{W^\varepsilon}^\ell  W^\varepsilon\big)\, \zeta\, \dd x \dd t
  & = - \int_0^T \int_{\mathbb{T}^2} \big(\partial_{W^\varepsilon}^{\ell-1} W^\varepsilon\big)\,
  \big(\partial_{W^\varepsilon}\zeta\big)\, \dd x \dd t \\
  &\rightarrow -\int_0^T\int_{\mathbb{T}^2} f_\ell \, \big(\partial_W \zeta\big) \,\dd x \dd t
  = \int_0^T \int_{\mathbb{T}^2} (\partial_W f_\ell) \,\zeta \, \dd x \dd t,
\end{align*}
and it follows from the uniqueness of the limit that $f_{\ell +1} = \partial_W f_\ell = \partial_W^\ell W$.
Thus the induction method ensures that $f_k = \partial_W^{k-1}W$, as desired.

Furthermore, the weak-$*$ limit of
$\partial_{W^\varepsilon}^{k-1} W^\varepsilon$
implies that $\partial_W^{k-1}W \in L^\infty_T(C^\gamma(\mathbb{T}^2))$.
In combination with \eqref{eq:targ1}, this shows the global persistence of $C^{k+\gamma}$ regularity of the boundary patch $\partial D(t)$.
Therefore,   the proof of Theorem \ref{thm:limit} is finished.
\end{proof}

\section{Proof of Lemma \ref{lem:CMXstra-es}: striated estimates}\label{sec:apdx}

Let us denote
\begin{align*}
  R_q\left(\alpha_1, \ldots, \alpha_m\right):= \int_{[0,1]^m} \int_{\mathbb{R}^d} \prod_{i=1}^m \alpha_i
  \left(x+f_i(\tau) 2^{-q} y\right) h(\tau, y) \,\dd y\dd\tau,
\end{align*}
where $q \in \mathbb{N}$, $h \in C\left([0,1]^m ; \mathcal{S}\left(\mathbb{R}^d\right)\right)$,
$f_i \in L^{\infty}\left([0,1]^m\right)$ for all $\tau \in(0,1)^m$.
Note that when $f_i \equiv 0$ and $\int_{\mathbb{R}^d} h(\tau, y) \mathrm{d} y=1$,
the identity becomes $R_q\left(\alpha_1, \ldots, \alpha_m\right)=\prod\limits_{i=1}^m \alpha_i(x)$.

First we recall the following important result whose proof  can be found \cite{CMX22}.
\begin{lemma}\label{lem:CMX5.1}
Let $(k,N) \in \mathbb{Z}^+\times \mathbb{Z}^+$, $\rho \in(0,1)$,
and $\mathcal{W}=\left\{W_i\right\}_{1 \leq i \leq N}$ be a set of regular divergence-free vector fields of
$\mathbb{R}^d$ satisfying that
\begin{equation}\label{eq:CMX5.2}
\begin{aligned}
  \|\mathcal{W}\|_{\widetilde{\mathcal{C}}_{\mathcal{W}}^{1+\rho, k-1}}
  := \sum_{\lambda=0}^{k-1}\big\|\left(T_{\mathcal{W} \cdot \nabla}\right)^\lambda
  \mathcal{W}\big\|_{C^{1+\rho}} <\infty.
\end{aligned}
\end{equation}
Let $\alpha_i$ $(i=1, \ldots, m)$ be such that $\operatorname{supp} \widehat{\alpha_i} \subset B\left(0, C_i 2^q\right)$, $C_i\geq 1$,
and let $\psi$ be a smooth function with compact support in a ball.
Then, we have that for all $s \in \mathbb{R},(p, r) \in[1, \infty]^2$ and $\ell \in\{0,1, \ldots, k\}$,
\begin{align}\label{eq:CMX5.3}
  \big\|(T_{\mathcal{W} \cdot \nabla})^\ell R_q\left(\alpha_1, \ldots, \alpha_m\right)\big\|_{L^p}
  \leq C \min _{1 \leq i \leq m}\bigg(\sum_{|\mu| \leq \ell}
  \left\|\left(T_{\mathcal{W\cdot \nabla}}\right)^{\mu_i} \alpha_i\right\|_{L^p} \prod_{1 \leq j \leq m, j \neq i}
  \left\|\left(T_{\mathcal{W} \cdot \nabla}\right)^{\mu_j} \alpha_j\right\|_{L^{\infty}}\bigg),
\end{align}
with $\mu=\left(\mu_1, \ldots, \mu_m\right)$ and $|\mu|=\mu_1+\cdots+\mu_m$, and
\begin{equation}\label{eq:CMX5.4}
  \big\|(T_{\mathcal{W} \cdot \nabla})^\ell \psi\left(2^{-q} D\right) \phi\big\|_{L^p}
  \leq C \sum_{\lambda=0}^\ell \big\|(T_{\mathcal{W} \cdot \nabla})^\lambda \phi \big\|_{L^p},
\end{equation}
and
\begin{align}\label{eq:CMX5.5}
  \Big\|\left(2^{q s}\big\|\left(T_{\mathcal{W} \cdot \nabla}\right)^\ell \Delta_q \phi
  \big\|_{L^p}\right)_{q \geq-1}\Big\|_{\ell^r}
  + \Big\|\left(2^{q(s-1)}\big\|\left(T_{\mathcal{W} \cdot \nabla}\right)^\ell \nabla \Delta_q \phi\big\|_{L^p}
  \right)_{q \geq-1}\Big\|_{\ell^r} \leq C\|\phi\|_{\widetilde{\mathcal{B}}_{p, r, \mathcal{W}}^{s, \ell}},
\end{align}
and
\begin{equation}\label{eq:CMX5.6}
  \|\nabla \phi\|_{\widetilde{\mathcal{B}}_{p, r, \mathcal{W}}^{s, \ell}}
  \le C\|\phi\|_{\widetilde{\mathcal{B}}_{p, r, \mathcal{W}}^{s+1, \ell}} .
\end{equation}
In the above the positive constant C depends on $\|\mathcal{W}\|_{\widetilde{\mathcal{C}}_{ \mathcal{W}}^{1+\rho, k-1}}$.
\end{lemma}

Based on Lemma~\ref{lem:CMX5.1}, we obtain the following useful striated estimates. 
\begin{lemma}\label{lem:CMX5.2}
Let $\mathcal{W}=\left\{W_i\right\}_{1 \leq i \leq N}$ be a set of regular divergence-free vector fields of
$\mathbb{R}^d$ satisfying \eqref{eq:CMX5.2} with $(k,N) \in \mathbb{Z}^+\times \mathbb{Z}^+$, $\rho \in(0,1)$.
Let $m(D):= \Lambda^\sigma m_0(D)$, $\sigma>-1$, and $m_0(D)$ be a zero-order pseudo-differential operator with
$m_0(\xi) \in C^{\infty}\left(\mathbb{R}^d \backslash\{0\}\right)$. Then, there exists a positive constant $C$ depending on
$\|\mathcal{W}\|_{\widetilde{\mathcal{C}}_{W}^{1+\rho, k-1}}$ such that the following statements hold
for all $\ell \in\{0,1, \ldots, k\}$.
\begin{enumerate}
\item We have that for all $q \geq-1$,
\begin{equation}\label{eq:CMX5.7}
  \big\|\Delta_q\left(T_{\mathcal{W} \cdot \nabla}\right)^\ell \nabla m(D) \phi\big\|_{L^p}
  \leq C \sum_{\lambda=0}^\ell 2^{q(1+\sigma)} \big\|(T_{\mathcal{W} \cdot \nabla})^\lambda \phi\big\|_{L^p},
\end{equation}
and for all $q \in \mathbb{N}$,
\begin{equation}\label{eq:CMX5.8}
  2^q\big\|(T_{\mathcal{W} \cdot \nabla})^\ell \Delta_q \phi\big\|_{L^p}
  \leq C \sum_{q_1 \in \mathbb{N},\left|q_1-q\right| \leq N_{\ell}} \sum_{\lambda=0}^\ell
  \big\|(T_{\mathcal{W} \cdot \nabla})^\lambda \Delta_{q_1} \nabla \phi\big\|_{L^p},
\end{equation}
with $N_{\ell} \in \mathbb{N}$ depending only on $\ell$.

\item For all $s<0$, we have
\begin{equation}\label{eq:CMX5.9}
  \left\|T_v w\right\|_{\widetilde{\mathcal{B}}_{p, r, \mathcal{W}}^{s, \ell}}
  \leq C \min \bigg\{\sum_{\mu=0}^{\ell}\|v\|_{\widetilde{\mathcal{B}}_{\mathcal{W}}^{0, \mu}}
  \|w\|_{\widetilde{\mathcal{B}}_{p, r, \mathcal{W}}^{s, \ell-\mu}},
  \sum_{\mu=0}^{\ell}\|v\|_{\widetilde{\mathcal{B}}_{p, r, \mathcal{W}}^{s, \mu}}
  \|w\|_{\widetilde{\mathcal{B}}_{\mathcal{W}}^{0, \ell-\mu}}\bigg\},
\end{equation}
and
\begin{equation}\label{eq:CMX5.10}
  \left\|T_v w\right\|_{\widetilde{\mathcal{B}}_{p, r, \mathcal{W}}^{0, \ell}}
  \leq C \sum_{\mu=0}^\ell \|v\|_{\widetilde{\mathcal{B}}_{\mathcal{W}}^{0, \mu}}
  \|w\|_{\widetilde{\mathcal{B}}_{p, r, \mathcal{W}}^{0, \ell-\mu}}.
\end{equation}
While for all $s<1$,
\begin{equation}\label{eq:CMX5.11}
  \left\|T_{\nabla w} v\right\|_{\widetilde{\mathcal{B}}_{p, r, \mathcal{W}}^{s, \ell}}
  \leq C \sum_{\mu=0}^\ell \|w\|_{\widetilde{\mathcal{C}}_{\mathcal{W}}^{s, \mu}}
  \|v\|_{\widetilde{\mathcal{B}}_{p, r, \mathcal{W}}^{1, \ell-\mu}},
\end{equation}
and for all $s \in \mathbb{R}$,
\begin{equation}\label{eq:CMX5.12}
  \left\|T_{\nabla w} v\right\|_{\widetilde{\mathcal{B}}_{p, r, \mathcal{W}}^{s, \ell}}
  \leq C \sum_{\mu=0}^{\ell}\|\nabla w\|_{\widetilde{\mathcal{B}}_{\mathcal{W}}^{0, \mu}}
  \|v\|_{\widetilde{\mathcal{B}}_{p, r, \mathcal{W}}^{s, \ell-\mu}}.
\end{equation}

\item Assume that $v$ is a divergence-free vector field of $\mathbb{R}^d$,
then we have that for all $s>-1$,
\begin{equation}\label{eq:CMX5.13}
   \|R(v\cdot,\nabla w)\|_{\widetilde{\mathcal{B}}_{p, r, \mathcal{W}}^{s, \ell}} \lesssim \min \bigg\{\sum_{\mu=0}^{\ell}\|v\|_{\widetilde{\mathcal{B}}_{\mathcal{W}}^{0, \mu}}
  \|\nabla w\|_{\widetilde{\mathcal{B}}_{p, r, \mathcal{W}}^{s, \ell-\mu}},
  \sum_{\mu=0}^{\ell}\|v\|_{\widetilde{\mathcal{B}}_{p, r, \mathcal{W}}^{s, \mu}}
  \|\nabla w\|_{\widetilde{\mathcal{B}}_{\mathcal{W}}^{0, \ell-\mu}},
  \sum_{\mu=0}^{\ell}\|v\|_{\widetilde{\mathcal{B}}_{\mathcal{W}}^{1, \mu}}
  \|w\|_{\widetilde{\mathcal{B}}_{p, r, \mathcal{W}}^{s, \ell-\mu}}\bigg\}.
\end{equation}
\item We have that
\begin{equation}\label{eq:CMX5.14}
  \|\phi\|_{\mathcal{B}_{p, r, \mathcal{W}}^{s, \ell}}
  \leq C\|\phi\|_{\widetilde{\mathcal{B}}_{p, r, \mathcal{W}}^{s, \ell}}
  \leq C\|\phi\|_{\mathcal{B}_{p, r, \mathcal{W}}^{s, \ell}},\quad \forall\, s\in (-1,1),
\end{equation}

\begin{equation}\label{eq:CMX5.15}
  \|\phi\|_{\mathcal{B}_{\mathcal{W}}^{1, \ell}} \le C\|\phi\|_{\widetilde{\mathcal{B}}_{\mathcal{W}}^{1, \ell}}
  \le C\|\phi\|_{\mathcal{B}_{\mathcal{W}}^{1, \ell}},
\end{equation}

\begin{equation}\label{eq:CMX5.16}
  \|\mathcal{W}\|_{\widetilde{\mathcal{B}}_{p, r, \mathcal{W}}^{s, \ell}}
  \leq C\|\mathcal{W}\|_{\mathcal{B}_{p, r, \mathcal{W}}^{s, \ell}},\quad \forall \,s>-1,
\end{equation}
and,
\begin{equation}\label{eq:CMX5.17}
  \|\phi\|_{\widetilde{\mathcal{B}}_{p, r, \mathcal{W}}^{s, \ell}} \le C\|\phi\|_{\mathcal{B}_{p, r, \mathcal{W}}^{s, \ell}}
  + C\|\phi\|_{\mathcal{B}_{\mathcal{W}}^{1, \ell}}\|\mathcal{W}\|_{\mathcal{B}_{p, r, \mathcal{W}}^{s, \ell}},
  \quad \forall\, s\geq 1.
\end{equation}
\end{enumerate}
\end{lemma}

\begin{proof}[Proof of Lemma \ref{lem:CMX5.2}]
The only statement which needs to be proved is \eqref{eq:CMX5.7} since the other estimates are the same as Lemma 5.2 in \cite{CMX22}.
We prove it via the induction method. We first remark that \eqref{eq:CMX5.7} for $\ell=0$ is true: this follows from Lemma \ref{lem:m(D)}.
Assume that it holds for all $\ell^\prime\in\{0,\cdots,\ell\}$ with some $\ell\in\{0,1,\cdots,k-1\}$,
we want to prove that \eqref{eq:CMX5.7} is true at the rank $(\ell+1)$.
Similarly as \eqref{eq:comm-formu1}, we notice that
\begin{align*}
  \left(T_{\mathcal{W} \cdot \nabla}\right)\nabla m(D) f
  & = - \big[\nabla m(D), T_{\mathcal{W}\cdot\nabla} \big] f + \nabla m(D) \left(T_{\mathcal{W}\cdot\nabla}\right)f \\
  & =-\sum_{q_1 \in \mathbb{N}}\big[\nabla m(D), S_{q_1-1} \mathcal{W} \cdot \nabla\big] \Delta_{q_1} f
  + \nabla m(D)\left(T_{\mathcal{W} \cdot \nabla}\right) f\\
  & =-\sum_{q_1 \in \mathbb{N}}\left[\nabla m(D) \widetilde{\psi}\left(2^{-q_1} D\right), S_{q_1-1} \mathcal{W} \cdot \nabla\right]
  \Delta_{q_1} f + \nabla m(D)\left(T_{\mathcal{W} \cdot \nabla}\right) f,
\end{align*}
and
\begin{align*}
  & {\left[\nabla m(D) \widetilde{\psi}\left(2^{-q_1} D\right), S_{q_1-1} \mathcal{W} \cdot \nabla\right] \Delta_{q_1} f } \\
  & = 2^{q_1 (d+ 1+ \sigma)} \int_{\mathbb{R}^d} \widetilde{h}_1(2^{q_1}y)
  \big( S_{q_1-1}\mathcal{W}(x-y) - S_{q_1 -1}\mathcal{W}(x) \big) \cdot \nabla \Delta_{q_1} f(x-y) \dd y \\
  & = 2^{q_1(1+\sigma)} \int_{\mathbb{R}^d} \widetilde{h}_1(y)\left(S_{q_1-1} \mathcal{W}\left(x-2^{-q_1} y\right)-S_{q_1-1}
  \mathcal{W}(x)\right)  \cdot \nabla \Delta_{q_1} f\left(x-2^{-q_1} y\right) \mathrm{d} y \\
  & = - 2^{q_1\sigma}\int_0^1 \int_{\mathbb{R}^d} \widetilde{h}_1(y) y \cdot \nabla S_{q_1-1}
  \mathcal{W}\left(x-\tau2^{-q_1} y\right) \cdot \nabla \Delta_{q_1} f\left(x-2^{-q_1} y\right) \mathrm{d} y \mathrm{~d} \tau,
\end{align*}
with $\widetilde{h}_1 := \mathcal{F}^{-1}( i \xi m(\xi)\widetilde{\psi}(\xi) ) \in \mathcal{S}(\mathbb{R}^d)$.
Thus by using the induction assumption together with Lemma~\ref{lem:CMX5.1}, we infer that for all $q \geq-1$,
\begin{eqnarray*}
  \big\|\Delta_q\left(T_{\mathcal{W} \cdot \nabla}\right)^{\ell+1} \nabla m(D) \phi\big\|_{L^p} &\leq& \big\|\Delta_q\left(T_{\mathcal{W} \cdot \nabla}\right)^\ell
  \left(\left[\nabla m(D), T_{\mathcal{W} \cdot \nabla}\right] \phi\right)\big\|_{L^p} \ + \ \big\|\Delta_q\left(T_{\mathcal{W} \cdot \nabla}\right)^\ell \nabla m(D)\left(T_{\mathcal{W} \cdot \nabla} \right)\phi\big\|_{L^p}\\
  &\lesssim& \sum_{q_1 \in \mathbb{N}, q_1 \sim q} \sum_{\mu_1+\mu_2 \leq \ell} 2^{q_1\sigma}
  \big\|(T_{\mathcal{W} \cdot \nabla})^{\mu_1} S_{q_1-1} \nabla \mathcal{W}\big\|_{L^{\infty}}
  \big\|(T_{\mathcal{W} \cdot \nabla})^{\mu_2} \nabla \Delta_{q_1} \phi\big\|_{L^p} \\
&& \ + \  2^{q(1+\sigma)} \sum_{\lambda=0}^\ell \big\|(T_{\mathcal{W} \cdot \nabla})^{\lambda+1} \phi\big\|_{L^p} \\
  &\lesssim& \sum_{\lambda=0}^\ell 2^{q(1+\sigma)}\big\|(T_{\mathcal{W} \cdot \nabla})^\lambda \phi\big\|_{L^p} +  \sum_{\lambda=1}^{\ell+1} 2^{q(1+\sigma)}\big\|(T_{\mathcal{W} \cdot \nabla})^\lambda \phi\big\|_{L^p} \\
  &\lesssim&  \sum_{\lambda=0}^{\ell+1} 2^{q(1+\sigma)} \big\|(T_{\mathcal{W} \cdot \nabla})^\lambda \phi\big\|_{L^p},
\end{eqnarray*}
where in the last line we have used \eqref{eq:CMX5.4} and the following estimate
\begin{align*}
  \left\|\left(T_{\mathcal{W} \cdot \nabla}\right)^{\mu_{1}} \nabla S_{q_1-1} \mathcal{W}\right\|_{L^{\infty}}
  & \leq \sum_{q_2=-1}^{q_1-1} 2^{-q_2 \rho}\Big(2^{q_2 \rho}\left\|\left(T_{\mathcal{W} \cdot \nabla}\right)^{\mu_{1}}
  \Delta_{q_2} \nabla \mathcal{W}\right\|_{L^{\infty}}\Big) \\
  & \leq C\|\nabla \mathcal{W}\|_{\widetilde{\mathcal{C}}_{\mathcal{W}}^{\rho, \mu_1}}
  \leq C\|\mathcal{W}\|_{\widetilde{\mathcal{C}}_{\mathcal{W}}^{1+\rho, \mu_{1}}} < \infty.
\end{align*}
Hence, by induction, we have proved  \eqref{eq:CMX5.7} which is the wanted control.
\end{proof}

Then, we turn to the proof of Lemma~\ref{lem:CMXstra-es}.
\begin{proof}[Proof of Lemma~\ref{lem:CMXstra-es}]
Since \eqref{eq:CMX2.13} was already proved in Lemma 2.4 in \cite{CMX22},
we only need to prove statements $(\mathrm{ii})$ and $(\mathrm{iii})$.
We shall prove \eqref{eq:CMX2.15} and \eqref{eq:str-es-mDphi} again by induction on the index $k$.
For $k=0$, \eqref{eq:CMX2.15} follows directly from \eqref{eq:commEs},
while \eqref{eq:str-es-mDphi} follows from \eqref{eq:paWmD-es} together with the estimate
\begin{align}\label{eq:mDphi-Bes}
  \|m(D) \phi\|_{B^s_{p,r}} & \leq C \|\Delta_{-1} m(D)\phi\|_{L^p} +
  C \big\| \big(2^{js} \|\Delta_j m(D)\phi\|_{L^p}\big)_{j\in\mathbb{N}}\big\|_{\ell^r} \nonumber \\
  & \leq C \|\Delta_{-1} m(D) \phi\|_{L^p} +
  C \big\| \big(2^{j(s+\sigma)} \|\Delta_j \phi\|_{L^p}\big)_{j\in\mathbb{N}}\big\|_{\ell^r} \nonumber \\
  & \leq C \mathbf{1}_{\{-1<\sigma\leq 0\}} \|\Delta_{-1}m(D)\phi\|_{L^p}
  + C \|\phi\|_{B^{s+\sigma}_{p,r}},
\end{align}
where $C>0$ is a universal constant
(the norm $\|\mathcal{W}\|_{\mathcal{C}_{\mathcal{W}}^{1+\rho, k-1}}$ plays no role).

Assume that \eqref{eq:CMX2.15} and \eqref{eq:str-es-mDphi} holds for $\ell \in\{0,1, \ldots, k-1\}$ (where $k=0$ when $\ell=0$ ) with $\ell$ in place of the $k$-index, we intend to prove that they also hold for the $\ell+1$ case. For the estimation of \eqref{eq:CMX2.15}, thanks to Bony's decomposition, we have
\begin{eqnarray*}
  [m(D), u \cdot \nabla] \phi &=& \sum_{j \in \mathbb{N}} \left[m(D), S_{j-1} u \cdot \nabla\right] \Delta_j \phi
  + \sum_{j \in \mathbb{N}}\left[m(D), \Delta_j u \cdot \nabla\right] S_{j-1} \phi \\
  &&+ \sum_{j \geq 3} m(D) \operatorname{div}\big(\Delta_j u \widetilde{\Delta}_j \phi\big)
  - \sum_{j \geq 3} \operatorname{div}\big(\Delta_j u ~ m(D) \widetilde{\Delta}_j \phi\big) \\
  &&+ \sum_{-1 \leq j \leq 2} \left[m(D), \Delta_j u \cdot \nabla\right] \widetilde{\Delta}_j \phi \\
  &:=&\sum_{j=1}^5  I_j.
\end{eqnarray*}
It follows from \eqref{eq:CMX5.14} that
\begin{equation*}
  \|[m(D), u \cdot \nabla] \phi\|_{\mathcal{B}_{p, r, \mathcal{W}}^{s,\ell+1}}
  \leq C \big\|\big(I_1,I_2,I_3,I_4,I_5\big)\big\|_{\widetilde{\mathcal{B}}_{p, r, \mathcal{W}}^{s, \ell+1}} .
\end{equation*}
For ${I}_1$, noticing that $m(D) \widetilde{\psi}\left(2^{-j} D\right)=2^{j(d+\sigma)} \widetilde{h}\left(2^j \cdot\right) *$
with $\widetilde{h}:= \mathcal{F}^{-1}(m \widetilde{\psi}) \in \mathcal{S}\left(\mathbb{R}^d\right)$, and
\begin{equation*}
\begin{aligned}
  \big[m(D), S_{j-1} u \cdot & \nabla\big] \Delta_j \phi(x)
  = \big[m(D) \widetilde{\psi}\left(2^{-j} D\right), S_{j-1} u\cdot\big] \nabla \Delta_j \phi(x) \\
  & = 2^{j\sigma}\int_{\mathbb{R}^d} \widetilde{h}(y)
  \left(S_{j-1} u \left(x-2^{-j} y\right)-S_{j-1} u(x)\right) \cdot \nabla \Delta_j \phi\left(x-2^{-j} y\right) \mathrm{d} y \\
  & = - 2^{j(\sigma-1)} \int_0^1 \int_{\mathbb{R}^d} \widetilde{h}(y) y \cdot
  \nabla S_{j-1} u\left(x-\tau 2^{-j} y\right) \cdot \nabla \Delta_j \phi\left(x-2^{-j} y\right) \dd y\dd\tau,
\end{aligned}
\end{equation*}
we apply Lemma~\ref{lem:CMX5.1} and Lemma~\ref{lem:CMX5.2} to obtain that for all $\lambda \in\{0,1, \ldots, \ell+1\}$,
\begin{eqnarray*}
  2^{qs}\big\|\Delta_q\left(T_{\mathcal{W} \cdot \nabla}\right)^\lambda {I}_1\big\|_{L^p} &\lesssim&  2^{qs} \sum_{j \in \mathbb{N}, j \sim q} \left\|\Delta_q (T_{\mathcal{W} \cdot \nabla})^\lambda
  \big(\big[m(D) \widetilde{\psi}\left(2^{-j} D\right), S_{j-1}u \cdot \nabla\big] \Delta_j \phi\big)\right\|_{L^p} \\
  &\lesssim&  2^{qs} \sum_{j \in \mathbb{N}, j \sim q} \sum_{\lambda_1+\lambda_2 \leq \lambda}
  2^{j(\sigma-1)} \big\|\left(T_{\mathcal{W}\cdot \nabla}\right)^{\lambda_1} \nabla S_{j-1} u
  \big\|_{L^\infty} \big\|\left(T_{\mathcal{W} \cdot \nabla}\right)^{\lambda_2} \nabla \Delta_j \phi\big\|_{L^p} \\
  &\lesssim&  \sum_{j \in \mathbb{N}, j \sim q} \sum_{\lambda_1+\lambda_2 \leq \lambda}
  \bigg(\sum_{j^\prime \leq j-1}\big\|\left(T_{\mathcal{W}\cdot \nabla}\right)^{\lambda_1} \nabla \Delta_{j^\prime}
  u\big\|_{L^\infty}\bigg) 2^{j(\sigma+s-1)}\big\|(T_{\mathcal{W}\cdot \nabla})^{\lambda_2} \Delta_j \nabla \phi\big\|_{L^p} \\
 &\lesssim& c_q \sum_{\lambda_1=0}^{\ell+1} \|\nabla u\|_{\widetilde{\mathcal{B}}_{\mathcal{W}}^{0, \lambda_1}}
  \|\nabla \phi\|_{\widetilde{\mathcal{B}}_{p, r, \mathcal{W}}^{\sigma-1+s, \ell+1-\lambda_1}} \\
  &\lesssim& c_q \sum_{\lambda_1=0}^{\ell+1} \|\nabla u\|_{\widetilde{\mathcal{B}}_{\mathcal{W}}^{0, \lambda_1}}
  \|\phi\|_{\widetilde{\mathcal{B}}_{p, r, \mathcal{W}}^{s+\sigma, \ell+1-\lambda_1}} \\
  &\lesssim& \|\nabla u\|_{\mathcal{B}_{\mathcal{W}}^{0, \ell+1}} \|\phi\|_{\mathcal{B}_{p, r, \mathcal{W}}^{s+\sigma, \ell+1}},
\end{eqnarray*}
with $\left\{c_q\right\}_{q \geq-1}$ satisfying $\left\|c_q\right\|_{\ell^r}=1$.
It immediately leads to
\begin{equation*}
  \|{I}_1\|_{\widetilde{\mathcal{B}}_{p, r, \mathcal{W}}^{s, \ell+1}}
  \leq C \|\nabla u\|_{\mathcal{B}_{\mathcal{W}}^{0, \ell+1}}\|\phi\|_{\mathcal{B}_{p, r, \mathcal{W}}^{\sigma+s, \ell+1}}.
\end{equation*}
Then, we use the strategy as the proof of \eqref{eq:comm-formu1}, that is, we  decompose $I_2$ as
\begin{equation*}
\begin{aligned}
  {I_2} 
  = \sum\limits_{j\in\NN} m(D) \widetilde{\psi}(2^{-j}D)\big(\Delta_j u\cdot\nabla S_{j-1}\phi\big)
  - \sum\limits_{j\in\NN}\Delta_j u\cdot\nabla m(D)S_{j-1} \phi
  := {I}_{2,1}+{I}_{2,2}.
\end{aligned}
\end{equation*}
For ${I}_{2,1}$, by using the induction assumptions and taking advantages of
\eqref{eq:CMX5.3}, \eqref{eq:CMX5.6}, \eqref{eq:CMX5.8} and \eqref{eq:CMX5.14},
we find that for all $\lambda \in\{0,1, \ldots, \ell+1\}$,
\begin{eqnarray*}
  2^{qs}\big\|\Delta_q\left(T_{\mathcal{W} \cdot \nabla}\right)^\lambda {I}_{2,1}\big\|_{L^p} &\lesssim& 2^{qs} \sum_{j \in \mathbb{N}, j \sim q}\left\|\Delta_q\left(T_{\mathcal{W}\cdot \nabla}\right)^\lambda
  m(D)\widetilde{\psi}(2^{-j}D)(\Delta_j u \cdot \nabla S_{j-1} \phi)\right\|_{L^p} \\
 &\lesssim& 2^{qs} \sum_{j \in \mathbb{N}, j \sim q}\sum\limits_{\lambda_1=0}^{\lambda}
  2^{j\sigma}\big\|(T_{\mathcal{W}\cdot \nabla})^{\lambda_1} (\Delta_j u \cdot \nabla S_{j-1} \phi)\big\|_{L^p}\\
  &\lesssim& \sum_{j \in \mathbb{N}, j \sim q}\sum\limits_{\lambda_2+\lambda_3\leq\lambda}
  2^{j(s+\sigma)} \big\|(T_{\mathcal{W}\cdot \nabla})^{\lambda_2} \Delta_j u\big\|_{L^\infty}
  \big\| (T_{\mathcal{W}\cdot \nabla})^{\lambda_3} \nabla S_{j-1} \phi\big\|_{L^p} \\
  &\lesssim&  \sum_{j \in \mathbb{N}, j \sim q} \sum\limits_{\lambda_2+\lambda_3\leq\lambda} 2^{j(s+\sigma-1)}
  \bigg(\sum_{q_1 \in \mathbb{N}, q_1 \sim j}\sum_{\lambda_4=0}^{\lambda_2}
  \big\|(T_{\mathcal{W}\cdot \nabla})^{\lambda_4}
  \Delta_{q_1}\nabla u\big\|_{L^\infty}\bigg) \\
  && \times \ \bigg(\sum\limits_{j^\prime\leq j-1}\big\|  (T_{\mathcal{W}\cdot \nabla})^{\lambda_3} \nabla \Delta_{j^\prime}\phi\big\|_{L^p}\bigg)\\
&\lesssim&  \sum_{j \in \mathbb{N}, j \sim q} \sum_{\lambda_4+\lambda_3 \leq \lambda}
  \|\nabla u\|_{\widetilde{\mathcal{B}}^{0,\lambda_4}_{\mathcal{W}}}
  \sum_{j^{\prime} < j-1} 2^{\left(j-j^\prime\right)(s+\sigma-1)} 2^{j^\prime(s+\sigma-1)}
  \big\|(T_{\mathcal{W} \cdot \nabla})^{\lambda_3} \nabla \Delta_{j^\prime} \phi \big\|_{L^p} \\
  &\lesssim& c_q\|\nabla u\|_{\widetilde{\mathcal{B}}_{\mathcal{W}}^{0, \ell+1}}
  \sum_{\lambda_3=0}^\lambda\|\nabla \phi\|_{\widetilde{\mathcal{B}}_{p, r, \mathcal{W}}^{\sigma-1+s, \lambda_3}} \\
  &\lesssim& c_q\|\nabla u\|_{\widetilde{\mathcal{B}}_{\mathcal{W}}^{0, \ell+1}}
  \|\phi\|_{\widetilde{\mathcal{B}}_{p, r, \mathcal{W}}^{\sigma+s, \ell+1}} \\
  &\lesssim& c_q\|\nabla u\|_{\mathcal{B}_{\mathcal{W}}^{0, \ell+1}}
  \|\phi\|_{\mathcal{B}_{p, r, \mathcal{W}}^{\sigma+s, \ell+1}},
\end{eqnarray*}
and
\begin{eqnarray*}
  2^{qs}\big\|\Delta_q\left(T_{\mathcal{W} \cdot \nabla}\right)^\lambda {I}_{2,2}\big\|_{L^p} &\lesssim& 2^{qs} \sum_{j \in \mathbb{N}, j \sim q}\big\|\Delta_q\left(T_{\mathcal{W}\cdot \nabla}\right)^\lambda
  \big(\Delta_j u\cdot\nabla m(D)S_{j-1} \phi\big) \big\|_{L^p} \\
   &\lesssim& 2^{qs} \sum_{j \in \mathbb{N}, j \sim q}
  \sum\limits_{\lambda_1+\lambda_2\leq\lambda} \big\|(T_{\mathcal{W}\cdot \nabla})^{\lambda_1}
  \Delta_j u\big\|_{L^\infty} \big\|(T_{\mathcal{W}\cdot \nabla})^{\lambda_2}
  \nabla m(D)S_{j-1} \phi \big\|_{L^p} \\
   &\lesssim& \sum_{j \in \mathbb{N}, j \sim q} \sum\limits_{\lambda_1+\lambda_2\leq\lambda} 2^{j(s-1)}
  \bigg(\sum_{q_1 \in \mathbb{N}, q_1 \sim j} \sum_{\lambda_3=0}^{\lambda_1}
  \big\|(T_{\mathcal{W}\cdot \nabla})^{\lambda_3}
  \Delta_{q_1} \nabla u\big\|_{L^\infty} \bigg) \times \\
   && \ \bigg( \big\| (T_{\mathcal{W}\cdot \nabla})^{\lambda_2}
  \nabla m(D) \Delta_{-1} \phi\big\|_{L^p}  + \sum_{0\leq j'\leq j-1}
  \big\|(T_{\mathcal{W}\cdot\nabla})^{\lambda_2} \nabla m(D) \Delta_{j'}\phi \big\|_{L^p} \bigg) \\
   &\lesssim&  \sum_{j \in \mathbb{N}, j \sim q} \sum_{\lambda_2+\lambda_3 \leq \lambda}
  \|\nabla u\|_{\widetilde{\mathcal{B}}^{0,\lambda_3}_{\mathcal{W}}}
  2^{j(s-1)}\bigg(\|\nabla m(D) \Delta_{-1} \phi\|_{L^p}  \\
  && \ + \ 
  \sum_{0\leq j^\prime \leq j-1} \sum_{\lambda_4=0}^{\lambda_2} 2^{j'(1+\sigma)}
  \big\|(T_{\mathcal{W} \cdot \nabla})^{\lambda_4} \Delta_{j'} \phi \big\|_{L^p} \bigg) \\
   &\lesssim& \|\nabla u\|_{\widetilde{\mathcal{B}}^{0,\ell+1}_\mathcal{W}}
   \\
  &\times& \sum_{j \in \mathbb{N},j\sim q} 2^{j(s-1)}  \bigg(\|\Delta_{-1} \phi\|_{L^p} +
  \sum_{\lambda_4=0}^\lambda \sum_{0\leq j^\prime \leq j-1} 2^{j'(1-s)}  2^{j'(s+\sigma)}
  \big\|(T_{\mathcal{W} \cdot \nabla})^{\lambda_4} \Delta_{j'} \phi \big\|_{L^p} \bigg) \\
   &\lesssim& c_q \|\nabla u\|_{\widetilde{\mathcal{B}}_{\mathcal{W}}^{0, \ell+1}}
  \sum_{\lambda_2=0}^\lambda \|\phi\|_{\widetilde{\mathcal{B}}_{p, r, \mathcal{W}}^{s+\sigma, \lambda_2}} \\
  &\lesssim&  c_q\|\nabla u\|_{\widetilde{\mathcal{B}}_{\mathcal{W}}^{0, \ell+1}}
  \|\phi\|_{\widetilde{\mathcal{B}}_{p, r, \mathcal{W}}^{s+\sigma, \ell+1}} \\
   &\lesssim& c_q\|\nabla u\|_{\mathcal{B}_{\mathcal{W}}^{0, \ell+1}}
  \|\phi\|_{\mathcal{B}_{p, r, \mathcal{W}}^{s+\sigma,\ell+1}}
\end{eqnarray*}
where $\left\{c_q\right\}_{q \geq-1}$ satisfies $\left\|c_q\right\|_{\ell^r}=1$.
Then the above estimates readily give 
\begin{equation*}
  \|{I_2}\|_{\widetilde{\mathcal{B}}_{p, r, \mathcal{W}}^{s, \ell+1}}
  \leq C\|\nabla u\|_{\mathcal{B}_{\mathcal{W}}^{0, \ell+1}}
  \|\phi\|_{\mathcal{B}_{p, r, \mathcal{W}}^{s+\sigma,\ell+1}}.
\end{equation*}
For ${I}_3$, by applying Lemma~\ref{lem:CMX5.1} and Lemma~\ref{lem:CMX5.2},
we find that for all $\lambda \in\{0,1, \ldots, \ell+1\}$,
\begin{eqnarray*}
   2^{qs}\big\|\Delta_q\left(T_{\mathcal{W}\cdot \nabla}\right)^\lambda {I}_3\big\|_{L^p} &\lesssim& 2^{qs} \sum_{j \geq \max \left\{3, q-N_\lambda\right\}}\left\|\Delta_q\left(T_{\mathcal{W}\cdot \nabla}
  \right)^\lambda \nabla m(D)\big(\Delta_j u \widetilde{\Delta}_j \phi\big)\right\|_{L^p} \\
  &\lesssim& 2^{q(1+s+\sigma)} \sum_{\lambda_1=0}^\lambda \bigg(\sum_{j \geq \max \left\{3, q-N_\lambda\right\}}
  \big\|(T_{\mathcal{W} \cdot \nabla})^{\lambda_1}\big(\Delta_j u\,
  \widetilde{\Delta}_j \phi\big)\big\|_{L^p}\bigg) \\
  &\lesssim& 2^{q(1+s+\sigma)} \sum_{j \geq \max \left\{3, q-N_\lambda\right\}}
  \sum_{\lambda_2+\lambda_3 \leq \lambda} \big\|(T_{\mathcal{W} \cdot \nabla})^{\lambda_2}
  \Delta_j u\big\|_{L^\infty}\big\|(T_{\mathcal{W} \cdot \nabla})^{\lambda_3}
  \widetilde{\Delta}_j \phi\big\|_{L^p} \\
  &\lesssim& \sum_{\lambda_2+\lambda_3 \leq \ell+1} \sum_{j \geq \max \left\{3, q-N_\lambda\right\}}
  2^{(q-j)(1+s+\sigma)}  2^j\big\|(T_{\mathcal{W}\cdot\nabla})^{\lambda_2}
  \Delta_j u\big\|_{L^\infty}   \\
  &&\ \ \  \times \ 2^{j(\sigma+s)} \big\|(T_{\mathcal{W}\cdot \nabla}
  )^{\lambda_3} \widetilde{\Delta}_j \phi\big\|_{L^p} \\
  &\lesssim& \sum_{\lambda_2=0}^{\ell+1} \sum_{j \geq \max \left\{3, q-N_\lambda\right\}}
  2^{(q-j)(1+s+\sigma)} \bigg(\sum_{j_1 \sim j} \sum_{\lambda_4=0}^{\lambda_2}
  \big\|(T_{\mathcal{W}\cdot \nabla})^{\lambda_4} \Delta_{j_1} \nabla u\big\|_{L^\infty}
  \bigg)\\
  && \ \ \times \ \|\phi\|_{\widetilde{\mathcal{B}}_{p, \infty, \mathcal{W}}^{\sigma+s, \ell+1}} \\
  &\lesssim& c_q\|\nabla u\|_{\widetilde{\mathcal{B}}_{\mathcal{W}}^{0, \ell+1}}
  \|\phi\|_{\widetilde{\mathcal{B}}_{p, r, \mathcal{W}}^{\sigma+s, \ell+1}} \\
  &\lesssim& c_q\|\nabla u\|_{\mathcal{B}_{\mathcal{W}}^{0, \ell+1}}
  \|\phi\|_{\mathcal{B}_{p, r, \mathcal{W}}^{\sigma+s, \ell+1}},
\end{eqnarray*}
which guarantees that
\begin{equation*}
  \left\|{I}_3\right\|_{\widetilde{\mathcal{B}}_{p, r, \mathcal{W}}^{\sigma+s, \ell+1}}
  \leq C\|\nabla u\|_{\mathcal{B}_{\mathcal{W}}^{0, \ell+1}}
  \|\phi\|_{\mathcal{B}_{p, r, \mathcal{W}}^{\sigma+s, \ell+1}} .
\end{equation*}
For ${I}_4$, similarly as above, we infer that for each $\lambda \in\{0,1, \ldots, \ell+1\}$,
\begin{eqnarray*}
  2^{qs}\big\|\Delta_q\left(T_{\mathcal{W} \cdot \nabla}\right)^\lambda {I}_4\big\|_{L^p} &\lesssim& 2^{qs} \sum_{j \geq \max \left\{3, q-N_\lambda\right\}}\left\|\Delta_q
  \left(T_{\mathcal{W} \cdot \nabla}\right)^\lambda \operatorname{div}
  \left(\Delta_j u \widetilde{\Delta}_j m(D) \phi\right)\right\|_{L^p} \\
  &\lesssim& 2^{q(1+s)} \sum_{\lambda_1=0}^\lambda \sum_{j \geq \max \left\{3, q-N_\lambda\right\}}
  \big\|(T_{\mathcal{W} \cdot \nabla})^{\lambda_1}
  \big(\Delta_j u \, \widetilde{\Delta}_j m(D) \phi\big)\big\|_{L^p} \\
  &\lesssim& 2^{q(1+s)}\sum_{j \geq \max \left\{3, q-N_\lambda\right\}}
  \sum_{\lambda_2+\lambda_3 \leq \lambda}\big\|(T_{\mathcal{W} \cdot \nabla})^{\lambda_2}
  \Delta_j u \big\|_{L^\infty} \big\|(T_{\mathcal{W} \cdot \nabla})^{\lambda_3}
  \widetilde{\Delta}_j m(D) \phi\big\|_{L^p} \\
  &\lesssim& \sum_{\lambda_2+\lambda_3\leq\ell+1} \sum_{j \geq \max \left\{3, q-N_\lambda\right\}}
  2^{(q-j)(1+s)} 2^j \big\|(T_{\mathcal{W} \cdot \nabla})^{\lambda_2}
  \Delta_j u\big\|_{L^\infty} \\
  && \ \times \ \bigg(2^{j(s+\sigma)} \sum_{\lambda_4=0}^{\lambda_3}
  \big\|(T_{\mathcal{W} \cdot \nabla})^{\lambda_4} \widetilde{\Delta}_j \phi\big\|_{L^p}\bigg) \\
  &\lesssim& \sum_{\lambda_2=0}^{\ell+1} \sum_{j \geq \max \left\{3, q-N_\lambda\right\}}
  2^{(q-j)(1+s)}\bigg(\sum_{j_1 \sim j} \sum_{\lambda_5=0}^{\lambda_2}
  \big\|(T_{\mathcal{W} \cdot \nabla})^{\lambda_5} \Delta_{j_1} \nabla u\big\|_{L^\infty}\bigg)
  \left\|\phi\right\|_{\widetilde{\mathcal{B}}_{p, \infty, \mathcal{W}}^{s+\sigma, \ell+1}} \\
  &\lesssim& c_q\|\nabla u\|_{\widetilde{\mathcal{B}}_{\mathcal{W}}^{0, \ell+1}}
  \left\|\phi\right\|_{\widetilde{\mathcal{B}}_{p, r, \mathcal{W}}^{s+\sigma, \ell+1}} \\
 &\lesssim& c_q\|\nabla u\|_{\mathcal{B}_{\mathcal{W}}^{0, \ell+1}}
  \left\|\phi\right\|_{\mathcal{B}_{p, r, \mathcal{W}}^{\sigma+s, \ell+1}},
\end{eqnarray*}
Therefore, we get
\begin{equation*}
  \left\|{I}_4\right\|_{\widetilde{\mathcal{B}}_{p, r, \mathcal{W}}^{s, \ell+1}}
  \leq C\|\nabla u\|_{\mathcal{B}_{\mathcal{W}}^{0, \ell+1}}
  \|\phi\|_{\mathcal{B}_{p, r, \mathcal{W}}^{\sigma+s, \ell+1}}.
\end{equation*}
Taking advantages of Lemma~\ref{lem:m(D)},
the term ${I}_5$ can be easily estimated as follows:
\begin{equation*}
\begin{aligned}
  \left\|{I}_5 \right\|_{\widetilde{\mathcal{B}}_{p, r, \mathcal{W}}^{s, \ell+1}}
  & \leq \sum_{q=-1}^{N_{\ell}} \sum_{\lambda=0}^{\ell+1} \sum_{j=-1}^2
  \left(\big\|\Delta_q(T_{\mathcal{W} \cdot \nabla})^\lambda m(D)
  \operatorname{div}\big(\Delta_j u \,\widetilde{\Delta}_j \phi\big)\big\|_{L^p}\right. \\
  & \left.\hspace{8em} + \big\|\Delta_q (T_{\mathcal{W} \cdot \nabla})^\lambda
  \big(\Delta_j u \cdot \nabla m(D) \widetilde{\Delta}_j \phi\big)\big\|_{L^p}\right) \\
  & \leq C \sum_{j=-1}^2\left(\big\|m(D) \operatorname{div}\big(\Delta_j u\, \widetilde{\Delta}_j \phi\big)
  \big\|_{L^p} + \big\|\Delta_j u \cdot \nabla m(D) \widetilde{\Delta}_j \phi\big\|_{L^p} \right) \\
  & \leq C\|u\|_{L^\infty} \bigg(\sum_{-1 \leq j \leq 2}\big\|\widetilde{\Delta}_j \phi\big\|_{L^p}\bigg)
  \leq C\|u\|_{L^{\infty}}\|\phi\|_{\mathcal{B}_{p, r, W}^{\sigma+s, \ell+1}} .
\end{aligned}
\end{equation*}
Hence, gathering the above estimates,  
we find the wanted inequality \eqref{eq:CMX2.15}. \\

Then, to prove the control  \eqref{eq:str-es-mDphi}, thanks to  \eqref{eq:stBes-fact} and \eqref{eq:mDphi-Bes}, we have that
\begin{eqnarray*}
  \|m(D) \phi\|_{\mathcal{B}_{p, r, \mathcal{W}}^{s,\ell+2}}
   &= & \|\partial_{\mathcal{W}}(m(D) \phi)\|_{\mathcal{B}_{p, r, \mathcal{W}}^{s, \ell+1}}
  + \|m(D) \phi\|_{B_{p, r}^{s}} \\
  &\leq& C\|\mathcal{W} \cdot \nabla(m(D) \phi)\|_{\mathcal{B}_{p, r, \mathcal{W}}^{s,\ell+1}}
  +  C\|\phi\|_{B_{p, r}^{s+\sigma}} + C \mathbf{1}_{\{-1<\sigma\leq 0\}}\left\|\Delta_{-1} m(D) \phi\right\|_{L^p}\\
 &\leq& \|[m(D), \mathcal{W}\cdot\nabla]\phi\|_{\mathcal{B}_{p, r, \mathcal{W}}^{s,\ell+1}}
  +\|m(D)\partial_{\mathcal{W}}\phi\|_{\mathcal{B}_{p, r, \mathcal{W}}^{s,\ell+1}} \\
  && \ + \ C\|\phi\|_{B_{p, r}^{s+\sigma}} + C \mathbf{1}_{\{-1<\sigma\leq 0\}}\left\|\Delta_{-1} m(D) \phi\right\|_{L^p}.
\end{eqnarray*}
Thanks to the induction assumptions of \eqref{eq:CMX2.15}-\eqref{eq:str-es-mDphi} and using \eqref{eq:Del-1mDpaW},
one can get
\begin{align*}
  \|[m(D), \mathcal{W}\cdot\nabla]\phi\|_{\mathcal{B}_{p, r, \mathcal{W}}^{s,\ell+1}}
  &\leq C \|\phi\|_{\mathcal{B}_{p, r, \mathcal{W}}^{s+\sigma,\ell+1}}
  \left(\|\nabla\mathcal{W}\|_{\mathcal{B}^{0,\ell+1}_{\mathcal{W}}} +\|\mathcal{W}\|_{L^\infty}\right)\\
  &\leq C \|\phi\|_{\mathcal{B}_{p, r, \mathcal{W}}^{s+\sigma,\ell+1}}
  \|\mathcal{W}\|_{\mathcal{B}^{1,\ell+1}_{\mathcal{W}}},
\end{align*}
and
\begin{align*}
  \|m(D)\partial_{\mathcal{W}}\phi\|_{\mathcal{B}_{p, r, \mathcal{W}}^{s,\ell+1}}
  & \leq C \|\partial_{\mathcal{W}}\phi\|_{\mathcal{B}_{p, r, \mathcal{W}}^{s+\sigma,\ell+1}}
  + C \|\mathcal{W}\|_{\mathcal{B}^{1,\ell}_{\mathcal{W}}}
  \Big(\|\partial_{\mathcal{W}}\phi\|_{\mathcal{B}_{p, r, \mathcal{W}}^{s+\sigma,\ell}}
  +  \left\|\Delta_{-1} m(D) \mathrm{div}\,(\mathcal{W}\, \phi)\right\|_{L^p} \Big) \\
  & \leq C \|\phi\|_{\mathcal{B}_{p, r, \mathcal{W}}^{s+\sigma,\ell+2}}
  + C\|\mathcal{W}\|_{\mathcal{B}^{1,\ell}_{\mathcal{W}}}
  \Big(\|\phi\|_{\mathcal{B}_{p, r, \mathcal{W}}^{s+\sigma,\ell+1}}
  +  \left\| \phi\right\|_{B^{s+\sigma}_{p,r}}\Big) \\
  & \leq C \|\phi\|_{\mathcal{B}_{p, r, \mathcal{W}}^{s+\sigma,\ell+2}},
\end{align*}
where $C>0$ depends on  $\|\mathcal{W}\|_{\mathcal{C}^{1+\rho,k-1}_{\mathcal{W}}}$.
Collecting the above estimates allows us to conclude that \eqref{eq:str-es-mDphi} holds in the step $\ell+1$, this ends the proof.
\end{proof}

\noindent {\bf{Acknowledgements:}} L. Xue has been supported by National Key Research and Development Program of China (No. 2020YFA0712900)
and by National Natural Science Foundation of China (No. 12271045).

\end{document}